\documentclass[11pt, twoside, psamsfonts]{amsart}

\newif\ifPDF
\ifx\pdfoutput\undefined\PDFfalse
\else \ifnum \pdfoutput > 0 \PDFtrue
        \else \PDFfalse
        \fi
\fi

\usepackage[centertags]{amsmath}
\usepackage{amsfonts}
\usepackage{mathrsfs}
\usepackage{textcomp}
\usepackage{amssymb}
\usepackage{amsthm}
\usepackage{newlfont}
\usepackage[all]{xy}

\ifPDF
  \usepackage[pdftex]{xcolor, graphicx}
  \usepackage[colorlinks=true,linkcolor=blue, citecolor=blue]{hyperref}%

\else
  \usepackage{color}
  \usepackage[dvips]{graphicx}
  \usepackage[dvips]{hyperref}
\fi

\usepackage{tkz-graph}
\tikzset{EdgeStyle/.style = {->}}
\tikzset{LabelStyle/.style= {fill=yellow}}
\usetikzlibrary{shapes,snakes,calendar,matrix,backgrounds,folding}

\usepackage{bbm}

\usepackage[top=1in, bottom=1.25in, left=1.25in, right=1.25in]
{geometry}

\newtheorem{theorem}{Theorem}[section]
\newtheorem{corollary}[theorem]{Corollary}
\newtheorem{lemma}[theorem]{Lemma}
\newtheorem{proposition}[theorem]{Proposition}
\theoremstyle{definition}
\newtheorem{definition}[theorem]{Definition}
\newtheorem{remark}[theorem]{Remark}
\newtheorem{example}[theorem]{Example}
\numberwithin{equation}{section}

\theoremstyle{definition}

\begin{document}
\newcommand{\norm}[1]{\left\Vert#1\right\Vert}
\newcommand{\abs}[1]{\left\vert#1\right\vert}
\newcommand{\To}{\longrightarrow}
\newcommand{\F}{\mathcal{F}}

\newcommand{\tcb}{\textcolor{blue}}
\newcommand{\tcr}{\textcolor{red}}
\newcommand{\ol}{\overline}
\newcommand{\ov}{\overline}
\newcommand{\wt}{\widetilde}
\newcommand{\N}{\mathbb N}
\newcommand{\Z}{\mathbb Z}
\newcommand{\E}{\mathcal E}
\newcommand{\B}{\mathcal B}
\newcommand{\om}{\omega}
\newcommand{\ent}{f^{(n)}_{vw}}

\newcommand{\be}{\begin{equation}}
\newcommand{\ee}{\end{equation}}
\newcommand{\ba}{\begin{aligned}}
\newcommand{\ea}{\end{aligned}}
\newcommand{\wh}{\widehat}
\newcommand{\mc}{\mathcal}

\newcommand{\vp}{\varphi}
\newcommand{\e}{\varepsilon}


\newcommand{\La}{\Lambda}
\newcommand{\Om}{\Omega}
\newcommand{\al}{\alpha}
\newcommand{\G}{\Gamma}
\newcommand{\g}{\gamma}
\newcommand{\T}{\theta}
\newcommand{\De}{\Delta}
\newcommand{\de}{\delta}
\newcommand{\s}{\sigma}
\newcommand{\A}{{\cal A}}
\newcommand{\M}{{\cal M}}
\newcommand{\bs}{(X,{\cal B})}
\newcommand{\Aut}{Aut(X,{\cal B})}
\newcommand{\h}{Homeo(\Om)}
\newcommand{\pos}{{\mathbb R}^*_+}
\newcommand{\R}{{\mathbb R}}
\newcommand{\ap}{{\cal A}p}
\newcommand{\per}{{\cal P}er}
\newcommand{\inc}{{\cal I}nc}


\title[Harmonic analysis on Bratteli diagrams]
{Harmonic analysis on graphs via Bratteli diagrams and path-space measures}

\author{Sergey Bezuglyi and Palle E.T. Jorgensen}
\address{Department of Mathematics, University of Iowa, Iowa City,
 Iowa, USA}
\email{sergii-bezuglyi@uiowa.edu}
\email{palle-jorgensen@uiowa.edu}

\thanks{}
\keywords{Graphs, Bratteli diagrams, electrical networks, graph Laplacian, 
unbounded operators, Hilbert space, spectral theory, 
harmonic analysis, dynamical systems, invariant measures, standard 
measure space, symmetric and path-space measure, Markov operator and 
process, finite energy space}

\subjclass[2010]{05C60, 37B10, 41A63, 42B37, 6N30, 47L50, 60J45} 
\date{\today}

\dedicatory{To the memory of \textit{Ron Graham} (1935-2020) a 
pioneer and leader in discrete mathematics.}\commby{}


\begin{abstract}

The past decade has seen a flourishing of advances in harmonic analysis 
of graphs. They lie at the crossroads of graph theory and such analytical 
tools as graph Laplacians, Markov processes and associated boundaries, 
analysis of path-space, harmonic analysis, dynamics, and tail-invariant 
measures. Motivated by recent advances for the special case of Bratteli 
diagrams, our present focus will be on those graph systems $G$ with the 
property that the sets of vertices $V$ and edges $E$ admit discrete level 
structures. A choice of discrete levels in turn leads to new and intriguing 
discrete-time random-walk models.  

   Our main extension (which greatly expands the earlier analysis of Bratteli 
diagrams) is the case when the levels in the graph system $G$ under 
consideration are now allowed to be standard  measure spaces. Hence, in
 the 
measure framework, we must deal with systems of \textit{transition probabilities}, 
as opposed to \textit{incidence matrices} (for the traditional Bratteli diagrams).

The paper is divided into two parts, (i) the special case when the levels are 
countable discrete systems, and (ii) the (non-atomic) measurable category,
i.e., when 
each level is a prescribed measure space with standard Borel structure. 
The study of the two cases together is motivated in part by recent new 
results on \textit{graph-limits.}
Our results depend  on a new 
analysis of certain duality systems for operators in Hilbert space; 
specifically, one dual system of operator for each level. We prove new 
results in both cases, (i) and (ii); and we further stress both similarities, 
and differences, between results and techniques involved in the two cases.

\end{abstract}

\maketitle

\tableofcontents

\setcounter{tocdepth}{1}

\section{Introduction}\label{sect Introduction}

\subsection{Motivation} A key tool in various
approaches to harmonic analysis on \textit{infinite graph networks} 
(sets of \textit{vertices} $V$ and \textit{edges} $E$) is notions of
``boundary''. While in classical 
harmonic analysis, the Poison boundary is a favorite, in carrying over key 
harmonic analysis ideas to infinite graph networks, the possibilities are 
much wider. But the analyses introduced here are all based on the notion 
of \textit{infinite paths}. The study of Markov transitions on graphs, and 
associated boundaries, are a case in point. Of special interest for 
harmonic analysis, and dynamics, on infinite graphs, is therefore the 
study of path-spaces, and path-space measures.

This analysis takes an especially nice form in the special case of 
\textit{infinite graphs} which admits representations as \textit{Bratteli 
diagrams}
 (including in their generalized forms, see Sections \ref{sect basics} -
 \ref{sec Laplace} and \ref{sect meas BD}). For details regarding Bratteli
  diagrams, we refer to the 
 literature cited below in this section, see \ref{ssect Literature}.
The main approach which we elaborate in  this paper is based on 
an application  of various ideas and methods developed  in the theory
of Bratteli diagrams to the case of measurable Bratteli diagrams, see 
Definitions \ref{def GBD} and \ref{def meas BD} from Sections 
\ref{sect basics} and \ref{sect meas BD}. 
We use this term for the case where every level of a diagram is
 represented by a $\sigma$-finite measure space, and a sequence  
\textit{transition kernels} plays the role of \textit{incidence matrices}. 

Motivated by numerous applications, we shall offer a systematic dynamical
 system  approach to both these extensions with countable and measure
 space levels.

It is worth noting that classical Bratteli diagrams (with finite levels)
have been used in the solution of diverse classification 
problems (operator algebras, representation theory, fast Fourier 
transform algorithms, and  dynamical systems on Cantor, Borel, and
measure spaces), in modeling 
\textit{electrical networks}, in \textit{neural networks}, in 
\textit{harmonic analysis,} and in an 
analysis of \textit{Cantor dynamics}. We collected the corresponding references at
the end of the introduction, see Subsection \ref{ssect Literature}.

A formal definition of a generalized Bratteli diagram is given in 
Section \ref{sect basics}, Definition \ref{def GBD}. 
Here we shall adopt the view of Bratteli 
diagrams as a special class 
of graphs $G$ with specified sets of vertices $V$ and edges $E$. What 
sets Bratteli diagrams apart from other graph systems $G = (V, E)$ is the 
introduction of \textit{levels}, such that edges are linking pairs of vertices only 
from neighboring levels of vertices. Hence, for each level, there will be a 
specification of a transition (incidence) matrix or  a transition kernel,
the term that is appropriate for measurable Bratteli diagrams.   The case 
of discrete levels is subdivided into two principally different classes: (i)
every level consists of a finite set of vertices, (ii) every level is an infinitely
countable set. In case (i), the path-space is a Cantor set, and such Bratteli 
diagrams represent models for homeomorphisms of Cantor sets. In case
(ii), the path-space is a zero-dimensional Polish space, the corresponding 
Bratteli diagrams are models for Borel automorphisms of a standard 
Borel space. 

Transition (incidence) matrices of a Bratteli diagram may then constitute a
 discrete-time \textit{Markov process}. In the traditional setting for Bratteli diagrams, this will 
form finite-state Markov process. But clearly, for many applications, it is 
natural to extend the setting in two ways: first, to consider instead a 
countably infinite set of vertices at each level; and secondly, to allow an 
even wider setting where each level is now a standard measure space. 
The corresponding discrete-time Markov process will then consist of a 
system of transition probability measures, indicating the probability 
distribution for transitions from one level to the next.

In more detail, a Bratteli diagram 
has a prescribed system of levels, see e.g., Figure 
\ref{fig Br D}, where the levels are sets 
of vertices, indexed by the natural numbers $\N_0$ including 0. For $n
\in \N$, the level $n$ vertex set $V_n$ then only admits edges 
backwards to $V_{n-1}$ and forwards to  $V_{n+1}$. 
For applications to Markov 
processes it is convenient to think of $n$ as discrete time, and 
``forward'' meaning increasing $n$ to $n+1$ for each $n$. Given a 
Bratteli diagram, a path is then an infinite string of  edges, $(e_n)$ such 
that, for every $n$, the range-vertex of $e_n$ in $V_{n+1}$ matches
the source of $e_{n+1}$. Our infinite paths do not admit loops. The set of 
all infinite paths is called the \textit{path-space}. A key tool for our
 harmonic analysis is the study of particular \textit{path-space 
 measures}. These 
notions are introduced in Section \ref{sect basics}.

A choice of path-space measure, then yields a \textit{random process}. 
The case of measures which yield discrete time \textit{Markov 
processes} includes the \textit{tail-invariant measures}. 
A tail invariant measure has the following “forgetful” property: For each level $n$, its value on  paths (cylinder sets) 
 that meet at a fixed vertex $v \in V_n$  is independent of the variety of
finite path segments from level $0$ up to $v$  in  $V_n$. (See Definition
\ref{def tail inv m}). An analysis of ordered Bratteli diagrams leads to 
one of our key tools, \textit{Kakutani-Rokhlin towers}, and 
\textit{Kakutani-Rokhlin partitions}, see Subsection \ref{ssect K-R towers}.

\subsection{Main results} 
We describe our main results proved in this paper. 
After defining the main concepts of discrete Bratteli diagrams $B = (V, E)$
in Section \ref{sect basics}, we focus on tail invariant measures. Clearly, 
every tail invariant measure $\mu$ is uniquely determined by its values on 
cylinder sets of the path-space $X_B$. It gives a sequence of non-negative
vectors $\mu^{(n)} = (\mu^{(n)}_v : v \in V_n)$ where $\mu^{(n)}_v$
is the measure of a cylinder set ending at $v$. By tail invariance, 
$\mu^{(n)}_v$ does not depend on a cylinder set. 
In  Theorem \ref{thm inv measures} we find a condition  
under which  a sequence of non-negative vectors determines a 
tail invariant measure. Theorem  \ref{thm inv meas stat BD} represents 
explicitly a tail invariant measure for a stationary Bratteli diagram. 

In Section \ref{sect Markov}, we consider Markov measures on the 
path-space of a generalized Bratteli diagram defined by a sequence of  
probability transition matrices $(P_n)$. Theorem  
 \ref{thm exist of Q_n} states the existence of co-transition probability 
 matrices $(Q_n)$. The properties of matrices $Q_n$ and $P_n$ 
 are discussed in 
 Propositions \ref{prop about wh Q_n}, \ref{prop T_P T_Q}.
 It is shown in Theorem \ref{thm inv meas is Markov}
 that every tail invariant  measure is a Markov measure.   

Section \ref{sec Laplace} contains some elements of harmonic analysis 
on generalized Bratteli diagrams. Starting with an initial distribution 
$q^{(0)}$ on the level $V_0$ and a sequence of transition kernels 
$(P_n)$, we define a  reversible Markov process, harmonic functions, 
\textit{Laplace operator}, and 
\textit{finite energy Hilbert space}. Our main results are obtained in Theorem \ref{thm 
harmonic}, where harmonic functions are described, and Proposition
 \ref{prop finite energy}. The latter contains 
a criterion for a function $f: V \to \R$ to be in the finite energy Hilbert
 space.

The goal of the second part of the paper (Sections \ref{sect Trans kernels}
 - \ref{sect meas BD}) is to build a theory of measurable Bratteli diagrams 
which is parallel to that of the purely discrete case. In Section  
\ref{sect Trans kernels} we consider the notion of a dual pair $(P,Q)$ of 
\textit{transition kernels} associated to $\sigma$-finite measure spaces 
$(X_1, \nu_1)$ and $(X_2, \nu_2)$. This 
means that the objects satisfy the relation 
$d\nu_1(x) P(x, dy) = d\nu_2(y) Q(y, dx)$. A number of 
results are  proved in this section which clarify the interplay between the
notions of transition kernels $P, Q$, linear (unbounded) operators
$T_P, T_Q$ generated by these kernels, and 
measures $\nu_1, \nu_2$. We refer to Theorems \ref{thm  on F(rho)} and 
\ref{thm_P determines Q} as the principal statements of Section 
\ref{sect Trans kernels}. 

In Section \ref{ssect Kernels on L^2} we develop the approach used in the 
previous section to the case when the operators  $T_P$ and  $T_Q$ are 
considered acting between the corresponding $L^2(\nu_i)$-spaces. 
 Theorem \ref{thm T_P contractive} proves that $T_P$ and $T_Q$ are
  contractive  operators
whenever $P$ and $Q$ are probability kernels. In the most general case,
$T_P$ and $T_Q$ are closable densely defined  operators such that 
$T_P \subset (T_Q)^*$ and $T_Q \subset (T_P)^*$. We give also a 
condition when the operators $T_P$ and $T_Q$  are bounded, 
Proposition \ref{prop bounded T_P T_Q}. 

We continue our analysis  of transition kernels $P, Q$ and the 
corresponding operators in Section \ref{ssect L-set}. If $R$ is a positive 
definite operator which generates a symmetric measure $\lambda$ on 
the product $X_1 \times X_2$, then a 
\textit{reproducing kernel Hilbert space}
$\mc H(\lambda)$ can be defined by the \textit{positive definite function} 
$(A, B) \mapsto \lambda(A \times B)$ where $A$ and $B$ are sets of finite 
measure. In Theorem \ref{thm RKHS}, we give
an explicit description of functions from $\mc H(\lambda)$. Theorem 
\ref{thm factorization} states that $R$ can be factorized in the product of 
two  operators.  

In the final Section \ref{sect meas BD} we apply the method 
developed in Sections \ref{sect Trans kernels} and  
\ref{ssect Kernels on L^2} to the case of measurable Bratteli diagrams. 
Recall that this name is used for a sequence of $\sigma$-finite measure 
spaces $(X_i, \nu_i)$ together with probability transition kernels 
$P_i, Q_i$. This approach is an extension of generalized Bratteli diagrams 
(discrete case) to the case of standard measure spaces. Our definitions 
and results proved in   Section \ref{sect meas BD} have discrete 
analogues considered in Section \ref{sec Laplace}. In particular, we 
consider a graph Laplacian defined on a measurable Bratteli diagram.
The most important results of this section are given  in Theorem 
\ref{prop Markov kernels} and Proposition \ref{prop meas energy}.

\subsection{Short outline of the paper}
The notions of path-space, and the corresponding measures, make 
sense for all three types of feasible Bratteli diagrams: (i) the standard 
case when each $V_n$ is assumed finite; (ii) the case when each $V_n$ 
is countably infinite, Section \ref{sect basics};  and (iii), measurable 
Bratteli diagrams, (the measurable category) when each $V_n$ is an 
uncountable standard Borel space, see Section \ref{sect meas BD}. 
However, as we show, 
the construction, and the relevant properties, of the path-space measures 
is quite different for the three cases.

Our analysis in the first three sections of the paper leads up to a study of 
graph Laplacians (Section \ref{sec Laplace}) and the corresponding 
harmonic functions. 
The path-space measures which have proved most successful for our 
present harmonic analysis are the tail-invariant measures. For the 
discrete Bratteli diagrams, we characterize these tail-invariant measures 
(Theorem \ref{thm inv measures}). In subsequent sections, we further 
demonstrate the use of tail-invariant measures in an harmonic analysis 
and discrete-time dynamics, for graphs with level structure (generalized 
Bratteli diagrams).

In Section \ref{sect basics}, we present an algorithm (Theorem 
\ref{thm inv meas stat BD}) for constructing tail-invariant measures on 
stationary Bratteli diagrams. It is based on a generalized 
Perron-Frobenius spectral analysis (Theorem \ref{thm general PF}).

In Sections \ref{sect Markov}, \ref{sect Trans kernels}, and 
\ref{ssect Kernels on L^2}, we introduce the more general class of
 Markov 
measures, and their corresponding transition kernels, referring to 
transition between levels. These results are based on a new analysis of 
associated systems of operators, transition operators. An explicit spectral 
theory for these transition operators is presented in Section 
\ref{ssect Kernels on L^2}. In Section \ref{ssect L-set}, we give a new
 tool for this analysis. It is a particular family of positive 
definite kernels, and their corresponding reproducing kernel 
Hilbert spaces (RKHSs).  Section \ref{sect meas BD} deals with the
 case of measurable 
Bratteli diagrams, and we present results which are based on a new 
analysis of path-space measures, and corresponding Markov processes.

\subsection{Literature} \label{ssect Literature}
Our paper makes direct connections to various adjacent areas such as:
stochastic analysis on graphs, graph-limits, dynamical systems, 
applied and computational harmonic analysis, weighted networks, spectral 
theory, Markov processes, etc.  For the reader's convenience we present the 
relevant references divided in several groups. 

(a) \textit{Standard Borel and measure spaces}. The literature where standard Borel spaces play a crucial role is very 
extensive. This subject is studied in ergodic theory, Borel dynamics, 
descriptive set theory, operator algebras, and many other fields. We refer 
to \cite{Kechris1995,
DoughertyJacksonKechris_1994, JacksonKechrisLouveau_2002, 
CiolettiSilvaStadlbauer2019, CornfeldFominSinai1982, Nadkarni1995,
BezuglyiDooleyKwiatkowski2006, KerrLi_2016} where the reader can find
 more information. 

(b) \textit{Cantor dynamics and $C^*$-algebras.} The problems of
 classification of dimension groups, $C^*$-algebras, 
dynamical systems in Cantor dynamics are naturally
connected with the study of invariants of the corresponding 
Bratteli diagrams. The list of relevant references is extremely long. We 
refer only to
\cite{Bratteli1972, LazarTaylor1980, GiordanoPutnamSkau1995, 
BratteliJorgensenKimRoush2001, BratteliJorgensenKimRoush2002,
Takesaki2003, Phillips2005, Jorgensen2006, Zerr2006, Putnam2018}. 
The reader can find numerous
 applications of Bratteli  diagrams in the theory of Cantor and Borel 
 dynamical systems, see, e.g.   
 \cite{HermanPutnamSkau1992,
 BezuglyiJorgensen2015, Durand2010, BezuglyiDooleyKwiatkowski2006,  BezuglyiKarpel2016}. 

(c) \textit{Random processes.} Random walks on Bratteli diagrams are discussed in 
\cite{FrickPetersen_2010, CarollPetersen_2016} and
\cite{Renault2018}.  A different approach is used in the works
\cite{Anandam_2012, Anandam2012, Anandam_2011}.


(d) \textit{Perron-Frobenius and Markov chains.} The Perron-Frobenius 
theory has various applications in Bratteli  diagrams,
networks, Markov chains, and other areas such as network models and
graph Laplacians: \cite{JorgensenPearse2011, Jorgensen_Pearse2013, 
JorgensenPearse_2019, Venkitaraman-2019, YankelevskyElad2019,
Li-2016, JorgensenPearse-2020}. 
We refer here also  to several
recent papers related to our study: \cite{CerfDalmau2019, 
ChaystiNoutsos2019, GautierTudiscoHein2019, GiladiRuffer2019, 
ChenVongLiXu2019, FalkNussbaum2018}.

(e) \textit{Transition probability kernels}. The notion of transition 
probability kernels is an important tool
for the study of Markov chains and properties of random walks. The 
reader can find more details, for example, in the books \cite{Douc_2018, 
Klenke_2014, LyonsPeres_2016, Nummelin_1984, Revuz_1984} 
and recent articles 
\cite{DeyTrivedi2019, Smirnov2019, EpsteinPop2019}.

(f) \textit{Weighted networks, neural networks, graph limits.} 
Engineering applications of deep neural 
networks realized as a 
generalized Bratteli diagram structure are discussed in 
\cite{YangDing2020,
Sun-2020, Chi-2020}; graph-limits were considered  in \cite{Dunlop2020,
Lovasz2012, KunszentiLovasz2019}. Important contribution to the graph  theory and
weighted networks have been made in \cite{Chung2007, Chung2010,
Chung2014, ChungKenter2014, ChungGraham2012}. 

(g) \textit{Generalized Bratteli diagram, machine learning.} 
The literature on optimization,  financial models, machine learning with 
deep neural 
networks and generalized Bratteli diagram structure includes 
\cite{JorgensenPearse_2019, GaoSu2020, Adams2020, HaninNica2020,
 Xiao2020, Xia2020}. More information can be found in  
  \cite{Baudot2019, SumLeung2019, 
KovachkiStuart2019}

\section{Graph analysis via Bratteli diagrams}\label{sect basics} 

In this section, we briefly discuss the main definitions and facts about 
Bratteli diagrams. 
Our emphasis is on those features of Bratteli diagrams, and their
generalizations, which will be important for our present analysis of general
 graphs which admit a system of levels as outlined above.
 
The literature on Bratteli diagrams and their 
application in dynamics is very extensive, we mention here  
\cite{HermanPutnamSkau1992, GiordanoPutnamSkau1995,
Durand2010,  BezuglyiKarpel2016, BezuglyiKarpel2020} 
(more references can be found therein). In contrast to most 
of the above sources,  we focus here on
the case of \textit{generalized  Bratteli diagrams}. This means that
the set of vertices in every level is infinitely countable. 

\subsection{Generalized Bratteli diagrams} 

In the introduction, we described the notion of a Bratteli diagram 
considering this as an infinite graded graph. Here is a natural extension
of this concept to the case of countable levels.

\begin{definition}\label{def GBD}
Let $V_0$ be a countable set (which may be identified with either 
$\N$ or $\Z$ for convenience).
Set $V_i = V_0$ for all $i \geq 1$. A countable graded graph   
$B = (V, E)$ is called a \textit{generalized Bratteli diagram} if it 
satisfies the following properties.

(i) The set of vertices $V$ of $B$ is $\bigsqcup_{i=0}^\infty  V_i$. 

(ii) The set of edges $E$ of $B$ is represented as $\bigsqcup_{i=0}^\infty  E_i$ where $E_i$ is the set of edges between the levels $V_i$ and 
$V_{i+1}$. 

(iii) For every $w \in V_i, v \in V_{i+1}$, the set of edges $E(w, v)$  between $w$ and $v$ is finite (or empty); set $|E(w, v)| =  
f^{(i)}_{vw}$. It defines a sequence of infinite (countable-by-countable) 
\textit{incidence  matrices} $(F_n ; n \in \N_0)$ whose entries are 
non-negative  integers: 
$$
F_i = (f^{(i)}_{vw} : v \in V_{i+1}, w\in V_i),\ \   f^{(i)}_{vw} 
 \in \N_0.
$$ 

(iv) The matrices 
$F_i$ have at most \textit{finitely many non-zero entries in each row}. 

(v) The maps  $r,s : E \to V$ are defined on the diagram $B$: 
for very $e \in 
E$ there are $w, v$ such that $e \in E(w, v)$; then $s(e) =w$ and
$r(e) = v$. They are called the \textit{range} ($r$) and \textit{source} 
($s$) maps. 

(vi) For  every $w \in V_i, \; i \geq 0$,
there exists an edge $e \in E_i$ such that $s(e) = w$ and edge $e' \in 
E_{i-1}$ such that $r(e') = w$. In other 
words, every incidence matrix $F_i$ has no zero row and zero column. 
\end{definition}

\begin{remark}
(1)  It follows from Definition \ref{def GBD} that every generalized 
Bratteli diagram is uniquely determined by a sequence of matrices $(F_n)$
such that every matrix satisfies (iii) and (iv). 
For this, one uses the rule that the entry $\ent$ indicates  the
 number of edges between the vertex $w \in V_n$ and vertex 
 $v\in V_{n+1}$. It defines the set $E(w, v)$; then one takes 
$$E_n = \bigcup_{w\in V_n, v \in V_{n+1}} E(w, v)$$

(2) To emphasize that a generalized  Bratteli diagram $B$ is determined
by the sequence of incidence matrices $(F_n)$, we will write $B =
B(F_n)$ if needed. 
An important particular case of a generalized Bratteli diagram is 
obtained when all incidence matrices $F_n$ are the same, $F_n = F$ 
for all $n\in \N_0$. Then, the generalized Bratteli diagram  $B(F)$ is
 called \textit{stationary.}
 \end{remark}

\begin{remark}
If $V_0$ is a singleton, and each $V_n$ is a finite  set, then we obtain   
 the standard definition of a Bratteli diagram originated in 
\cite{Bratteli1972}. Later it was used in the theory
 of $C^*$-algebras  and dynamical systems for solving some 
 classification  problems  and constructions of models (for references,
  see Introduction \ref{ssect Literature}). 

It is important to emphasize that we have no restriction on the entries of
columns of  the incidence matrices $F_n$. They may have any number 
of non-zero (or zero) entries. 
\end{remark}

On Figure 1, we give an example of a Bratteli diagram. 
As a matter of  fact, this example is a small (finite) part of the diagram
 since every Bratteli
diagram has infinitely many levels and every level is a countably 
infinite set. 

\begin{figure}[!htb]\label{fig Br D} 
\centering
  \includegraphics[width=0.75\textwidth, height=0.45\textheight]
  {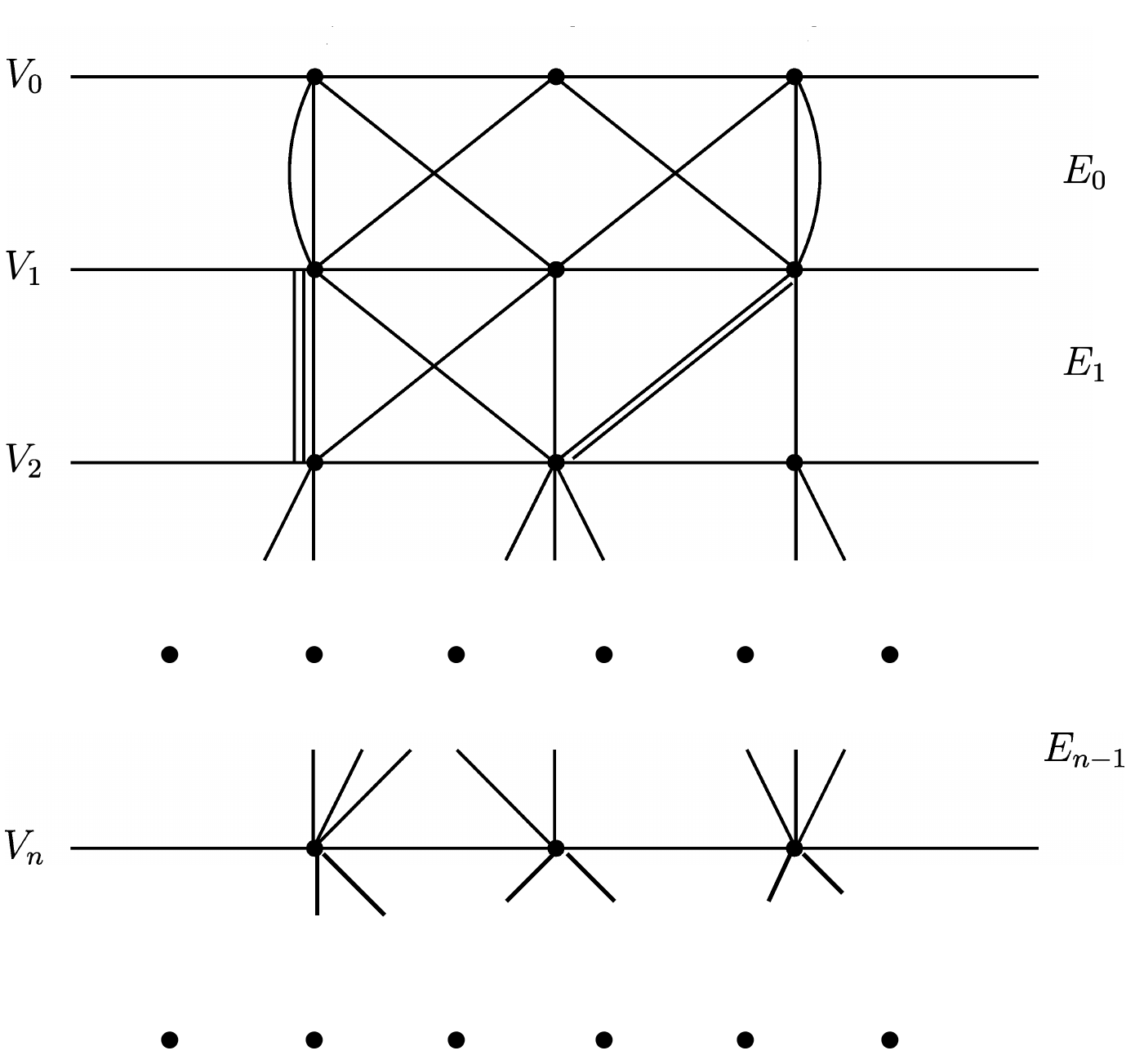}
  \caption{Example of a Bratteli diagram: levels, verices, and edges
  (see Definition \ref{def GBD})}
  \label{fig:kernels}
\end{figure}

\begin{definition}\label{def path space}
A finite or infinite \textit{path} in a Bratteli diagram $B = (V,E)$ is a
 sequence of edges $(e_i : i \geq 0)$ such that $r(e_i) = s(e_{i+1})$. 
Denote by $X_B$ the set of all infinite paths. Every finite path $\ol e =
(e_0, ... , e_n)$ determines a cylinder subset $[\ol e]$ of $X_B$:
$$
[\ol e] := \{x = (x_i) \in X_B : x_0 = e_0, ..., x_n = e_n\}.
$$
The collection of all cylinder subsets forms a base of neighborhoods  for a
topology on $X_B$. In this topology, $X_B$ is a  Polish 
zero-dimensional space and every cylinder set is clopen. 
\end{definition}

\begin{remark} In this remark, we collected a number of simple 
statements about properties of generalized Bratteli diagrams.
 
(1) It follows from Definition \ref{def path space} that  $X_B$ is a
 standard Borel space whose Borel structure  is generated by clopen
(cylinder) sets. In contrast to the classical case, we note that $X_B$ is
not compact and even not locally compact. Furthermore, clopen subsets
are not compact, in general.

(2) If $x = (x_i)$ is a point in $X_B$, then it is obviously represented as
 follows:
$$
\{x\} = \bigcap_{n\geq 0} [\ol e]_n
$$ 
where $[\ol e]_n = [x_0, ... ,x_n]$. But, in general, it is not true that 
any decreasing (nested) sequence of cylinder sets determines a point in 
 $X_B$  since  it can be empty. If the set $\bigcap_{n\geq 0} [\ol e]_n$
 is non-empty, then it contains just a singe point.     
   
(3) The metric on $X_B$, which is compatible with the clopen topology,  
can be defined as follows: for $x = (x_i)$ and $y = (y_i)$ from $X_B$,
$$
\mathrm{dist}(x, y) = \frac{1}{2^N},\ \ \ N = \min\{i \in \N_0 : x_i 
\neq y_i\}.
$$

(4) Considering $X_B$ as a zero-dimensional metric space, we will 
\textit{assume} 
that the diagram $B$ is chosen so that the space $X_B$ 
\textit{has no isolated points}.  This means
that for every infinite  path $(x_0, x_1, x_2, ... ) \in X_B$ and every 
$n \geq 1$, there exists $m > n$  such that $|s^{-1}(r(x_m))| > 1$. Therefore, without loss of 
 generality, we can assume that every column of the incidence matrix 
 $F_n, n \in \N_0,$  has more than one non-zero entry. 

(5) We observe that if, for every vertex $v \in V$, the set $s^{-1}(v)$ is 
finite, then the path-space $X_B$ is a locally compact Polish space. 
The finiteness of $r^{-1}(v)$, which is a requirement included in the 
definition of a generalized Bratteli diagram, is constantly used in
 the next sections.

(6) In the study of Bratteli diagrams the \textit{telescoping} procedure 
is often used. 
This means that, for a given generalized Bratteli diagram $B = (V,E)$, 
one can take any monotone increasing sequence $(n_k : k \in 
\N_0), n_0 = 0,$ 
and construct  a new Bratteli diagram $B' = (V', E')$ where $V' = 
\bigsqcup_{k=0}^\infty V_{n_k}$, and the set of edges $E'$ is 
determined  by the sequence of incidence matrices $F'_k$,
$$
F'_k = F_{n_k} \cdots F_{n_{k+1} -1}.
$$ 
Since every matrix $F_n$ has finitely many non-zero entries in 
each row,  the set 
$E(V_0, v)$ of all finite paths $\ol e$ with  $r(\ol e) = v$ is finite,  
where  $v \in V_n, n \geq 1$. Clearly, 
$$
|E(V_0, v)| = \sum_{w \in V_0} (F_0 \cdots F_{n-1})_{w,v}.
$$

(7) Generalized Bratteli diagram arise in Borel dynamics as 
models for aperiodic Borel automorphisms of standard Borel spaces.
We refer to \cite{BezuglyiDooleyKwiatkowski2006} where such models
have been constructed. 
 \end{remark}
 
Next, we will define the notion of an \textit{irreducible} 
generalized Bratteli diagram. 
For this, it is convenient to identify all sets $V_n, n \geq 0,$, i.e.,
we can think that the same vertex $v$ is a vertex of each level. 

\begin{definition}\label{def irreducible}
It is  said  that a generalized 
Bratteli diagram $B$ is \textit{irreducible} if, for any two 
vertices $v$ and 
$w$,  there exists a level $V_m$ such that $v \in V_0$ and $w\in V_m$ 
are connected by a finite path. This is equivalent to the property that,
for any fixed $v,w$, there exists $m \in \N$ such that the product
of matrices $F_{m-1} \cdots F_0$ has non-zero $(w,v)$-entry.
\end{definition}

Let $B = (V, E)$ be a generalized Bratteli diagram. 
Define the \textit{tail equivalence relation} $\mathcal E$ on the path 
space of $X_B$.

\begin{definition}\label{def tail}
It is said that two paths $x = (x_i)$ and $y = (y_i)$ are \textit{tail
equivalent} if there
exists  $m \in \N$ such that $x_i = y_i$ for all $i \geq m$. Let 
$[x]_{\E} := \{ y \in X_B : (x,y) \in \E\}$ be the set of points tail 
equivalent
to $x$. We say that a point $x$ is \textit{periodic} if $| [x]_{\E} | <
\infty$.  If there is no periodic points, then the tail equivalence relation is 
called  \textit{aperiodic}.
\end{definition}

Without loss of generality, we will assume that $\E$ is aperiodic. Clearly,
$\E$ is a countable  Borel equivalence relation. It can be easily proved 
that $\E$ is \textit{hyperfinite} in the context of Borel dynamical
systems. This notion and results are discussed, for example, in 
\cite{DoughertyJacksonKechris1994}.

\subsection{Kakutani-Rokhlin towers and ordered Bratteli diagrams}
\label{ssect K-R towers} 

\textit{ Kakutani-Rokhlin towers} (partitions) proved to be a very fruitful 
tool in dynamical systems. They have been used to construct the
 approximation of an aperiodic transformation by periodic ones. The idea
 to us a refining sequence of Kakutani-Rokhlin partition leads to the 
 realization of aperiodic transformation as a map acting on the path-space
 of a Bratteli diagram \cite{ConnesKrieger1977, 
 Vershik1981, HermanPutnamSkau1992, 
 GiordanoPutnamSkau1995, BezuglyiDooleyKwiatkowski2006,  
 VianaOliveira2016}. 
 
 We begin with a generalized Bratteli diagram $B = (V, E)$ defined by 
 a sequence of matrices $(F_n)$.
Let $v$ be a vertex from $V_n, n \geq 1$. Then, for every $v_0 \in 
V_0$, 
we consider the set $E(v_0, v)$ (which is non-empty only for finitely 
many vertices $v_0$, see Definition \ref{def GBD}). Set 
$h^{(n)}_{v_0, v} = |E(v_0, v)|$, $v \in V_n$,  and define
$$
H^{(n)}_v = \sum_{v_0 \in V_0} h^{(n)}_{v_0, v}.
$$
It gives  the sequence of vectors $H^{(n)} = (H^{(n)}_{v} : 
v \in V_n)$ which is assigned to vertices of the corresponding
 level $V_n$.
 
 Define the vector $H^{(0)} = (H^{(0)}_{v} : v \in V_0)$ such that 
$H^{(0)}_{v} = 1$ for all $v$. 
Then we see that  the following relation holds.

\begin{lemma}\label{lem vector H} 
$F_n H^{(n)} = H^{(n+1)}$ and $F_n \cdots F_0 
 H^{(0)} = H^{(n+1)}$, $n \in \N_0$.
 \end{lemma}

The proof of this fact follows immediately from the relation
$$
H^{(n+1)}_v = \sum_{w \in V_n} \ent H^{(n)}_w,  \ \ \ v \in V_{n+1}.
$$

Similarly, we can define a clopen subset $X_v^{(n)}$ of $X_B$ setting 
\be\label{eq X_v^{(n)}}
X_v^{(n)} = \bigcup_{v_0 \in V_0}\bigcup_{\ol e \in E(v_0, v)} [\ol e].
\ee
Recall that $X_v^{(n)}$ is a \textit{finite union} of cylinder sets. The 
number
of cylinder sets used  in the definition of  $X_v^{(n)}$  is
x $H^{(n)}_v$ because $h^{(n)}_{v_0, v}$ gives the exact number 
of finite paths between $v_0$ and $v$. The set $X_v^{(n)}$ is 
viewed as a tower assigned to the vertex $v$ and $H_v^{(n)}$ is the 
height of this tower.

\begin{lemma} Let $B =(V,E)$ be a generalized Bratteli diagram. The
 sets $(X_v^{(n)} : v \in V_n)$ constitute a partition $\xi_n$ of 
$X_B$ into disjoint clopen sets for every $n$. The sequence of partitions 
$(\xi_n)$ is a refining sequences such that the elements of all these 
partitions generate the topology (and Borel $\sigma$-algebra) on  $X_B$. 
\end{lemma}

This partition  can be 
viewed as an analogue of a \textit{Kakutani-Rokhlin partition} 
which is widely used in the ergodic theory and Cantor dynamics. The 
difference is that we have not defined a transformation on $X_B$ so far.
For this, we introduce the notion of an
\textit{ordered generalized Bratteli diagram}. We will see  that such
diagrams arise naturally in Borel dynamics. 
 
Let $B=(V,E,\geq)$ be a generalized Bratteli diagram $(V,E)$ equipped 
with a \textit{partial order} $\geq$ defined on each set $E_i,\ 
i=0,1, ...,$ such that two edges
$e,e'$ are comparable if and only if $r(e)=r(e')$. In other words, a 
\textit{linear order} $\geq$ is defined on each (finite) set $r^{-1}(v),
\ v\in V\setminus V_0$. For a Bratteli diagram $(V,E)$ equipped with 
such a partial order $\geq$ on $E$, one can also define a \textit{partial 
lexicographic order} on
the set $E_{k}\circ\cdots\circ E_{l-1}$ of all paths from $V_k$ to $V_l$:
$(e_{k},...,e_{l-1}) > (f_{k},...,f_{l-1})$ if and only if for some $i$ 
with $k \le i <l$, $e_j=f_j$ for $i<j < l$ and $e_i> f_i$. Then we see
 that any two paths from $E(V_0, v)$, the (finite) set of all paths 
 connecting vertices from $V_0$ with $v$, are comparable with respect 
 to the introduced lexicographic
order. We call a path $e= (e_0, e_1,..., e_i,...)$ \textit{maximal 
(minimal)} if  every
$e_i$ has a maximal (minimal) number among all elements from
$r^{-1}(r(e_i))$. Note that there are unique minimal and unique maximal 
finite paths in $E(V_0,v)$ for each $v\in V_i,\ i > 0$.

\begin{definition}\label{order} A generalized Bratteli diagram 
$B=(V,E)$
together with a partial order $\geq$ on $E$ is called an \textit{ordered
Borel-Bratteli diagram} $B=(V,E,\geq)$ if the space $X_B$ has no 
cofinal
minimal and maximal paths. This means that $X_B$ does not contain 
infinite paths $e=
(e_0, e_1,..., e_i,...)$ such that for all sufficiently large $i$ the edges 
$e_i$ have maximal (minimal) number in the set $r^{-1}(r(e_i))$.
\end{definition}

For each ordered Borel-Bratteli diagram $B=(V,E,\geq)$, define a Borel
transformation $\varphi$, which is also called the \textit{Vershik map 
(or automorphism)}, acting on the space 
$X_B$ as follows. Given $x=(e_0, e_1,...)\in X_B$, let $k$ be the
 smallest number such that $e_k$ is not a maximal edge. Let $f_k$ be
the successor of
$e_k$ in $r^{-1}(r(e_k))$. Then we define $\varphi(x)=
(f_0, f_1,...,f_{k-1},f_k,e_{k+1},...)$ where $(f_0, f_1,..., f_{k-1})$ 
is the minimal path in $E(V_0, r(f_{k-1}))$. Obviously, $\varphi$ is a
 one-to-one mapping
of $X_B$ onto itself. Moreover, $\varphi$ is a homeomorphism of $X_B$
where the 0-dimensional topology is defined by cylinder sets.

Thus, given an ordered Borel-Bratteli diagram $B=(V,E,\geq)$, we have
defined a Borel dynamical system $(X_B, \varphi)$. It turns out that 
every Borel aperiodic automorphism of a standard Borel space can be realized as a Vershik transformation acting on the 
space of infinite paths of an ordered Borel-Bratteli diagram.

\begin{theorem}[\cite{BezuglyiDooleyKwiatkowski2006}]
\label{thm T4.2} Let $T$ be an aperiodic Borel
 automorphism acting
on a standard Borel space $(X, \A)$. Then there exists an ordered
Borel-Bratteli diagram $B=(V,E,\geq)$ and a Vershik automorphism 
$\varphi : X_B \to X_B$ such that $(X, T)$ is Borel isomorphic to 
$(X_B,\varphi)$.
\end{theorem}

\subsection{Measures on the path-space of a Bratteli diagram}
\label{ssect measures}
 
In this subsection, we will consider Borel probability measures which are
invariant with respect to the tail equivalence  relation, see Definition
\ref{def tail}. In papers by Vershik and his colleagues such measures are
called central measures, see e.g. \cite{Vershik2015}.

\begin{definition}\label{def tail inv m} Let $B = (V, E)$ be a 
generalized Bratteli diagram, and $X_B$ the path-space of  $B$. 
Let $\mu$ be a Borel measure  on $X_B$. The measure $\mu$ is 
called \textit{tail invariant} if, for any two finite paths  $\ol e$ and 
$\ol e'$ such that $r(\ol e) = r(\ol e')$, one has 
\be\label{eq_inv measures}
\mu([\ol e]) = \mu([\ol e']),
\ee
where $[e]$ and $[e']$ denote the corresponding cylinder sets. 
\end{definition}

Given a generalized Bratteli diagram $B=(V,E)$, let $M_1(B, \E)$ 
denote the set of Borel positive  probability measures on $X_B$ which 
are invariant with respect to the tail equivalence relation $\E$. 
Then the property of tail invariance means that the probability to arrive
at a vertex $v \in V_n$ does not depend on a starting point $w \in V_0$
and does not depend on the path connecting $w$ and $v$.

Hence, if $\mu$ is a fixed measure from $M_1(B, \E)$, relation 
\eqref{eq_inv measures}  allows us to define a  sequence of
non-negative  vectors $(\mu^{(n)})$ where $\mu^{(n)} =
( \mu^{(n)}_v : v \in V_n)$ and 
\be\label{eq def og mu(n)}
\mu^{(n)}_v = \mu([\ol e]), \ \ \ \ol e\in E(V_0, v), \ v \in V_n.
\ee
By tail invariance of $\mu$  the value $\mu^{(n)}_v $ does
not depend on the choice of $\ol e\in E(V_0, v)$. 

\begin{theorem}\label{thm inv measures} 
Let $B =(V,E)$ be  a generalized Bratteli diagram defined
by a sequence of incidence matrices $F_n$. 
Let $\mu$ be a Borel probability measure  on the path-space $X_B$ of 
$B$ which is tail invariant. Then the corresponding sequence of vectors 
$\mu^{(n)}$ (defined as in \eqref{eq def og mu(n)}) satisfies the 
property 
\be\label{eq_inv meas via A_n}
A_n \mu^{(n+1)} = \mu^{(n)}, 
\ee
where $A_n = F_n^T$ is the transpose of $F_n$.

Conversely, if a sequence of vectors $\mu^{(n)}$ satisfies 
\eqref{eq_inv meas via A_n}, then it defines  a unique  tail invariant
 measure $\mu$.
 
 The theorem remains true for $\sigma$-finite measures $\nu$ 
 satisfying the property that
  $\nu([\ol e]) < \infty$ for every cylinder set $[\ol e]$.
\end{theorem}

\begin{proof} The proof is straightforward. Indeed, if $\mu$ is a tail invariant 
measure, then one can define the sequence of vectors $(\mu^{(n)})$ as 
in \eqref{eq def og mu(n)}. Then, for every $n\in \N$, $v \in V_n$,
 and $\ol e\in E(V_0, v)$ with $r(\ol e) =v$, we  see that 
\be\label{eq Kolmog consist}
[\ol e] = \bigcup_{u \in V_{n+1}} \bigcup_{e' \in E(v, u)} [\ol e\ol e'].
\ee
Recall that the cardinality of the set $E(v, u)$ equals $f^{(n)}_{u, v}$. 
By tail invariance of $\mu$, every subset of $X_B$, which is determined by
 an edge from $E(v, u)$, has the same measure, so that relation
 \eqref{eq Kolmog consist} implies that
 $$
 \mu^{(n)}_v = \sum_{u \in V_{n+1}} \sum_{e' \in E(v, u)} \mu([\ol e
 \ol e']) = \sum_{u \in V_{n+1}} f^{(n)}_{u, v} \mu^{(n+1)}_u =
 \sum_{u \in V_{n+1}} a^{(n)}_{v, u} \mu^{(n+1)}_u,
 $$
and these equalities prove \eqref{eq_inv meas via A_n}.

Conversely, suppose a sequence of non-negative vectors $(\mu^{(n)})$
 is given and satisfies \eqref{eq_inv meas via A_n}. Then we define 
a measure $\mu$ on $X_B$ by setting 
\be\label{eq inv for ol e}
\mu([\ol e]) = \mu^{(n)}_v
\ee 
for every $\ol e$ with $r(\ol e) = v$. Relation 
\eqref{eq_inv meas via A_n}
can be interpreted as the Kolmogorov consistency condition, and it 
guarantees that the above definition of a measure on cylinder sets 
can be extended to all Borel subsets of 
$X_B$. It follows from \eqref{eq inv for ol e}  that the obtained measure 
$\mu$ is tail invariant.  

The case of a $\sigma$-finite measure $\nu$ is considered similarly if we
 know that all values $\nu_v^{(n)}$ are finite. 
\end{proof}

Let $B = (V,E)$ be a generalized Bratteli diagram and $(F_n)$ the
sequence of incidence matrices.  We define the
sequence of row stochastic incidence matrices $(\wh F_n)$ with entries
\be\label{eq wh F_n}
\wh f^{(n)}_{vw} = \frac{H^{(n)}_w}{H^{(n+1)}_v}\ent
\ee
where $H^{(n)} = (H^{(n+1)}_v)$ is the vector of heights of 
Kakutani-Rokhlin towers $X^{(n)}_v, v \in V_n$.

\begin{corollary}\label{cor meas of towers}
Let $\mu$ be a tail invariant measure on a generalized Bratteli diagram
$B= (V, E)$. Define $s^{(n)} =  (s^{(n)}_v)$ where $s_v^{(n)} =
\mu_v^{(n)} H^{(n)}_v$ is the measure of the tower $X_v^{(n)}$.
Then 
$$
s^{(n+1)}\wh F_n = s^{(n)}, \quad n \in \N_0, 
$$
where $s^{(n)}$ is considered as a row vector.
\end{corollary}

\begin{proof}
We recall that $F_n H^{(n)} = H^{(n+1)}$ and $\mu^{(n+1)}F_n =
\mu^{(n)}$. Then, for every $w \in V_{n}$,
$$
\ba
(s^{(n+1)}\wh F_n)_w & = \sum_{v \in V_{n+1}} 
\mu_v^{(n+1)} H^{(n+1)}_v \frac{H^{(n)}_w}{H^{(n+1)}_v}\ent\\
& = H^{(n)}_w \sum_{v \in V_{n+1}} \mu_v^{(n+1)}\ent\\
& = \mu_w^{(n)} H^{(n)}_w\\
& = s^{(n)}_w.
\ea
$$
\end{proof}

\begin{remark}
We recall one more method for constructing measures on the path-space
$X_B$ of a generalized Bratteli diagram $B = (V, E)$. By definition, 
$X_B$ is a Borel subset of the product space $\Omega = E_0 \times E_1
 \times \cdots $.  Let $\tau_i$ be a probability distribution  on the 
 countable set $E_i$, $i \in \N_0$, and let $\tau $ denote the product
 measure $\otimes_i \tau_i$. 
 
 Consider a Borel map $\Phi$ from $\Omega$ onto $X_B$. Define a
new measure $\nu$ on $X_B$ by letting 
$$
\nu(A) = \tau(\Phi^{-1}(A)). 
$$

\textit{Question.} Can a tail invariant measure on $X_B$ be obtained 
as a pull back of a product measure? What properties have measures
$\mu$ defined in this way?
\end{remark}

\subsection{ Stationary processes as generalized Bratteli diagrams}
\label{ssect stationary BD}

We recall that a generalized Bratteli diagram $B = B(F_n) $ is called
\textit{stationary} if all incidence matrices $F_n$ are the same, 
$F_n =F$. Therefore, the set of edges $E_n$ does not depend on the level. It is convenient to enumerate the vertices of every level by 
integers $\Z$. 

In the following definition, we included a few basic properties of
countable nonnegative matrices, see \cite{Kitchens1998} for a detailed 
exposition of the Perron-Frobenius theory for countable matrices.

\begin{definition}\label{def PF}
A stationary Bratteli diagram $B =B(F)$ is called \textit{irreducible} if 
for all $i, j \in \Z$ there exists some $n \in \N_0$ such that 
$a^{(n)}_{ij} >0$. (We prefer to work here with the transpose matrix 
$A = F^T)$.) This means that there exists a finite path from 
$i \in V_0$ to $j \in V_{n+1}$. An irreducible matrix $A$ has period
$p$ if, for all vertices $i \in \Z$, 
$$
p = \gcd \{ \ell : a^{(\ell)}_{ii} >0\}.
$$
If $p=1$, the matrix $A$ is called \textit{aperiodic}. 

An irreducible aperiodic nonnegative matrix $A$ admits a 
\textit{Perron-Frobenius eigenvalue} $\lambda$ defined by 
$$
\lambda = \lim_{n\to \infty} (a^{(n)}_{ii})^{\frac{1}{n}}
$$
(the limit exists and does not depend on $i$).

An  irreducible aperiodic nonnegative matrix $A$ is called 
\textit{transient} if 
$$
\sum_{n} a^{(n)}_{ij} \lambda^{-n} < \infty;
$$
otherwise, $A$ is called \textit{recurrent}. For a recurrent matrix $A$,
define $\ell_{ij}(1) = a_{ij}$ and
$$
\ell_{ij}(n+1) =\sum_{k \neq i} \ell_{ik}(n) a_{kj}.
$$
The matrix $A$ is called \textit{null-recurrent} if
$$
\sum_{n} n \ell_{ii}(n) \lambda^{-n} < \infty;
$$
otherwise $A$ is called \textit{positive recurrent}.
\end{definition}

\begin{remark} The terminology used in Definition \ref{def PF} comes
from probability theory. If $A$ is a stochastic matrix, then $a_{ij}$ is
the probability of going from state $i$ to state $j$ in one step. If one
uses $n$ steps to reach $j$ from $i$, then  the probability is 
$a^{(n)}_{ij}$. The quantity $\ell_{ij}(n)$ is the probability of going
from state $i$ to state $j$ in $n$ steps without returning to $i$. 
The stochastic matrix is transient if the expected number of returns to
$i$ is of a random walk beginning at $i$ is finite and $A$ is recurrent
if it is infinite. A recurrent matrix $A$ is null recurrent if the expected 
time of return to $i$ of a walk beginning at $i$ is infinite and positive
recurrent if the expected time is finite.
\end{remark}

The following theorem is taken from \cite{Kitchens1998}.

\begin{theorem}[Generalized Perron-Frobenius theorem]
\label{thm general PF}
Suppose  $A$ is  a countable non­negative irreducible aperiodic matrix.  
Suppose that $A$  is recurrent. Then there exists a Perron-Frobenius
 eigenvalue 
 $$\lambda = 
 \lim_{n\to \infty} (a^{(n)}_{ij})^{\frac{1}{n}} >0
 $$ 
(assumed  to be finite) such that:

(a) there are strictly positive left $s$ and right $t$  eigenvectors 
corresponding  to $\lambda$,

(b) the eigenvectors are unique  up to constant  multiples,

(c) $s \cdot t = \sum_i s_it_i <\infty$ if and  only if $A$ is positive
 recurrent,

(d) if $0 \leq A' \leq A$ and $\lambda'$ is the Perron-Frobenius
 eigenvalue 
for $A'$, then $\lambda' \leq \lambda$; the equality holds if and only
if $A' = A$,

(e) $\displaystyle{\lim_{n\to \infty} A^n\lambda^{-n}=0}$ if $A$ 
is null-recurrent, and  $\displaystyle{\lim_{n\to \infty} A^n
\lambda^{-n}= s\cdot t}$ (normalized so that $s\cdot t =1$) if $A$ is 
positive  recurrent.
\end{theorem}

\begin{example}[\textit{Substitutions on a countable alphabet}]
\label{ex subst}
Let $\mc A$ be an infinite alphabet, $\A =\{a_1, a_2, ...\}$. 
Denote by $\A^*$ the set of all finite words on the alphabet $\A$ 
(including the empty word). Suppose that $\sigma :  \A \to \A^*$ is a 
\textit{substitution}, i.e., $\sigma(a) = a_{i_1} \cdots a_{i_k}$ where
all letters are from $\A$ and the word $\sigma(a)$ may contain 
repeated letters. Consider the matrix $M =(m_{ab})$ of the substitution 
$\sigma$ 
whose entries are defined as follows: $m_{ab}$ is the   number of 
occurrences of the letter $b$ in the word $\sigma(a)$, $a \in \A$. 

The matrix $M$
can be used to construct a stationary ordered Bratteli diagram $B = 
B(M)$ by the following rule. The set of vertices $V_i$ at each level $i$ 
is $\A$. The set of edges $E_i = \bigcup_{a, b \in \A} E(b, a)$ is the
same for each $i$ where $E(b, a)$ is the set of $m_{a, b}$ edges 
between $a \in V_{i+1}$ and $b \in V_i$. To define a linear order on
$r^{-1}(a)$, we identify edges from $r^{-1}(a)$ with the letters
$a_{i_1}, \cdots,  a_{i_k}$ in the word $\sigma(a)$ and assign the 
order on $r^{-1}(a)$ accordingly to the natural left-to-right order
of letters  in $\sigma(a)$, see Figure 2.

\begin{figure}[!htb]\label{fig subst} 
\centering
  \includegraphics[width=0.6\textwidth, height=0.3\textheight]
  {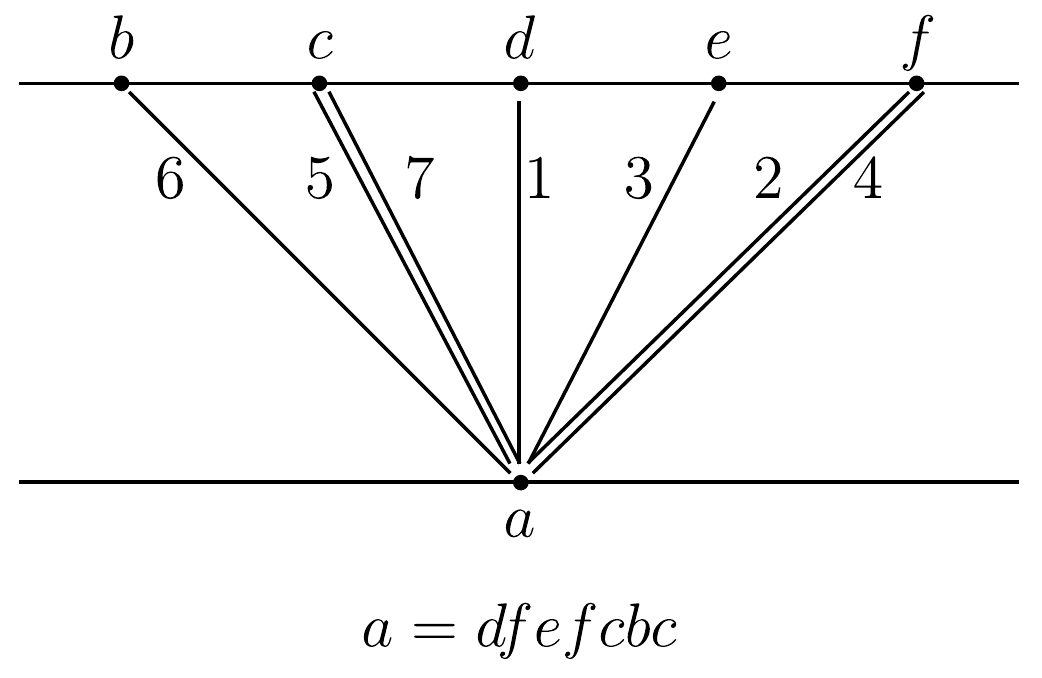}
  \caption{Example of a linear order on $r^{-1}(a)$ (see Example 
  \ref{ex subst})}
  \label{fig:kernels}
\end{figure}

It is known that  stationary  Bratteli diagrams $B$ with finite levels 
are models for
minimal  (or aperiodic if $B$ is non-simple) substitution  dynamical 
systems in symbolic dynamics, see 
\cite{Forrest1997}, \cite{DurandHostSkau1999}, 
\cite{BezuglyiKwiatkowskiMedynets2009}, \cite{Durand2010}. 
We do not know whether a similar result holds for substitutions defined
on a countable alphabet. The reader can find a number of interesting 
results in \cite{Ferenczi2006} about some classes of such substitutions.
\end{example}

For stationary Bratteli diagrams, we can find explicit formulas for 
tail invariant measures.

\begin{theorem} \label{thm inv meas stat BD}
(1) 
Let $B = B(F)$ be a stationary Bratteli diagram such 
that the incidence matrix $F$ (and therefore $A = F^T$) is irreducible,
aperiodic, and recurrent.  Then there exists a tail invariant measure 
$\mu$ on the path-space $X_B$. 

(2) The measure $\mu$ is finite if and only if the right eigenvector 
$t = (t_v)$ has the property $\sum_v t_v <\infty$.
\end{theorem}

\begin{proof} (1) Let $V = V_i$ denote the set of vertices at each level 
of the diagram $B$. By Theorem \ref{thm general PF}, find the 
Perron-Frobenius eigenvalue  $\lambda$ for $A$. Take a right
 eigenvector  $t = (t_v)$ such that $At = \lambda t$.  
 
 For every finite path $\ol e(w, v)$ that begins at $w \in V_0$ and 
 terminates at $v \in V_n$, $n \in \N$, we associate the cylinder set
 $[\ol e(w, v)]$ and set
 \be\label{eq inv meas stat BD}
 \mu([\ol e(w, v)]) = \frac{t_v}{\lambda^{n-1}}. 
 \ee

We need to check that this definition of the measures $\mu$ on cylinder
sets is correct; in other words, it satisfies the Kolmogorov consistency 
condition. Indeed, we have
$$
\ol e(w, v) = \bigcup_{u \in V_{n+1}} \ol f_u
$$
where $f_u = \ol e(w, v)e(v,u)$ is the concatenation of the path 
$\ol e(w, v)$ and the edge $e(v, u)$. Applying \eqref{eq inv for ol e}
and the relation $At = \lambda t$, we compute
$$
\mu \big(\bigcup_{u \in V_{n+1}} \ol f_u\big)  = 
\sum_{u \in V_{n+1}} a_{vu}\frac{t_u}{\lambda^{n}} = 
\frac{t_v}{\lambda^{n-1}} = \mu([\ol e(w, v)]).
$$ 

We note that the measure $\mu$ is tail invariant because, for 
any two finite path $\ol e$ and $\ol e'$  with the same terminal vertex, 
$$
\mu([\ol e(w, v)]) = \mu([\ol e'(w', v)]),  
$$
and the value depends on $v \in V_n$ only. Now we can refer to Theorem 
\ref{thm inv measures} to finish the proof of (1).

(2) Since every $t_v$ is finite and positive, the measure $\mu$ is 
$\sigma$-finite, in general. It follows from the definition of $\mu$ as
in \eqref{eq inv meas stat BD} that 
$$
\mu(X_B)  = \sum_{v \in V_0} t_v.
$$
If this sum is finite, $\mu$ can can be made a probability measure.
\end{proof}

\begin{example}[\textit{ERS and ECS Bratteli diagrams}]\label{ex ERS}
It is said that a matrix $A = (a_{i,j})$ satisfies the \textit{equal row 
sum} property (ERS property) if $\sum_j a_{i,j}$ does not depend on 
$i$. In other words, the sum of entries in each row is the same. If this 
sum is $r$ we will say that $A$ belongs to the class $ERS(r)$. 
Similarly, it is said that a matrix $A$ has the \textit{equal column sum}
property (ECS property)  if $\sum_i a_{i,j}$ is a constant independent of
 $j$. If $A$ is in the class $ECS(c)$, then $c$ is a Perron-Frobenius 
 eigenvalue of $A$ with the corresponding constant eigenvector. 
 
 We observe that the following fact holds. Let matrices $A_1$ 
and $A_2$ belong to the classes $ERS(r_1)$ and $ERS(r_2)$ respectively.
Then the product $A_1A_2$ is in the class $ERS(r_1r_2)$.

We say that a generalized Bratteli diagram $B = B(F_n)$, defined by the
 sequence of incidence 
matrices  $(F_n)$, \textit{has the ERS (ECS)  property} if every matrix
 $F_n$ has this property.  It follows then
that the following result is true for  ERS Bratteli diagrams.
 In particular, a  substitution of constant length generates a
  stationary Bratteli diagram with the ECS property.

\begin{lemma} (1) Let $B$ be an ERS generalized Bratteli diagram such 
that incidence matrices $F_n$ belong to $ERS(r_n)$ for all $n \in \N$.
 Then, for every tail invariant probability measure $\mu$, we have 
\be\label{eq mu for ERS}
\sum_{v \in V_{n+1}} \mu^{(n+1)}_v = \frac{1}{r_0\cdots r_n}, 
\quad n \in \N. 
\ee

(2) If $B$ is an ECS generalized Bratteli diagram and $F_n \in 
ECS(c_n)$, then the measure $\mu$ such that the vectors 
$\mu^{(n)}$ have the entries $\mu^{(n)}_v = (c_0\cdots 
c_{n-1})^{-1}, v \in V_n,$ is tail invariant.
\end{lemma}

\begin{proof} (1) Recall that we assume that $H^{(0)}_v = 1$ for all 
$v \in 
V_0$.  The fact that $F_i \in ERS(r_i), i = 0, 1, ..., n$, means that 
$H_v^{(n+1)} = H_u^{(n+1)} = r_0 r_1 \cdots r_n$. Let $\mu$ 
be a tail invariant probability measure. Then 
$$
\sum_{v\in V_{n+1}} \mu^{(n+1)}_v H_v^{(n+1)} =  r_0 r_1 
\cdots r_n \sum_{v\in V_{n+1}} \mu^{(n+1)}_v = 1,
$$
and relation \eqref{eq mu for ERS} follows.

(2) It suffices to check that relation \eqref{eq_inv meas via A_n} of 
Theorem \ref{thm inv measures} holds. Indeed, for $w \in V_n$,
$$
(A\mu^{(n+1)})_w = \sum_{v \in V_{n+1}} a^{(n)}_{w,v} 
\frac{1}{c_0\cdots c_n} = \frac{1}{c_0\cdots c_{n-1}}.
$$
This  proves that the measure $\mu$ is tail invariant. 
\end{proof}
\end{example}

\section{Transition kernels and Markov measures on the path-space
of generalized Bratteli diagrams}\label{sect Markov}

The central themes in the section are harmonic analysis, dynamics, and 
measures on graphs, but with an emphasis on the special case when the 
dynamics is specified by transition between levels in the graph. We 
use \textit{tail invariant measures}, to play a central role, also in subsequent sections in the paper. 
We consider \textit{discrete Markov processes} on the path-space of a generalized
Bratteli diagram. Relative results can be found in 
\cite{DooleyHamachi2003}, \cite{Vershik2015}, and \cite{Renault2018}.

\subsection{Graph induced Markov measures}

 \begin{definition}\label{def Mark meas} 
Let $B = (V, E)$ be  a generalized Bratteli diagram constructed by 
a sequence of incidence matrices $(F_n)$.  Let $q = (q_{v})$ be a 
strictly positive  vector, $q_v >0, v\in V_0$, and let $(P_n)$ 
be a sequence of non-negative infinite matrices with entries 
$(p^{(n)}_{v,e})$ where $v \in V_n, e \in E_{n}, n= 0, 1, 2, ... $. To 
define a
\textit{Markov measure} $m$, we require that the sequence $(P_n)$
 satisfies the  following properties:
\begin{equation}\label{defn of P_n}
(a)\ \ p^{(n)}_{v,e} > 0 \ \Longleftrightarrow \ (s(e) = v); \ \ \ \ (b)\ \  
\sum_{e : s(e) = v} p^{(n)}_{v,e} =1.
\end{equation}
Condition \eqref{defn of P_n}(a) shows that $p^{(n)}_{v,e}$ is positive 
only on the edges outgoing from the vertex $v$, and therefore the
 matrices $P_n$ and $A_n =F_n^T$ share the same set of zero entries. 
 For any cylinder set $[\overline e] = [(e_0, e_1, ... , e_n)]$
generated by the path $\ol e$ with $v =s(e_0) \in V_0$, we set
\begin{equation}\label{eq m([e])}
m([\overline e]) = q_{s(e_0)}p^{(0)}_{s(e_0), e_0} \cdots 
p^{(n)}_{s(e_n), e_n}.
\end{equation}
Relation \eqref{eq m([e])} defines  a measure $m$ of the set $[\ol e]$.
By \eqref{defn of P_n}(b),  this measure satisfies the \textit{Kolmogorov 
consistency condition} and can be extended to the $\sigma$-algebra 
of  Borel sets. 
To emphasize that $m$ is generated by a sequence of \textit{stochastic 
matrices}, we will also write $m = m(P_n)$.
\end{definition}

\begin{remark}\label{rem P_n gnrts kernels}
(1) More generally, we can consider a sequence of matrices $(P_n)$ 
such that, for every $v \in V_n$, 
\be\label{eq finite sums}
\sum_{e \in E_n: s(e) =v} p^{(n)}_{v,e} < \infty, \ \ \ n \in \N_0.
\ee
Then, it follows from Definition \ref{def Mark meas} that the sequence 
of matrices $(P_n)$ determines finite transition kernels: given 
a vertex $v \in V_n$ and a set $B \subset V_{n+1}$, we define  
\be\label{eq-P_n kernel}
P_n(v, B) := \sum_{e \in E_n : s(e) = v, r(e) \in B} 
p^{(n)}_{v, e}. 
\ee

(2) If the matrices $P_n$ satisfy \eqref{eq finite sums}, then, without 
loss of generality, one can normalize $P_n$ such that the sum of entries
in every row is one. 
\end{remark}

Define 
$$
\wh p_n(v, u) = \sum_{e\in E(v,u)} p^{(n)}_{s(e), e} 
$$
where $E(v, u) = \{ e \in E_n : s(e) = v, r(e) = u\}$. Then the 
entries $\wh p_n(v, u)$ form a matrix $\wh P_n$ indexed by the set
$V_n \times V_{n+1}$, $n \in \N_0$. 

The following lemma is a simple observation. 

\begin{lemma}\label{lem stoch wh P_n}
 The matrices $P_n$ and $\wh P_n$ have the same row sums:
$$
\sum_{e : s(e) = v} p^{(n)}_{s(e), e} = \sum_{u \in V_{n+1}}
\wh p_n(v, u), \ \ \ v \in V_n, n \in \N_0.
$$
Hence $\wh P_n$ is a row stochastic matrix if $P_n$ is a Markov matrix.
\end{lemma}

Using the initial vector $q =q^{(0)}$ and the sequence of matrices 
$(P_n)$, one can define another sequence of co-transition vectors 
$(q^{(n)})$ whose entries are assigned to the vertices of 
$V_n, n \in \N$.

\begin{proposition}\label{prop vectors q}
(1) Let $(P_n)$ be a sequence of matrices that defines a 
Markov measure $m = m(P_n)$, and let $q^{(0)}$ be an initial distribution
 on  $V_0$. Then the formula  
\be\label{eq def vect q^k}
q^{(k+1)}_v  = \sum_{e : r(e) = v} q^{(k)}_{s(e)} 
p^{(k)}_{s(e), e}, \quad v \in V_{k+1},
\ee
defines inductively a sequence of positive  vectors $q^{(k)} = 
(q^{(k)}_v : 
v \in V_k),\ k \geq 1$. Equivalently, relation \eqref{eq def vect q^k} 
is represented as  $q^{(k+1)} = q^{(k)}\wh P_k, k \in
 \N_0$.

(2) If $q^{(0)}$ is a probability vector, then all vectors $q^{(k)}$ are 
probability. 
 
\end{proposition} 

\begin{proof} (1) We first note that $q_v^{(k+1)}$ is well defined
 because  the sum in \eqref{eq def vect q^k} is finite.  Next, 
 the set of all edges from $ r^{-1}(u)$ can be represented as 
 $$
 r^{-1}(u) = \bigcup_{v\in V_n} E(v, u).
 $$
Hence, for  $u \in V_{k+1}$,
$$
\ba
q^{(k+1)}_u & = \sum_{e : r(e) = u} q^{(k)}_{s(e)} 
p^{(k)}_{s(e), e}\\
& = \sum_{v \in V_n} \sum_{e \in E(v, u)}  q^{(k)}_{s(e)} 
p^{(k)}_{s(e), e}\\
& = \sum_{v \in V_n} q^{(k)}_{v} \wh p_n(v, u)\\
\ea 
$$
which proves (1).

We now check that (2) holds.
 Suppose that we proved the statement for $ i = 1, ... ,k$.
We will show that $q^{(k+1)}$ is also a probability vector.
$$
\ba 
\sum_{u\in V_{k+1}}  q^{(k+1)}_u =  & \sum_{u\in V_{k+1}} 
\sum_{e  \in r^{-1}(u)} q^{(k)}_{s(e)} p^{(k)}_{s(e), e}\\
= & \sum_{v\in V_{k}} 
\sum_{e  \in s^{-1}(v)} q^{(k)}_{v} p^{(k)}_{s(e), e}\\
= & \sum_{v\in V_{k}} q^{(k)}_{v}
\sum_{e  \in s^{-1}(v)}  p^{(k)}_{s(e), e}\\
= &  1.
\ea
$$
We remark that, in the above calculation, we used the summation over the
set of  all edges from $E_k$ represented in two different but equivalent 
ways, namely, 
$$
E_k = \bigcup_{u\in V_{k+1}} \bigcup_{e \in r^{-1}(u)} e
=  \bigcup_{v\in V_{k}} \bigcup_{e \in s^{-1}(v)} e.
$$
\end{proof}

Beginning with a vector $q^{(0)}$ and a sequence of probability 
transition kernels $(P_n)$, we constructed the vectors $q^{(n)}$.
We want to find out whether there exists a transition kernel 
$Q_n(u, \cdot): V_{n+1} \to V_n$ such that, for any subsets 
$A \subset V_n$ and $B \subset V_{n+1}$, the relation 
\be\label{eq balance eqn}
\sum_{v \in A} q^{(n)}_v P_n(v, B) = 
\sum_{u \in B} q^{(n+1)}_u Q_n(u, A)
\ee
holds, where $P_n(v, B)$ is defined by \eqref{eq-P_n kernel}.

\begin{theorem}\label{thm exist of Q_n}
For a given positive  vector $q^{(0)}$ and  sequence of probability
 transition kernels $(P_n)$, there exist a sequence of probability
  transition kernels $Q_n$ such that relation 
  \eqref{eq balance eqn} holds.
\end{theorem}

\begin{proof} We first find the sequence of vectors $q^{(n)}, n \in 
\N,$ according to Proposition \ref{prop vectors q}. For fixed vertices
$v \in V_n, u \in V_{n+1}$, we set 
$$
\wh q_n(u, v) = \frac{q_v^{(n)}}{q_u^{(n+1)}} \wh p_n(v, u).
$$
For $e \in E(v, u)$,  we take 
$$
q_{r(e), e}^{(n)}:= \frac{1}{|E(v, u)|} \wh q_n(u, v)
$$
and define 
$$
Q_n(u, A) = \sum_{v \in A} \wh q_n(u, v) = \sum_{v \in A}
\sum_{e \in E(u, v)} q_{r(e), e}^{(n)}.
$$
The fact that $Q_n$ is probability follows from the relations:
$$
\sum_{v \in V_n} \wh q_n (u, v) = \sum_{e \in r^{-1}(u)} 
q_{r(e), e}^{(n)}
$$
and 
$$
\sum_{v \in V_n} \wh q_n (u, v) = \frac{1}{q_u^{(n+1)}}
\sum_{v \in V_n} q_v^{(n)}\wh p_n(v, u) = 1
$$
because of \eqref{eq def vect q^k}, see Proposition 
\ref{prop vectors q}.

Finally, we check that $Q_n$ satisfies  \eqref{eq balance eqn}:
$$
\ba 
\sum_{v \in A} q^{(n)}_v P_n(v, B) & = 
 \sum_{v \in A}\ \sum_{e: r(e) \in B, s(e) = v} q^{(n)}_v 
 p^{(n)}_{s(e), e} \\
 & =  \sum_{v \in A}\  \sum_{u \in B} \ \sum_{e: \in E(v, u)}
 q^{(n)}_v  p^{(n)}_{s(e), e} \\
 & = \sum_{v \in A}\  \sum_{u \in B} q^{(n)}_v \wh p_n(v, u)\\
 & =  \sum_{v \in A}\  \sum_{u \in B}  q^{(n+1)}_u \wh q_n(u, v)\\
 &= \sum_{u \in B}  q^{(n+1)}_u Q(u, A).
\ea
$$
\end{proof}

From the proof of Theorem \ref{thm exist of Q_n} we deduce the 
following  facts. (We use here  the notation introduced above.) 

\begin{proposition} \label{prop about wh Q_n}
(1) The sequence of transition kernels $Q_n$ is 
uniquely determined by 
$q^{(0)}$ and the kernels $(P_n)$ if and only if the Bratteli diagram 
has no multiple edges, i.e., $|E(v, u)| \leq  1$ for all vertices $v$ 
and $u$. 

(2) If $\wh Q_n$ is the infinite positive matrix  defined by its  entries
$\wh q_n (u, v)$, then $q^{(n+1 )} \wh Q_n = q^{(n)}$.

(3) Relation \eqref{eq balance eqn} holds also for the matrices 
$\wh P_n$ and $\wh Q_n$. In particular, for any $v \in V_n, u\in 
V_{n+1}$, $n \in \N_0$,
\be\label{eq wh p and wh q}
q^{(n)}_v \wh p(v, u) =q^{(n+1)}_u \wh q(u, v). 
\ee
\end{proposition}

\begin{proof}
(1) Indeed, the result follows from the proof of Theorem 
\ref{thm exist of Q_n}: 
$$
|E(v, u)| =1 \ \Longleftrightarrow \ q^{(n)}_{r(e), e} = 
\wh q_n(u, v)\ \ \forall e \in E(v, u). 
$$ 
Moreover, if the quantities $\wh q_n(u,v)$ are defined as in Theorem 
\ref{thm exist of Q_n}, then there are infinitely many solutions 
 of  the equation 
$$
\sum_{e \in E(v,u)} q_{r(e), e}^{(n)} = \wh q_n(u, v)
$$
if and only if $|E(v, u)| >1$.

(2) We compute
$$
\sum_{u \in V_{n+1}} q^{(n+1)}_u \wh q_n(u,v) 
= \sum_{u \in V_{n+1}} q^{(n)}_v \wh p_n(v,u) 
= q^{(n)}_v
$$
since $\wh P_n$ is a row stochastic matrix.

(3) It can be checked that
$$
\ba 
\sum_{v \in A} q^{(n)}_v \wh P_n(v, B) & = \sum_{v \in A} 
q^{(n)}_v \sum_{u \in B}  \wh p_n(v, u)\\
&= \sum_{u \in B} \sum_{v \in A} q^{(n+1)}_u \wh q_n(u, v)\\
& = \sum_{u \in B}  q^{(n+1)}_u \wh Q_n(u, A).
\ea
$$
Relation \eqref{eq wh p and wh q} is a particular case of the proved 
property.
\end{proof}

The proposition below clarifies the meaning of vectors $q^{(n)}, n \geq 1$. 
Recall that a  measure on the path-space of a Bratteli diagram is 
completely determined by its values on the cylinder sets which are 
represented by finite paths in a Bratteli diagram. For every vertex 
$v \in V_n, n \in \N$, we defined the tower $X_v^{(n)}$ of the 
Kakutani-Rokhlin partition which is 
formed by all cylinder sets that end at $v$.  It turns out that the 
measures of towers $X_v^{(n)}$ are exactly the entries of the vector
 $q^{(n)}$.

\begin{theorem}
Let $m$ be a Markov measure defined by a sequence of Markov matrices 
$(P_n)$. Let $(q^{(k)})$ be a sequence of probability vectors 
constructed accordingly to \eqref{eq def vect q^k}. Then 
$$
m(X_v^{(n)}) = q^{(n)}_v,\quad v \in V_n,  n \in \N_0,
$$
where $X_v^{(n)}$ is the Kakutani-Rokhlin tower corresponding to the 
vertex $v$. 
\end{theorem}

\begin{proof}
The statement can be proved by induction. Indeed, it is trivial for 
$n =0$. For $n = k+1$, we see that
$$
m(X_v^{(k+1)}) = \sum_{e\in E_n : r(e) = v} m(X^{(k)}_{s(e)}) 
p^{(k)}_{s(e), e} = q^{(k+1)}_v,\ \ \ v \in V_{k+1}, 
$$
since $m(X^{(k)}_{s(e)} ) = q^{(k)}_v$ by the induction hypothesis. 
\end{proof}

\subsection{Tail invariant Markov measures}
We will consider here Markov measures that are 
invariant with respect to the tail equivalence relation.

The next result proves that every tail invariant probability measure is, in 
fact, a Markov measure. The sequence of Markov matrices can be 
explicitly described. Recall that $ M_1(B, \mathcal E)$ denotes 
the set of all tail invariant probability measures.

\begin{theorem}\label{thm inv meas is Markov}
Let $\nu \in M_1(B, \mathcal E)$ be a tail invariant probability 
measure on the path-space $X_B$  of a generalized Bratteli diagram
$B = (V, E)$. Then there
exists a sequence of Markov matrices $(P_n)$ such that $\nu = m(P_n)$. 
\end{theorem}

\begin{proof} Given an invariant probability  measure $\nu$, we have 
 the sequence 
of vectors $(\nu^{(n)})$ that satisfies \eqref{eq_inv meas via A_n} and
is uniquely determined by $\nu$. We recall that the entry 
$\nu^{(n)}_v$ gives the measure of a cylinder set which is determined 
by a finite path from $V_0$ to the vertex $v \in V_n$. In the proof, we
 will construct inductively a sequence of Markov matrices
$(P_n)$ such that the corresponding Markov measure $m = m(P_n)$ 
satisfies the property: for all $n \in \N_0$ and $v\in V_n$,  
$\nu([\ol e]) = m([\ol e])$ where $\ol e$ is a  finite path  with 
$r(\ol e) = v$.

For $n =0$, we  define the probability vector $q^{(0)}$ by setting 
\be\label{eq q^0}
q^{(0)}_v = \nu (X^{(0)}_v) = \nu^{(0)}_v, \ \  v \in V_0.
\ee 
To define the entries $p^{(0)}_{v, e}$ of $P_0$ where $v \in V_0$,  we
 set first $p^{(0)}_{v, e} = 0$ if $v \neq s(e)$. Recall that 
\be\label{eq nu^0 via nu^1}
\sum_{v \in V_1} a^{(0)}_{w,v} \nu^{(1)}_v = \nu_w^{(0)}, \quad
a^{(0)}_{w,v} \in A_0.
\ee
(see \eqref{eq_inv meas via A_n}) where  $A_n = F_n^T$. This fact allows us to define the entries of $P_0$ by
\be\label{eq def of P_0}
p^{(0)}_{v_0, e} = \frac{\nu^{(1)}_{v_1}}{q^{(0)}_{v_0}}
= \frac{\nu^{(1)}_{v_1}}{\nu^{(0)}_{v_0}},\quad \forall e\in  
E(v_0, v_1).
\ee
It follows from \eqref{eq def of P_0} that $P_0$ is a Markov matrix 
because 
$$
\ba
\sum_{e : s(e)  = v_0} p^{(0)}_{v_0, e} =  & \ \sum_{v\in V_1} \ \ 
\sum_{e \in E(v_0, v)}   p^{(0)}_{v_0, e}\\
 = & \  \sum_{v\in V_1} \frac{\nu^{(1)}_v}{q^{(0)}_{v_0}} 
 a^{(0)}_{v_0,v} \\
 = & \  1
\ea 
$$
in virtue of \eqref{eq q^0} and \eqref{eq nu^0 via nu^1}. Relation 
\eqref{eq def of P_0} says that the value of the entry 
$p^{(0)}_{v_0, e}$ does not depend
on an edge $e \in E_0(v_0,v_1)$. 
Hence we can denote it by $p^{(0)}_{v_0,v_1}$. 

Assume that the Markov matrices $P_i$ are defined for $i =0, 1, ... ,
 n-1$, and 
their entries  satisfy the property $p^{(i)}_{v_i, e} = p^{(i)}_{v_i, e'} 
= p^{(i)}_{v_i, v_{i+1}}$ for all $e, e' \in E(v_i,v_{i+1})$. We 
define the entries of $P_n$ as follows: $p^{(n)}_{v_n, e} = 0$ if $v_n
 \neq s(e)$, and  
\be\label{eq entries of P_n}
p^{(n)}_{v_n, e} = \frac{\nu_{v_{n+1}}^{(n+1)}}
{q^{(0)}_{v_0}  p^{(0)}_{v_0,v_1} 
\cdots p^{(n-1)}_{v_{n-1},v_n} } ,\quad \forall e \in 
E(v_n, v_{n+1}), \ v_n  \in V_n,  v_{n+1} \in V_{n+1}.
\ee
The meaning of this definition is explained by the following fact.
Let $(v_0, v_1, ..., v_{n+1})$ be a finite sequence of vertices such 
that there exists a path $\ol e$ from $v_0$ to $v_{n+1}$. The 
$\nu$-measure
of the corresponding cylinder set is $\nu_{v_{n+1}}^{(n+1)}$. 
We want to define a Markov matrix $P_n$ such that the Markov measure
${q^{(0)}_{v_0}  p^{(0)}_{v_0,v_1}\cdots p^{(n)}_{v_{n},
v_{n+1}} }$   of $[\ol e]$ is exactly $\nu_{v_{n+1}}^{(n+1)}$.

We claim that relation \eqref{eq entries of P_n} is equivalent to
\be\label{eq entry of P_n}
p^{(n)}_{v_n, e} = \frac{\nu_{v_{n+1}}^{(n+1)}}
{\nu_{v_{n}}^{(n)}},\ \ e \in E(v_n, v_{n+1}). 
\ee
Indeed, by induction 
$$
\nu_{v_{n+1}}^{(n+1)} =  q^{(0)}_{v_0}  p^{(0)}_{v_0,v_1}
\cdots p^{(n)}_{v_{n},
v_{n+1}}  = q^{(0)}\frac{\nu^{(1)}_{v_1}}{\nu^{(0)}_{v_0}}
\frac{\nu^{(2)}_{v_2}}{\nu^{(1)}_{v_1}} \cdots 
\frac{\nu^{(n)}_{v_n}}{\nu^{(n-1)}_{v_{n-1}}} p^{(n)}_{v_n, e} =
\nu^{(n)}_{v_n} p^{(n)}_{v_n, e}.
$$

It remains to check that the matrix $P_n$ is row stochastic:
$$
\ba 
\sum_{e : s(e) = w} p^{(n)}_{w, e} & =  \ \sum_{v\in V_{n+1}} \ \ 
\sum_{e \in E_n(w, v)} \frac{\nu_v^{(n+1)}}{\nu_{v_{n}}^{(n)}} \\
& =  \ \sum_{v\in V_{n+1}}  a^{(n)}_{v_n, v_{n+1}}
\frac{\nu_{v_{n+1}}^{(n+1)}}{\nu_{v_n}^{(n)}} \\
& = 1.
\ea
$$
We used here relation \eqref{eq_inv meas via A_n} of  Theorem 
\ref{thm inv measures}. We remark that this property can be deduced
also from \eqref{eq entries of P_n}. 

\end{proof}

\begin{remark}
(1) We note that the Markov measure $m$ constructed by the invariant 
measure
$\nu$ has the property of equal values on all edges connecting two fixed
vertices: for every $e \in E(w, v)$, $p^{(n)}_{s(e), e} = 
p^{(n)}_{w, v}$ where  $w\in V_n, v\in V_{n+1}, n \in \N_0$. This 
means that, for  every
two finite paths  $\ol e, \ol e'$ such that $s(\ol e) = s(\ol e')$, 
$r(\ol e) = r(\ol e')$, and going through the same vertices $v_0 = 
s(\ol e), v_1, ... v_k = r(\ol e)$, the Markov measure of the 
corresponding cylinder sets coincide.

(2) It can be easily seen from the proof of Theorem 
\ref{thm inv meas is Markov} that the Markov measure $m(P_n)$ is 
uniquely determined by the tail invariant measure $\nu$. 
\end{remark}

Let $\nu\in M_1(B, \mc E)$ be  a tail invariant probability measure.
The following result clarifies the meanings of the vectors  
$q^{(n)}$ and the sequences of matrices $\wh P_n$ and $\wh Q_n$
generated by $\nu$.

\begin{corollary}\label{cor wh Q}
 Let $\nu\in M_1(B, \mc E)$, and let the sequence
of probability vectors $(s^{(n)})$ determine the measures of 
Kakutani-Rokhlin towers $X_v^{(n)}, v \in V_n$ as in Corollary
\ref{cor meas of towers}. Then co-transition probabilities $q^{(n)}$ 
coincide with $s^{(n)}$, i.e., $q^{(n)}_v = 
H_v^{(n)}\nu_v^{(n)}, v \in V_n, n \in \N_0$.
Furthermore, $\wh Q_n = \wh F_n$ for all $n$ where $\wh F_n$ is 
defined  in \eqref{eq wh F_n}.
\end{corollary}

\begin{proof} We will prove the result by induction. We first find the
matrices $\wh P_n$. It follows from 
\eqref{eq entry of P_n} that entries of $\wh P_n$ are 
$$
\wh p_n(v_{n+1}, v_n) = |E(v_n, v_{n+1})| 
\frac{\nu_{v_{n+1}}^{(n+1)}}{\nu_{v_{n}}^{(n)}} =
a^{(n)}_{v_n, v_{n+1}} \frac{\nu_{v_{n+1}}^{(n+1)}}
{\nu_{v_{n}}^{(n)}}.
$$
Let $q^{(0)}= \nu^{(0)}$ as in the proof of Theorem 
\ref{thm inv meas is Markov}. Since $H^{(0)}_v =1, v \in V_0$, the 
result holds for $n =0$.  Define $q^{(n)} = q^{(0)}
\wh P_0 \cdots \wh P_{n-1}$. Suppose that we proved the statement
for $ i \leq n$. Compute the entries of $q^{(n+1)} = q^{(n)}\wh P_n$:
$$
\ba
q^{(n+1)}_{v_{n+1}} & = \sum_{v_n\in V_n} q^{(n)}_{v_n}
a^{(n)}_{v_n, v_{n+1}} \frac{\nu_{v_{n+1}}^{(n+1)}}
{\nu_{v_{n}}^{(n)}}\\
& = \sum_{v_n\in V_n} H_{v_n}^{(n)}\nu_{v_n}^{(n)}
a^{(n)}_{v_n, v_{n+1}} \frac{\nu_{v_{n+1}}^{(n+1)}}
{\nu_{v_{n}}^{(n)}}\\
& = \nu_{v_{n+1}}^{(n+1)} \sum_{v_n\in V_n} f^{(n)}_{v_{n+1},
v_n} H_{v_n}^{(n)}\\
& = \nu_{v_{n+1}}^{(n+1)} H_{v_{n+1}}^{(n+1)}.
\ea
$$

Having the vectors $q^{(n)}$ determined, we can find the matrices
$\wh Q_n$, see Theorem \ref{thm exist of Q_n} (we use here that
 $F_n^T =  A_n$):
$$
\ba
\wh q_n(v_{n+1}, v_n) & = \frac{q_{v_n}^{(n)}}
{q_{v_{n+1}}^{(n+1)}} \wh p_n(v_n, v_{n+1})\\ 
& = \frac{H_{v_n}^{(n)}\nu_{v_n}^{(n)}}
{H_{v_{n+1}}^{(n+1)}\nu_{v_{n+1}}^{(n+1)}}
a^{(n)}_{v_n, v_{n+1}} \frac{\nu_{v_{n+1}}^{(n+1)}}
{\nu_{v_{n}}^{(n)}}\\
& = f^{(n)}_{v_{n+1}, v_n}\frac{H_{v_n}^{(n)}}
{H_{v_{n+1}}^{(n+1)}}\\
& = \wh f^{(n)}_{v_{n+1}, v_n}.
\ea
$$
This proves the equality $\wh Q_n = \wh F_n$.
\end{proof}

\begin{remark}
The fact proved in Corollary \ref{cor wh Q} says that the matrix 
$\wh Q$ equals to the matrix $\wh F$ independently of a tail invariant
measure $\nu$. In other words, $\wh Q$ does not depend on $\nu$.

\end{remark}

\textit{Stationary Markov measure.} The following observation is useful for the study of Markov measures on
stationary generalized Bratteli diagrams  $B$ with the incidence matrix
$F$.  It is natural to consider a special subset of Markov measures 
$m = m(P)$ on such diagrams, the so called \textit{stationary Markov
 measures}. They
are determined by the property that all matrices $P_n, n\in \mathbb N,$
are the same and equal  to a fixed matrix  $P$. Formula 
\eqref{eq m([e])}, which defines a Markov measure, is transformed then
as follows:
\begin{equation}\label{nu(P)}
m([\overline e]) = q_{s(e_0)}p_{s(e_0), e_0}p_{s(e_1), e_1} \cdots p_{s(e_n), e_n}.
\end{equation}

Suppose that $B$ is a stationary Bratteli diagram  and $\mu$ is a tail
 invariant  measure satisfying \eqref{eq inv meas stat BD}, see Theorem
\ref{thm inv meas stat BD}. In other words, we assume
that the transpose $A$ of the incidence matrix $F$ satisfies the 
conditions
of the Perron-Frobenius theorem: there exists a positive right 
eigenvector  $x = (x_v)$ corresponding to the
 Perron-Frobenius eigenvalue $\lambda$. 

\begin{lemma} Let $B =B(F)$ be a stationary Bratteli diagram such that
the matrix $A= F^T$ satisfies Theorem \ref{thm general PF} and
let  $x = (x_v)$ be the right Perron-Frobenius eigenvector corresponding
to the eigenvalue $\lambda$. Suppose that $\mu$ is a tail invariant  
measure on $X_B$  defined by \eqref{eq inv meas stat BD}. 
Then $\mu$ can be determined as a stationary Markov measure 
$m(P)$  on $X_B$ where the initial distribution $q^{(0)}$ is the 
vector $x$, and the Markov matrix $P$ has the entries
 $$
 p_{s(e), e} = \frac{x_{r(e)}}{\lambda x_{s(e)}},  \ \ \ e \in E.
 $$
 \end{lemma}

\begin{proof} It was proved in Theorem \ref{thm inv meas is Markov}
that any tail invariant measure is a Markov measure for an 
appropriate choice of the matrices $(P_n)$. In conditions of the lemma,
we can specify that $P_n = P$ and  vectors $\mu^{(n)} = (x_v 
\lambda^{n-1}_{v \in V_n})$. It follows then from 
\eqref{eq entry of P_n} that, for every $e \in E(w, v)$, $w\in V_0,
v\in V_1$, 
$$
p_{w, e} = \frac{\mu_v^{(1)}}{\mu^{(0)}_w} = 
\frac{x_{v}}{\lambda x_{w}}.
$$ 
\end{proof}

\subsection{Existence of finite tail invariant measures}
\label{ssect Existence}

It is well known that every homeomorphism of a compact metric space has
 an invariant probability measure. Rephrasing this statement, we 
 conclude that every (classical) Bratteli diagram has a probability tail
 invariant measure. This follows from the fact that every homeomorphism 
 of a Cantor set admits its realization on the path-space of a Bratteli
 diagram whose orbits are essentially the same as orbits of the tail 
 equivalence relation. We refer to \cite{HermanPutnamSkau1992},
 \cite{Medynets2006}, \cite{BezuglyiKarpel2020} where the reader can
 find more details. 
 
The situation with generalized Bratteli diagrams is more difficult. First of
all, there are Borel automorphisms $T$ of a standard Borel space 
$(X, \B)$ that do not admit a probability invariant measure. Two 
Borel sets, $A$ and $B$, are called \textit{equivalent} ($A \sim B$ in symbols) with respect to $T$ if 
there exists a one-to-one Borel map $f : A \to B$  such that $f(x) $ is 
in the $T$-orbit of $x$, i.e., $(x, f(x)) \in E_T$ where $E_T$ is the
orbit equivalence relation on $X\times X$. It is said that $A 
\preceq B$ if $A \sim B'$ where $B' \subset B$. It can be shown that
$A\sim B$ if and only if $A \preceq B$ and $B \preceq A$. 
A Borel aperiodic automorphism $T$ is called \textit{compressible} if
there is a Borel set $A$ such that $A \sim X$ and $X\setminus A$ is 
a complete section, that it it meets every $T$-orbit. 
It turns out that compressible automorphisms do not admit finite
invariant measures.

\begin{theorem}[\cite{Nadkarni1990, Nadkarni1995}] 
\label{thm compr}
Let $T$ be an aperiodic Borel automorphism of a standard Borel space.
The following are equivalent:

(i) $T$ is not compressible,

(ii) $T$ admits an invariant probability measure.
\end{theorem}

It follows from Theorem \ref{thm T4.2} that we can apply Theorem
\ref{thm compr} to generalized Bratteli diagrams.

\begin{corollary}
There exist generalized Bratteli diagrams that do not admit tail invariant
probability measures. 
\end{corollary}

\begin{example}[\cite{Ferenczi2006}]\label{ex subst DS}
Let $\sigma : \A \to \A^*$ be a substitution on an countably infinite
alphabet, see  Example \ref{ex subst}. We say that a finite word $w$
belongs to the language $\mc L(\sigma)$ of the substitution $\sigma$
if $w$ is a subword of $\sigma^n(a)$ for some $a\in \A$ and $n \in \N$.
Let now $X$ be a subset of one-sided sequences $x = x_0x_1 \cdots $
from $\A^{\N}$ such that every finite word occurring in $x$ is in the
language $\mc L(\sigma)$. 

It can be easily seen that $X$ is a closed subset of the Polish space 
$\A^{\Z}$ (with respect to the product topology). Let $T$ be the 
two-sided shift on $\A^{\Z}$. Then $TX \subset X$. The pair $(X, T)$
is a non-compact symbolic dynamical system associated to the 
substitution $\sigma$. 

Let now $\A = 2\Z$ and the substitution $\sigma_0$ is defined by the
rule:
$$
\sigma_0 : n \to (n-2) nn (n+2),\quad n \in 2\Z.
$$
In \cite{Ferenczi2006}, $\sigma_0$ is called the 
\textit{squared drunken man 
substitution}. The following result was proved there.

\begin{lemma}
The Borel dynamical system $(X, T)$ associated to $\sigma_0$ has 
no finite invariant measure.
\end{lemma}

\end{example} 

We finish this subsection by giving another result on the existence of
finite invariant measures. We recall that a substitution $\sigma : \A
\to \A^*$ is called of \textit{constant length} $L$ if $|\sigma(a)| = L$
for all $a \in \A$. It is obvious that if $M$ is the matrix of substitution 
$\sigma$ of constant length $L$, then the vector $(... , 1, 1, ....)$ is
the right eigenvector corresponding to the eigenvalue $L$. 

\begin{lemma}[\cite{Ferenczi2006}] Let $\sigma$ be a constant length
substitution on a countable alphabet $\A$. Suppose that the matrix
of substitution is irreducible, aperiodic, and positive recurrent. Then
the associated Borel dynamical system $(X, T)$ admits a probability
invariant measure. 
\end{lemma}

An example that illustrates the above lemma is the following:
$\A = \N_0$ and 
$$
\sigma(0) = 01, \ \sigma(1) = 02, \ \sigma(n) = (n-2)(n+1), \ n \geq 2.
$$

\subsection{Operators generated by transition kernels}
\label{ss operators}

Let $B = B(F_n)$ be a generalized Bratteli diagram. Suppose that  
$(P_n)$ is a sequence of Markov matrices (equivalently, probability
transition kernels), and $q^{(0)}$ is an initial distribution (a discrete 
infinite measure on $V_0$). As was shown above, the matrices $P_n$
and vector $q^{(0)}$ determine the stochastic matrices $\wh P_n$
 (Lemma 
\ref{lem stoch wh P_n}),  vectors $q^{(n)}$ (Proposition 
\ref{prop vectors q}) such that $q^{(n+1)} = q^{(n)} \wh P_n$, and 
dual probability kernels $\wh Q_n$ such that $q^{(n+1)} \wh Q_n =
q^{(n)}$ (Theorem \ref{thm exist of Q_n} and Proposition 
\ref{prop about wh Q_n}). We use these objects to define operators
acting in the weighted $\ell^2$-spaces.

For each $P_n$, we can define a linear operator: if $f$ is a bounded 
function  defined on the set of vertices $V_{n+1}$, then setting
$$
(T_{P_n}f)(v)  = \sum_{e : s(e) = v} p^{(n)}_{v, e} f(r(e)), \ \ v 
\in V_n,
$$  
we obtain a function defined on vertices of $V_n$. For $Q_n$, we have
$$
(T_{Q_n}g)(u)  = \sum_{e : r(e) = v} q^{(n)}_{u, e} g(s(e)), \ \ u \in
 V_{n+1}.
$$  
Similarly, 
$(T_{\wh P_n}f)(v) = \sum_{u \in V_{n+1}} \wh p_n(v, u)f(u)$ and 
$(T_{\wh Q_n}g)(u) = \sum_{v \in V_{n}} \wh q_n(u, v)f(v)$.

\begin{lemma}\label{lem operators T}
For $P_n, \wh P_n$ and $Q_n,  \wh Q_n$ as above, we have
$$
T_{P_n}(f) = T_{\wh P_n}(f), \quad T_{Q_n}(g) = T_{\wh Q_n}(g), 
\ \ n\in \N_0, 
$$
where $f$ is a bounded function on $V_{n+1}$, and $g$ is a bounded 
function on $V_{n}$.
\end{lemma}

\begin{proof} We calculate 
$$
\ba 
T(P_n)f(v)  & = \sum_{e : s(e) = v} p^{(n)}_{v, e} f(r(e)) \\
& = \sum_{u \in V_{n+1}} \sum_{e \in E(v, u)} 
 p^{(n)}_{v, e} f(r(e))\\
& =  \sum_{u \in V_{n+1}} \wh p_n(v, u)f(u) \\
& = T(\wh P_n)f(v).
\ea
$$
The other relation is proved analogously. 
\end{proof}

Let $\mc H_n$ be the linear space of all sequences $f =(f(v) : v \in 
V_n)$  such that 
\be\label{eq H_n}
||f ||^2_{\mc H_n} := \sum_{ v \in V_n} q^{(n)}_v f(v)^2 < \infty.
\ee
Then $\mc H_n$, equipped with this norm, is a Hilbert space with the 
inner product
$$
\langle \varphi, \psi \rangle_{\mc H_n} = \sum_{v\in V_n} \varphi(v)
\psi(v) q^{(n)}_v.
$$

\begin{proposition}\label{prop T_P T_Q}
(1) The operators $T_{P_n} : \mc H_{n+1} \to \mc H_n$ 
and $T_{Q_n} : \mc H_{n} \to \mc H_{n+1}$ are positive and
 contractive for all $n \in \N_0$. 
 
 (2) $(T_P)^* = T_Q$ and $(T_Q)^* = T_P$.
\end{proposition}

\begin{proof} It is obvious that  $T_{P_n}(f) > 0$ whenever $f >0$. 
By Schwarz inequality, we have $ T_{P_n}(f)^2 \leq T_{P_n}(f^2)$.
 
It follows from Lemma \ref{lem operators T} that the
operator $T_{P_n}$ is defined on vectors from $\mc H_{n+1}$. One
needs to show that $T_{P_n}(f) \in \mc H_n$ if $f\in \mc H_{n+1}$. 
Indeed, using the equality $q^{(n+1)} = q^{(n)}\wh P_n$, we 
compute
$$
\ba 
|| T_{P_n}(f) ||^2_{\mc H_n} &  =  \sum_{ v \in V_n} q^{(n)}_v
T_{P_n}(f)^2 \\
& \leq \sum_{ v \in V_n} q^{(n)}_v T_{P_n}(f^2)\\
& = \sum_{ v \in V_n} q^{(n)}_v \sum_{u \in V_{n+1}} \wh 
p_n(v, u)f^2(u) \\
& = \sum_{u \in V_{n+1}} f^2(u) \sum_{ v \in V_n} q^{(n)}_v \wh 
p_n(v, u)\\
& = \sum_{u \in V_{n+1}} q^{(n+1)}_u f^2(u) \\
& = || f ||^2_{\mc H_{n+1}}. 
\ea
$$
It follows also from the proof that 
$$
|| T_{P_n}||_{\mc H_{n+1 \to \mc H_n}} \leq 1.
$$
The same proof works for $T_{Q_n}$. 

(2) To prove that $(T_P)^* = T_Q$, we use the equality 
$q^{(n)}_v \wh p_n(v,u) = q^{(n+1)}_u\wh q_n(u, v) $, see Theorem
\ref{thm exist of Q_n}. 
$$
\ba 
\langle f, T_{P_n}(g)\rangle_{\mc H_n} & = \sum_{v\in V_n} f(v)
\big(\sum_{u \in V_{n+1}} \wh p_n(v,u)g(u) \big) q^{(n)}_v\\
&= \sum_{v\in V_n} \sum_{u \in V_{n+1}} f(v) g(u) q^{(n)}_v
\wh p_n(v,u)\\
&=  \sum_{u \in V_{n+1}}  \sum_{v\in V_n}f(v) g(u) q^{(n+1)}_u
\wh q_n(u, v)\\
&= \sum_{u \in V_{n+1}} g(u) T_{Q_n}(f)(u) q^{(n+1)}_u\\
&= \langle T_{Q_n}(f), g\rangle_{\mc H_{n+1}}. 
\ea
$$ 
\end{proof}

The following corollary is immediate. 

\begin{corollary}\label{cor operator T_n}
 Let $T_n = T_{P_n}T_{Q_n}$. Then 

(1) $T_n$ is a 
self-adjoint operator on $\mc H_n$ acting on functions from $\mc H_n$
by the formula
$$
(T_n f)(v) = \sum_{w\in V_n} \sum_{u \in V_{n+1}} \wh p_n(v,u)
\wh q_n(u,w)f(w),
$$

(2) $T_n$ is represented by a row stochastic matrix $\wh T_n$ with 
entries 
$$
\wh t^{(n)}_{v w} =\sum_{u\in V_{n+1}} \wh p_n(v,u)\wh q_n(u,w),
$$

(3)  $q^{(n)}$ is a left eigenvector for the matrix $\wh T_n$ 
corresponding to the eigenvalue $1$.
$n \in \N$.
\end{corollary}

\begin{proof}
Statement (1) follows from Proposition \ref{prop T_P T_Q}.
To see that (2) is true, we can use that $\wh P_n$ and $\wh Q_n$ are row
stochastic.
For (3), we compute 
$$
\ba 
\sum_{v\in V_n}  q^{(n)}_v \wh t^{(n)}_{v w} & = 
\sum_{v\in V_n}  q^{(n)}_v \sum_{u\in V_{n+1}} \wh p_n(v,u)\wh 
q_n(u,w)\\
& = \sum_{u\in V_{n+1}} \left(\sum_{v\in V_n} q^{(n)}_v \wh 
p_n(v,u)\right) \wh q_n(u,w)\\
&=\sum_{u\in V_{n+1}} q^{(n+1)}_u \wh  q_n(u,w)\\
& = q^{(n)}_w.
\ea
$$

\end{proof}

\section{ Graph Laplacians and associated harmonic functions}
 \label{sec Laplace}

In this section we consider the main concepts  of the theory of weighted 
networks $(G, c)$ in the case when the graph $G$ is represented by a
 generalized Bratteli diagram. We refer to 
 \cite{BJ-2019} where a similar approach was applied to finite Bratteli 
 diagrams.  The reader can find more information on this subject, for
 example,  in 
\cite{Anandam_2011, Anandam_2012, Cho2014, 
Dutkay_Jorgensen2010, Georgakopoulos2010,
Jorgensen_Pearse2010, Jorgensen_Pearse2013, Kigami2001,
LyonsPeres_2016, Petit2012}. Our  
goal is to show how the notions of weighted networks can be used for
Bratteli diagrams. We plan to study them in details in a forthcoming 
paper.

We use the notation of Section \ref{sect Markov} to fix our setting for
this section: $B = (V, E)$ is a generalized 
Bratteli diagram (see Definition \ref{def GBD}), $(\wh P_n)$ and 
$(\wh Q_n)$ are the sequences of
row stochastic matrices defined by Markov matrices $(P_n)$, 
$(q^{(n)})$ is a sequence of positive vectors where $q^{(n)}$ is 
indexed by vertices of $V_n$ and such that $q^{(n)} \wh P_n =
q^{(n+1)}$, $q^{(n+1)} \wh Q_n =q^{(n)}$.

By definition, a \textit{weighted network} $(G, c)$ is an undirected  
countable connected  graph $G = (V, E)$ without loops together with a
 weight
function $c : E \to [0, \infty)$. In the literature, weighted networks  are 
also called \textit{electrical networks} where $c$ is viewed as a 
conductance function (see Subsection \ref{ssect Literature} for 
references) .   The notation $v \sim w$ means that 
there exists an edge between vertices $v$ and $w$ (we consider the
graphs with single edges between connected vertices). Define the path
space $X_G$ of $G$ as a set of 
all infinite sequences $x= (v_0, v_1, ... )$ such that $e = 
(v_i, v_{i+1})$ is an edge from $E$, $i \in \N_0$. 

We will begin with a generalized Bratteli diagram $B = B(V, E)$ and 
show how to  associate a weighted network $G= G(B)= (V', E')$ to
$B$. It is important to stress that our setting includes the existence 
of Markov matrices $(\wh P_n)$ and $(\wh Q_n)$ defined by 
row stochastic matrices $(P_n)$ and the vectors $(q^{(n)}$ (see their
properties above). 

We set $V' = V= \bigcup_n V_n$ so that all vertices are partitioned into 
levels. Define the set of edges $E'$ by the following rule: $E' = 
\bigcup_n E'_n$ where $E'_n$ is the set of edges between vertices 
of $V_n $ and $V_{n+1}$.  Two vertices, $v \in V_n$ and $u \in 
V_{n+1}$, are connected  by an edge from $E'$ if and only if the set
$E(v, u)\neq \emptyset$ in the Bratteli diagram $B$.
In other words, we replace  the set $E(v, u)$ with a singe edge if
this set is not empty.  
 
Define now weight function $c : E' \to [0, \infty)$. Let $e' = (v, u)$
where $v\in V_n, u \in V_{n+1}$. For this, we fix a vertex $v \in V_n$
and assign a weight $c^{(n)}_{vu}$ for all edges $(v,u)$, $u \in 
V_{n+1}$ and $c^{(n-1)}_{vw}$ for all edges $(w,v)$, $w \in 
V_{n-1}$:
\be\label{eq cond fn}
c^{(n)}_{vu} = \frac{1}{2} q^{(n)}_v \wh p_n(v,u),\quad 
c^{(n-1)}_{vw} = \frac{1}{2} q^{(n)}_v \wh q_{n-1}(v,w). 
\ee

\begin{lemma}\label{lem c(v)}
(1) $c^{(n)}_{vu} = c^{(n)}_{uv}$, i.e., the conductance function 
$c$  is correctly defined on edges from $E'$.

(2) 
$$
c_n(v) = q^{(n)}_v,\ \ \ v\in V_n, \ n\in \N_0.
$$
\end{lemma}

\begin{proof}
(1) Formula \eqref{eq cond fn} defines the value of the function $c$ 
on an edge connecting $w \in V_{n-1} $ and $v \in V_n$ in  two 
(formally different) ways.  In fact, we use \eqref{eq wh p and wh q} to
show that  
$$
c^{(n-1)}_{wv} =  \frac{1}{2} q^{(n-1)}_w \wh p_{n-1}(w,v) 
= \frac{1}{2} q^{(n)}_v \wh q_{n}(v,w) = c^{(n-1)}_{vw}. 
$$

(2) We calculate the sum using statement (1):
$$
\ba 
c_n(v) & = \sum_{u \in V_{n+1}} c^{(n)}_{vu}  +
\sum_{w\in V_{n-1}} c^{(n-1)}_{vw}\\
& = \frac{1}{2} \sum_{u \in V_{n+1}} q^{(n)}_v \wh 
p_{n}(v,u) +  \frac{1}{2} \sum_{w\in V_{n-1}} q^{(n)}_v \wh
 q_{n-1}(v,w)\\
 & = q^{(n)}_v.
\ea
$$
\end{proof}

Lemma \ref{lem c(v)} shows that $c(v) = (c_n(v))$ is finite for every 
$v \in V$. We will omit the index $n$ in $c_n(v)$ if it is clear that 
$v $ is taken from $V_n$.

\begin{definition}\label{def Markov krnl} 
For $G = G(B)$ and the conductance function $c$ as
above, we define a \textit{reversible Markov kernel} $M = \{m(v,u) : 
v, u \in V\}$ by setting
$$
m(v, u)=
\begin{cases}
\dfrac{c^{(n)}_{vu}}{c_{n}(v)} = \frac{1}{2} \wh p_n(v, u), 
& v \in V_n, u\in V_{n+1},\\
\\
\dfrac{c^{(n-1)}_{uv}}{c_{n}(v)} =\frac{1}{2} \wh q_{n-1}(v, u), 
& v \in V_n, u\in V_{n- 1}.\\
\end{cases}
$$
\end{definition}

\begin{remark}
It follows from Lemma \ref{lem c(v)} that $m(v,u)$ can be viewed as 
the probability to get to $u$ from $v$ because $\sum_{u \sim v}
m(v, u) =1$. 

The Markov kernel $M$ is \textit{reversible}, i.e., 
$$
c(v) m(v, u) = c(u) m(u, v), \quad \forall v, u,\ v\sim u. 
$$ 
This fact follows from \eqref{eq wh p and wh q}.
\end{remark}

\begin{definition}
Let $f = (f_n)$ be a function defined on vertices of the graph $G(B)$
(or the Bratteli diagram). Suppose that a Markov kernel $M$ is defined as 
in \ref{def Markov krnl}.  The operator 
$$
(Mf)(v) = \sum_{u\sim v} m(v,u)f(u), \ \ v \in V.
$$
is called a \textit{Markov operator} acting on the weighted network 
$(G,c)$.

Define a \textit{Laplacian operator} $\Delta$:
$$
(\Delta f)(v) = \sum_{u\sim v} c_{vu}(f(v) - f(u)) = c(v)[f(v) - 
(Mf)(v)], \ \ \ v \in V. 
$$

A function $f$ is called \textit{harmonic}, if $M(f) = f$. Equivalently, 
$f$ is harmonic if $\Delta f = 0$. 
\end{definition}

\begin{theorem}\label{thm harmonic}
Let $f = (f_n)$ be a function on the vertex set of $G(B)$. Then $f$ is
harmonic if and only if 
$$
2 f_n = \sum_{u\in V_{n+1}} \wh p_n(v, u) f_{n+1}(u)
+ \sum_{w\in V_{n-1}} \wh q_{n-1}(v, w) f_{n-1}(w), \ \ \forall 
n\in \N,
$$
or, equivalently, $2f_n = \wh P_n f_{n+1} + \wh Q_{n-1} f_{n-1}$.

Similarly, $f$ is harmonic if and only if $D_n f_n = \wh P_n f_{n+1} + 
\wh Q_{n-1} f_{n-1}$ where $D_n$ is the diagonal matrix with entries 
$2q^{(n)}_v, v \in V_n,$ on the main diagonal, $n \geq 1$.

In particular, 
\be\label{eq q M}
q^{(n)} M = \frac{1}{2}( q^{(n+1)} + q^{(n-1)}), \ \ \ n \in \N.
\ee
\end{theorem}

\begin{proof} It follows from Definition \ref{def Markov krnl} that
the action of a Markov operator $M$ on a function $f$ can be 
represented in  the following form:
$$
(Mf_n)(v) = \frac{1}{2}\left(\sum_{u\in V_{n+1}} \wh p_n(v, u) 
f_{n+1}(u)+ \sum_{w\in V_{n-1}} \wh q_{n-1}(v, w) f_{n-1}(w)
 \right).
$$
Then we use Definition \ref{def Markov krnl} to deduce the first 
statement.

For the Laplacian operator $\Delta = c(Id - M)$, we obtain that 
$\Delta (f) =0$ if and only if 
$$
c_n(v)(f_n(v) - \frac{1}{2}\left(\wh P_n (f_{n+1})(v) + \wh Q_{n-1} 
(f_{n-1})(v)\right) =0.
$$
By Lemma \ref{lem c(v)} we have $2c_n(v) f_n(v) = (Df_n)_v$ and the
result follows.

To prove \eqref{eq q M}, we note that $ q^{(n+1)} = q^{(n)} \wh
P_n$ and $q^{(n-1)}   = q^{(n)} \wh Q_{n-1}$.
\end{proof}

Our next focus is on the  \textit{path-space} of a weighted network 
$G(B)$ defined by a generalized Bratteli diagram $B$ and measures on
this space. We recall a few well known notions related to \textit{random walks on graphs}.  Let $\Omega$ be a subset of $V_0 \times V_1 \times \cdots$
of infinite paths $\omega = (v_i)_{i \in \N_0}$ such that $(v_i, 
v_{i+1}) $ is an edge for every $i$. By $\Omega_{v}$ we denote the
subset of $\Omega$ formed by all paths beginning at $v \in V_0$. 
Let $\xi_n$ be a random variable $\xi_n : \Omega \to V_n$ defined on
 the  path-space $\Omega$: for $\omega = (v_i)$ we set $\xi_n(x) 
 = v_n$, $n \in \N_0$.
 
The Markov kernel $M$ defines a probability measure $m_v$ on cylinder
sets of $\Omega_v$ for every $v\in V_0$ as follows: 
$$
m_v(\omega : \xi_1(\omega) = v_1, ... , \xi_n(\omega) = v_n \ |\ 
\xi_0 = v)  = m(v, v_1)\cdots m(v_{n-1}, v_n).
$$
Then it is extended to all Borel sets. The sequence of random variables
$(\xi_n)$ defines a Markov chain on $(\Omega_v, m_v)$ such that
$m_v(\xi_{n+1} = y\ |\ \xi_n = z) = m(z, y)$. 
Let $\lambda = (\lambda_v)$ be a positive probability vector on $V_0$.
If $\lambda M = \lambda$, then $\mathbb P = \sum_{v} \lambda_v
m_v$ is a Markov measure on $\Omega$.

We will assume that the transition probability kernel $M$ is 
\textit{irreducible.} 
The  kernel $M$ generates the random walk on the graph 
$G(B)$. It is said that $M$ is \textit{recurrent} if for every  vertex 
$v \in V$ the random walk returns to $v$ infinitely often with probability 
one. Otherwise it is called \emph{transient}.

\begin{remark}
We point out the difference between the path-spaces of a generalized 
Bratteli diagram $B$ and that of the graph $G(B)$. For a Bratteli 
diagram $B$, the path-space $X_B$ is formed by concatenation of 
consecutive edges $(e_0, e_1, ...)$ such that $e_i$ is an edge between 
the levels $V_i$ and $V_{i+1}$. The path-space of $G(B)$ is formed 
by the sequences of vertices of $B$, moreover a vertex $v_i$ is not 
necessarily  from $V_i$ because the transition probability kernel $M$ is
defined on both \textit{incoming} and \textit{outgoing} edges.
\end{remark}

Our next small topic is the \textit{finite energy space} $\mc H_E$ for a 
weighted network $(G(V,E), c)$. Consider
functions on vertices $V$ of the graph $G$. It is said that 
two functions $f$ and $g$ are equivalent if $f -g = \mathrm{const}$.

\begin{definition}
Define the \textit{finite energy space} $\mc H_E$ as the set of equivalence 
classes of functions $f$ on $V$ such that
\be\label{eq-energy norm} 
|| f ||^2_{\mc H_E} := \frac{1}{2}\sum_{(v, u)\in E} c_{vu} 
(f(v) - f(u))^2 < \infty. 
\ee
The space $\mc H_E$ equipped with this norm is a Hilbert space.
\end{definition}

We refer to the papers \cite{Jorgensen2012, JorgensenPearse_2016,
BezuglyiJorgensen_2019} where the reader will find more details about
the finite energy space.

Now we adapt this definition to the case of weighted networks $G(B)$
 defined by a generalized Bratteli diagram $B$. We note that the 
coefficient $1/2$ in \eqref{eq-energy norm} is used because the 
contribution of each edge is counted twice. For a generalized Bratteli 
diagram, we have all edges partitioned in the levels, so that we do not
need this coefficient any more. 

Hence, we can rewrite  \eqref{eq-energy norm} as
\be\label{eq norm in energy}
\ba 
|| f ||^2_{\mc H_E} & = \sum_{n \in \N_0}\ \ 
\sum_{v \in V_n, u \in V_{n+1}} c^{(n)}_{vu}(f_n(v) - 
f_{n+1}(u))^2\\
& = \frac{1}{2}\sum_{n \in \N_0}\ \ \sum_{v \in V_n, u \in V_{n+1}} 
q^{(n)}_{v} \wh p_n(v, u)(f_n(v) - f_{n+1}(u))^2\\
\ea
\ee
where $f = (f_n)$ is a function from $\mc H_E$.

\begin{proposition}\label{prop finite energy}
Suppose that a function $f = (f_n)$ is such that every 
$f_n$  belongs to $ \mc H_n, n \geq 0,$ where the Hilbert space 
$\mc H_n$ is  defined in \eqref{eq H_n}. Then $ f\in \mc H_E$ if and
 only if
$$
\sum_{n \geq 0} \left(||f_n||^2_{\mc H_n} - 2\langle f_n, T_{\wh P_n}
(f_{n+1})\rangle_{\mc H_n} + ||f_{n+1}||^2_{\mc H_{n+1}}
\right) < \infty
$$
where the operator $T_{\wh P_n} : \mc H_{n+1} \to \mc H_n$ is 
defined in Proposition \ref{prop T_P T_Q}. 
\end{proposition}

\begin{proof}
We apply formula \eqref{eq norm in energy} and compute the norm
of $f$:

$$
\ba 
|| f ||^2_{\mc H_E} & =   \frac{1}{2}\sum_{n \in \N_0}\ \ 
\sum_{v \in V_n, u \in V_{n+1}} 
q^{(n)}_{v} \wh p_n(v, u)(f_n(v) - f_{n+1}(u))^2\\
& = \frac{1}{2}\sum_{n \in \N_0}\  
\sum_{v \in V_n} \ \sum_{u \in V_{n+1}} 
q^{(n)}_{v} \wh p_n(v, u)[f_n(v)^2 -  2f_n(v) f_{n+1}(u)
+ f_{n+1}(u)^2]\\
& = \frac{1}{2}\sum_{n \in \N_0}\  \sum_{v \in V_n} \left(
q^{(n)}_{v} f_n(v)^2  -  2 q^{(n)}_{v} f_n(v) T_{\wh P_n} 
(f_{n+1})(v) + \sum_{u \in V_{n+1}} q^{(n)}_{v} \wh p_n(v, u)
 f_{n+1}(u)^2\right)\\
& = \frac{1}{2}\sum_{n \in \N_0} \left((||f_n||^2_{\mc H_n} - 
2\langle f_n, T_{\wh P_n}(f_{n+1})\rangle_{\mc H_n} + 
\sum_{u\in V_{n+1}} q^{(n+1)}_{u} f_{n+1}(u)^2\  \right) \\
&=  \frac{1}{2}\sum_{n \geq 0} \left(||f_n||^2_{\mc H_n} - 
2\langle f_n, T_{\wh P_n}(f_{n+1})\rangle_{\mc H_n} + 
||f_{n+1}||^2_{\mc H_{n+1}}\right)
\ea
$$

\end{proof}

\section{Graph transition kernels and measures}
\label{sect Trans kernels}

In this section, we discuss the interplay between 
transition (probability) kernels and the corresponding 
measures.  
While our focus below is on the operator theory of transition kernels, we 
stress that their application (in subsequent sections) will be to our study of 
the class of \textit{graph dynamical systems }which are specified by transition between levels in the graph.

\subsection{Definitions of transition  kernels and associated
measures}
\label{ssect kernels and measures}

We first  recall the definition of a transition  kernel and 
introduce linear operators generated by such a kernel. The reader can
find the references in Subsection \ref{ssect Literature}.

Let $(X_i, \A_i), i =1,2,$ be two standard (uncountable) Borel spaces. 
Without loss of generality, one can assume that these spaces  are two
 copies of the same standard Borel space $(X, \A)$. 
 
\begin{definition}\label{def tran pr kern}
A map $R : X_1
\times \A_2 \to [0, \infty)$ is called a \textit{transition  kernel} if
it satisfies the following conditions: 

(i) for every set $C \in \A_2$, the function $x \mapsto R(x, C), x \in 
X_1$, is Borel;

(ii) for every $x \in X_1$, the
map $C \mapsto R(x, C)$ is a $\sigma$-finite Borel measure.

If $R(x, \cdot)$ is a finite measure for every $x \in X_1$, i.e., $0 < R(x,
 X_2) < \infty$, then the kernel $R$ is called \textit{finite}.
If the measure $R(x, \cdot)$ is probability for every $x \in X_1$, i.e.,
 $R(x, X_2) = 1$, then  $R$ is called a \textit{transition probability 
 kernel}.
\end{definition}

Denote by $\mc F(X, \A)= \mc F(X)$ the linear space of bounded Borel
 functions and
by $M(X, \A)= M(X)$ the set of Borel positive $\sigma$-finite measures. 
We  define 
actions of $R$ on these sets (with some abuse of notation, we will use 
the same letter $R$ for these actions).  
For $f \in \mc F(X_2, \A_2)$,  we set
\be\label{eq R acts}
(Rf)(x) = \int_{X_2} R(x, dy) \; f(y).  
\ee
Then, relation \eqref{eq R acts} determines a positive linear operator 
$T_R : \mc F(X_2, \A_2) \to \mc F(X_1, \A_1)$. Applying 
\eqref{eq R acts} to  characteristic functions, we obtain 
\be\label{eq R and chi_A}
R(x, A) = R(\chi_A)(x),\quad A\in \A.
\ee
In particular, $R$ is a probability kernel if $R(\mathbbm 1) =1$ where 
$\mathbbm 1$ is a constant function taking the value $1$.

In contrast to \eqref{eq R acts}, the action of $R$ on measures 
generates a map from  $M(X_1, \A_1)$ to $M(X_2, \A_2)$:
\be\label{eq_R on meas}
(\mu R)(A)  =  \int_{X_1} d\mu(x) R(x, A) = \int_{X_1} d\mu(x_1)
\left(\int_A R(x, dy)\right), 
\ \ \mu \in  M(X_1, \A_1). 
\ee

Writing $R(f)$ and $\mu R$ as in \eqref{eq R acts} and 
\eqref{eq_R on meas}, we stress on the similarity with multiplication
of a matrix and  row and column  vectors.

In the following remark, we provide several examples of transition 
probability kernels. 

\begin{remark}\label{rem 1} (1) Let $(X_1, \A_i, \nu_i), i =1,2,$ 
be standard probability measure spaces. Then, defining $R_1(x, B) = 
\nu_2(B), B \in \A_2$, we obtain a constant transition probability
kernel. We will see below that this kernel determines the product
 measure $\rho= \nu_1 \times  \nu_2$ on $X_1\times X_2$.

(2) For standard Borel spaces $(X_i, \A_i)$, take $\nu_1 =
 \delta_{x_0}, x_0 \in X_1,$ the Dirac measure on 
$(X_1, \A_1)$. Then, setting $R_0(x_0, \cdot) = \nu_2$, where $\nu_2$
is a $\sigma$-finite (or probability)  measure on $(X_2, \A_2)$, we obtain a
$\sigma$-finite (or probability) transition kernel $R_0 : X_1 \times \A_2 
\to [0, 1]$. 

(3) Let $T$ be a positive operator acting on the set of
bounded Borel functions $\mc F(X, \A)$ on a standard Borel space $(X,
\A)$. This means that $T(f) \geq 0$ whenever $f \geq 0$. 
It is said that $T$ has the \textit{Riesz property} if, for every $x \in X$,
 there  exists a Borel measure $\mu_x$ such that 
$$
T(f)(x) = \int_X f(y) \; d\mu_x(y), \quad f \in \mc F(X, \A).
$$ 
The set $(\mu_x : x \in X)$ is called a \textit{Riesz family} of measures
corresponding to $T$.

If $T$ is normalized ($T(\mathbbm 1) = 1$), then every measure 
$\mu_x$ is probability. Observe that the field of measures $x \mapsto
\mu_x$ is Borel in the sense that the function $x\mapsto \mu_x(f)$ is
Borel for every $f \in \mc F(X, \A)$. As a conclusion, we see that 
the Riesz family $(\mu_x)$ determines a transition probability kernel
$R =R_T : X \times \A \to [0, 1]$ by setting $R(x, B) = \mu_x(B)$, 
$B \in \A$.  
\end{remark}

In what follows, we will consider an interaction between transition 
kernels and measures. 

Let $(X, \A, \mu)$ be  a $\sigma$-finite measures space. Denote by 
$\mc D(\mu)$ the collection of Borel sets $C \in \A$ such that 
$\mu(C) < \infty$. Then $\mc D(\mu)$ generates the $\sigma$-algebra
of Borel sets $\A$. By $\mc F(\mu)$ we denote the linear space of
functions $\varphi $ spanned by characteristic functions of the sets 
from $\mc D(\mu)$, i.e.,  
$$
\varphi = \sum_{i=1}^n \alpha_i \chi_{A_i}, \ \mu(A_i) < \infty, 
\quad \alpha_i \in \R.
$$

\begin{definition}\label{def P,Q first} 
For given $\sigma$-finite measure spaces $(X_i, \A_i, \nu_i), i=1,2$, 
 suppose that there are  transition kernels 
$$
P: X_1 \times \A_2 \to [0, \infty)\ \ \mbox{and} \ \  
Q: X_2 \times \A_1 \to [0, \infty)
$$ 
such that, for any bounded Borel function $f$,
$$
\int_{X_1 \times X_2} d\nu_1(x) P(x, dy)\; f(x, y) =
\int_{X_1 \times X_2} d\nu_2(y) Q(y, dx)\; f(x,y),
$$
or in a short form,
\be\label{eq-meas and kernels}
d\nu_1(x) P(x, dy) = d\nu_2(y) Q(y, dx),\ \  \ (x, y) \in  X_1 \times
X_2.
\ee
Then these kernels $P$ and $Q$ are called \textit{associated to the 
measures} $\nu_1$ and $\nu_2$. We will also say that $P$ and $Q$ 
satisfying \eqref{eq-meas and kernels} form a \textit{dual pair} of 
transition kernels, see Figure \eqref{fig:kernels} for illustration. 
\end{definition}

For $f(x, y) = \chi_A(x) \chi_B(y)$,  relation \eqref{eq-meas and kernels}
 defines a Borel $\sigma$-finite measure $\rho$ on $(X_1 \times X_2, 
 \A_1 \times \A_2)$ by the formula
\be\label{eq_rho first def}
\rho(A\times B) = \int_A d\nu_1(x) P(x, B) = \int_B d\nu_2(y) Q(y, A)
\ee
where $A \in \A_1, B \in \A_2$. More precisely, the measure $\rho$ is 
defined by \eqref{eq-meas and kernels} for Borel sets of finite measure,
i.e., $A\in \mc D(\nu_1), B\in \mc D(\nu_2)$, 
and then it is extended to all Borel sets. Obviously, the measure $\rho$ 
is probability if the kernel $P$  (or $Q$) and 
measure $\nu_1$  (or $\nu_2$) are probability. The converse is not true, in 
general.

\begin{figure}[!htb]
\centering
  \includegraphics[width=0.55\textwidth, height=0.35\textheight]
  {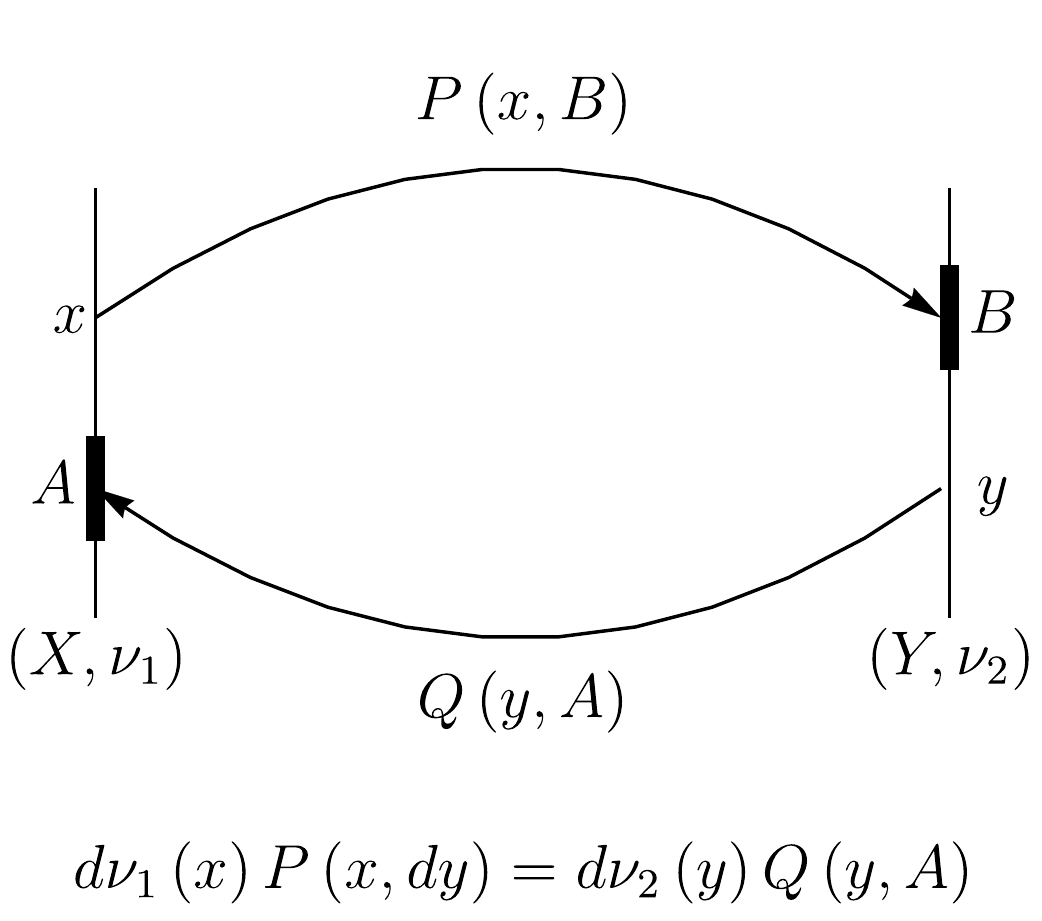}
  \caption{Kernels P and Q in duality, see Definition \ref{def P,Q first}}
  \label{fig:kernels}
\end{figure}

\begin{remark}\label{rem c_i}
(1) If $P(x, B) = \nu_2(B)$ and $Q(y, A) =\nu_1(A)$ are constant 
transition kernels (see Remark \ref{rem 1} (1)), then the corresponding
measure $\rho$ is the product measure $\nu_1 \times \nu_2$. 

(2) Suppose that the kernels $P$ and $Q$ are finite. 
Let $c_1(x) = P(x, X_2)$ and $c_2(y) = Q(y, X_1)$. The functions
$c_1, c_2$ are Borel and take finite values for all $x$ and $y$. In 
general, they may be unbounded. In a number of statements below, we 
will assume that these function are locally integrable, i.e.,
$$
c_i \in L^1_{\mathrm loc}(\nu_i) \Longleftrightarrow 
\int_B c_i(\cdot) \; d\nu_i(\cdot) < \infty,\ \ \ B \in \mc D(\nu_i),
\ \ i =1,2. 
$$

(3) In case of need,  we will assume the property $c_i \in 
L^2_{\mathrm loc}(\nu_i)$. This requirement will be used for the study
of unbounded operators in $L^2$-spaces. 
\end{remark}

Consider the product space $Z := X_1 \times X_2$ equipped 
with the product Borel structure $\mc C$. Denote by $\pi_i $ the natural 
projection from $(Z, \mc C, \rho)$ onto $(X_i, \A_i, \nu_i)$. 
Let $E\subset Z$ be the essential support of the measure $\rho$. This
 set is defined up to a null set. There are  two measurable 
partitions of $E$, $\xi_1$ and $\xi_2$, where $\xi_1$ is formed
by the subsets $E_x := \{(x, y) \in E\}, x \in X_1,$ and $\xi_2$ consists 
of subsets
$E^y := \{(x,y)\in E\}, y \in X_2$ (measurable partitions are discussed 
in many books and articles on descriptive set theory and ergodic theory, 
e.g., \cite{Rohlin1949, Kechris1995, CornfeldFominSinai1982}).
 If needed, we will identify the sections 
 $E_x$ and $E^y$ with the corresponding projections 
$\pi_1(E^y)$ and $\pi_2(E_x)$ onto $X_1$ and $X_2$.

The next result follows directly from Definition \ref{def P,Q first} and 
Remark \ref{rem c_i}.

\begin{lemma}\label{lem loc int}
 Let the kernels $P$ and $Q$ be as Definition 
\ref{def P,Q first} with locally integrable functions $c_1(x) = P(x, X_2)$
 and $c_2(y) = Q(y, X_1)$. Then 
$$
\rho(A \times X_2) = \rho\circ \pi_1^{-1}(A) < \infty, \ \ \ \ \
\rho(X_1 \times B) = \rho\circ \pi_2^{-1}(B) < \infty
$$ 
for all $A \in \mc D(\nu_1), B \in \mc D(\nu_2)$. In particular, 
 $\rho(A \times B) < \infty $ for all $A \in \mc D(\nu_1), B \in \mc 
D(\nu_2)$. 

Moreover, the kernels $P$ and $Q$ define  linear operators
$T_P$ and $T_Q$ such that 
\be\label{eq-T_P loc int}
T_P : \mc F(\nu_2) \ \to \ L^1_{\mathrm loc}(\nu_1),
\quad T_Q : \mc F(\nu_1) \ \to \ L^1_{\mathrm loc}(\nu_2).
\ee
\end{lemma}

\begin{proof}
Indeed, if $A \in \mc D(\nu_1)$, then
$$
\rho\circ \pi_1(A) =\rho(A \times X_2) = \int_A c_1(x) \; d\nu_1(x) 
< \infty.
$$
Similarly, we have $\rho\circ \pi_2(B) < \infty$ for $B\in 
\mc D(\nu_2)$. 

For the second statement, it suffices to note that  the condition 
$\rho(A \times B) < \infty$, where $A \in \mc D(\nu_1)$, $B \in \mc 
D(\nu_2)$, implies that 
$P(\chi_B)(x) = P(x, B) \in  L^1_{\mathrm loc}(\nu_1)$ and 
$Q(\chi_A)(y) = Q(y, B) \in  L^1_{\mathrm loc}(\nu_2)$. These properties 
are extended to $\mc F(\nu_i)$ by linearity. 
\end{proof}

We remark that \eqref{eq-T_P loc int} can be extended by continuity 
to the set of  functions $g$ such that $supp g \subset A$ where $A$ has 
a finite measure. 

\begin{lemma} In conditions of Lemma \ref{lem loc int}, 
$$
T_P : L^1(c_2\nu_2) \to L^1(\nu_1),\ \ \ \  \
T_Q : L^1(c_1\nu_1) \to L^1(\nu_2).
$$

\end{lemma}

\begin{proof}
For the proof we use \eqref{eq-meas and kernels} and 
\eqref{eq_rho first def}: let $f \in L^1(c_2\nu_2)$, then  
$$
\ba 
\int_{X_1} |T_P(f)|(x)\; d\nu_1(x) & = \int_{X_1} \left| \int_{X_2} 
P(x, dy) f(y) \right| \; d\nu_1(x)\\
& \leq  \int_{X_1} \int_{X_2}  |f(y)|\; P(x, dy) d\nu_1(x)\\
& = \int_{X_1} \int_{X_2}  |f(y)|\; Q(y, dx) d\nu_2(y)\\
&= \int_{X_2} |f(y)| Q(y, X_1) \; d\nu_2(y)\\
&\leq \int_{X_2}  |f(y)|c_2(y)\; d\nu_2(y)
\ea
$$
A similar proof works for $T_Q$.
\end{proof}

We recall the following definition and result proved in 
\cite{Simmons2012} (we give here a modified statement adapted to our 
purposes). 

\begin{definition} \label{def cond meas} 
Let $(Z, \mc C, \mu)$ and $(Y, \mc D, \nu)$ be standard 
$\sigma$-finite measure spaces, and let $\pi : Z \to Y$ be a measurable
 function. A system of \textit{conditional measures} of $\mu$ 
with respect to $\pi$ is a collection of measures $\{\mu_y : y \in Y\},$
such that \\
(i) $\mu_y$ is a Borel measure on $\pi^{-1}(y)$,\\
(ii) for every $B \in \mc C$,  $\mu(B) = \int_Y \mu_y(B) \; d\nu(y)$,
i.e., $\mu$ is disintegrated by the conditional measures. 
\end{definition}

\begin{theorem}[\cite{Simmons2012}] \label{thm Simmons}
Let $(Z, \mc C, \mu)$ and $(Y, 
\mc D, \nu)$ be as above. For any measurable function  $\pi : Z \to Y$ 
such that $\mu\circ \pi^{-1} \ll \nu$ there exists a uniquely determined
system of conditional measures $(\mu_y)_{y \in Y}$ which 
disintegrates the measure $\mu$.
\end{theorem}

Under conditions of Theorem \ref{thm Simmons}, we will write down
$$
\mu = \int_{Y} \mu_y d\nu(y).
$$

In particular,  Theorem \ref{thm Simmons} determines a system of 
conditional measures for 
any  measure $\rho$ which is defined as in \eqref{eq_rho first def}.

\begin{corollary}\label{cor disint}
 Let $(X,\A_i, \nu_i), i=1,2,$ be $\sigma$-finite standard
measure spaces, and let $\rho$ be defined by \eqref{eq_rho first def}.
Then $\rho\circ\pi^{-1}_i \ll \nu_i$ and the measure $\rho$ admits its
disintegration with respect to $\pi_1$ and $\pi_2$: the families 
$(\delta_x \times P(x, \cdot))_{x \in X_1}$ and $(\delta_y \times 
Q(y, \cdot))_{y \in X_2}$ are 
the corresponding systems of conditional measures.
\end{corollary}

We observe that the measures $P(x, \cdot)$ and $Q(y, \cdot)$ are 
defined on $X_2$ and $X_1$, respectively. So that we need to identify 
the spaces $\{x\} \times X_2$ and $X_1 \times \{y\}$ with $X_2$ and 
$X_1$ if we want to treat these families of  measures as conditional 
measures.

\begin{proof}
We first note that if a measure $\rho$ is defined as in  
\eqref{eq_rho first def}, then the condition $\nu_i(C) =0$ implies that
$\rho\circ\pi^{-1}_i(C) =0$ for $i = 1,2$. Hence, we can apply 
Theorem \ref{thm Simmons}. It claims the disintegrating system of
conditional measures is unique and therefore $P(x, dy)$ is a measure
on $\pi_1^{-1}(x)$ and $Q(y, dx)$ is a measure on $\pi_2^{-1}(y)$. 
Since $P$ and $Q$ are transition kernels, they satisfy the properties
(i) and (ii) of Definition \ref{def cond meas}. 
\end{proof}

It follows from Corollary \ref{cor disint} that the measure $\rho$
on the product space $X_1\times X_2$ admits disintegrations 
$(\rho_x)$ and $(\rho^y)$ with respect to the projection maps
$\pi_1 : Z \to X_1$ and $\pi_2 : Z \to X_2$:
$$
\rho = \int_{X_1}\rho_x\; d\nu_1(x) = \int_{X_2}\rho^y \; d\nu_2(y).
$$

\subsection{Three sets}\label{ssect three sets}

In the next definition we discuss relations between transition kernels
$P, Q$ and measures $\nu_1, \nu_2$, and $\rho$.

\begin{definition}\label{def_three sets}
  Let $(X_i, \A_i), i=1,2, $ be standard Borel spaces.

(i) For a Borel $\sigma$-finite measure $\rho$ on $(X_1 \times X_2,
 \A_1 \times \A_2)$, we define the set 
 $$
 \ba 
 F(\rho) = \{(\nu_1, \nu_2) : & \ \mathrm{there\ are\ finite\ 
  transition \
 kernels}\ P\ \mbox{and}\ Q \ \mathrm{such \ that} \ \\
 & d\rho(x, y) =d\nu_1(x) P(x, dy) = d\nu_2(y) Q(y, dx)\}.
 \ea
 $$

(ii) If $\nu_1$ and  $\nu_2$ are $\sigma$-finite measures 
on Borel spaces $(X_1, \A_1)$ and  $ X_2,  \A_2)$, respectively, then 
define the set 
$$
\ba
L(\nu_1, \nu_2) = \{\rho : &\ \mathrm{there\ are\ finite\ transition \
 kernels}\ P\ \mbox{and}\ Q \ \mathrm{such \ that} \ \\
& \ \mathrm{\eqref{eq_rho first def} \ holds}\}
 \ea
$$ 
where $\rho$ is a measure on the product space $Z = X_1\times X_2$.

(iii) If $P$ and $Q$ are two transition kernels as in Definition 
\ref{def tran pr kern}, then we set
$$
M(P, Q) = \{(\nu_1, \nu_2) : \mathrm{such \ that} \ 
d\nu_1(x) P(x, dy) = d\nu_2(y) Q(y, dx)\}
$$ 
\end{definition}

One of our purposes is to characterize the elements of these sets. It will 
be done in this and next sections.

 A few obvious facts are included in the following remark.

\begin{remark}
(1) For a given measure $\rho$ on $Z =X_1 \times X_2$, every
pair $(\nu_1, \nu_2) \in F(\rho)$ generates a subset
$F_{\nu_1,\nu_2}(\rho)$ of  $F(\rho)$ whose 
elements $(\nu'_1, \nu'_2)$ are the measures absolutely continuous
with respect to $\nu_1$ and $\nu_2$ (see also Theorem 
\ref{thm meas pairs} below). 

It is important to observe that if one additionally requires that the kernels 
$P$ and $Q$ in the definition of $F(\rho)$ must be finite, then these 
kernels are determined uniquely by the measures $\nu_1$ and $\nu_2$.
This fact follows, in general,  from
Theorem \ref{thm Simmons} (though it can be proved directly).

(2) It follows from Theorem \ref{thm Simmons}  that 
the set $F(\rho)$ is not empty for any measure $\rho$ as the pair
 $(\rho\circ\pi_1^{-1}, \rho\circ\pi_2^{-1})$ is  always in $F(\rho)$. 

(3) On the other hand, in the definition of the set  $L(\nu_1, \nu_2)$
one can take different pairs of kernels $P, Q$ and $P', Q'$ such that the 
corresponding measures $\rho$ and $\rho'$ are in the same 
set $L(\nu_1, \nu_2)$.
Moreover,  these pairs of kernels can be chosen such that $\rho$ and 
$\rho'$ are mutually singular measure. 

(4) It follows from the definitions of the sets
$F(\rho)$ and $L(\nu_1, \nu_2)$ that if $(\mu_1, \mu_2)$ is in 
$F(\rho)$,
then the set $L(\mu_1, \mu_2)$ will contain the measure $\rho$.
 Conversely, if $\rho \in L(\nu_1, \nu_2)$, the $(\nu_1, \nu_2) \in
  F(\rho)$.
  
(5) The following   \textit{question} is interesting: 
When does the set $M(P,Q)$ contain a pair of positive measures?

It can be shown that, even in the case of matrices, there are matrices  
$P$ and $Q$ such that $M(P, Q)$ contains only the pair $(0, 0)$. 
For this, one can take $P = \begin{pmatrix}
1 & 0\\
0 & 1\\
\end{pmatrix}$ and $P = \begin{pmatrix}
0 & 1\\
1 & 0\\
\end{pmatrix}$.
\end{remark}

Suppose that $(\nu'_1, \nu'_2)$ is  a pair of measures on $X_1$ and
 $X_2$
that defines the same measure $\rho$ as the pair $(\nu_1, \nu_2)$, 
i.e., $(\nu'_1, \nu'_2) \in F(\rho)$. What can be said about relations 
between these two
pairs $(\nu_1, \nu_2)$ and $(\nu'_1, \nu'_2)$? 

\begin{theorem}\label{thm meas pairs}
Suppose that a measure $\rho$ is defined on $(X_1, \A_1, \nu_1)
 \times (X_2, \A_2, \nu_2)$ as in \eqref{eq_rho first def}. 
 Suppose that some  measures
$\nu'_1$ and $\nu'_2$ form a pair from $F(\rho)$. Then the 
corresponding measures  $\nu_1, \nu'_1$ and $\nu_2, \nu'_2$  are 
pairwise equivalent, i.e., there are positive functions $f_1$ and  $f_2$ such that
$$
d\nu'_i = f_1 d\nu_i, \ \ i =1,2.
$$
\end{theorem}

\begin{proof}
Since $(\nu_1, \nu_2)$ and $(\nu'_1, \nu'_2)$ are in $F(\rho)$, there are finite transition kernels  $P, Q$ and $P', Q'$, respectively, such that 
\eqref{eq-meas and kernels} holds. Then 
$$
\frac{d\nu'_1}{d\nu_1}(x) = \frac{P(x, dy)}{P'(x, dy)}, \quad
\frac{d\nu'_2}{d\nu_2}(y) = \frac{Q(y, dx)}{Q'(y, dx)}.
$$
We see that the quotients of measures in the right-hand sides of the 
both equations do not depend on $dy$ and $dx$, respectively. Hence 
one can take apply these relations to any sets. 
We consider the positive Borel  functions $c_1(x) = P(x, X_2)$ and 
$c'_1(x) = P'(x, X_2)$. Then relation \eqref{eq_rho first def} leads to
$$
\rho(A\times X_2) = \rho\circ \pi^{-1}_1(A) = 
\int_A c_1(x) \; d\nu_1(x)
$$
or 
$$
c_1(x)=  \frac{d\rho\circ\pi_1^{-1}}{d\nu_1}(x).
$$
In the same way we  obtain similar relations for $\nu'_1$ and $P'$ and
deduce that 
$$
c'_1(x)=  \frac{d\rho\circ\pi_1^{-1}}{d\nu'_1}(x).
$$
Hence, setting $f_1(x) = c_1(x) (c'_1(x))^{-1}$, we see that
$f_1(x)$ is the Radon-Nikodym derivative for the measures $\nu'_1$ and
$\nu_1$. 

Analogously, we can define $c_2(y) = Q(y, X_1)$, $c'_2(y) = 
Q'(y, X_1)$ and  show that 
$$
c_2(y)=  \frac{d\rho\circ\pi_2^{-1}}{d\nu_2}(y).
$$
Therefore 
$f_2(y) = c_2(y) (c'_2(y))^{-1}$ is the 
Radon-Nikodym derivative for $\nu_2$ and $\nu'_2$. Also, the above
argument proves that the measures $P(x, \cdot), P'(x, \cdot)$ and 
$Q(y, \cdot), Q'(y, \cdot)$ are equivalent. Moreover, the functions
$f_1(x)$ and $f_2(y)$ give the Radon-Nikodym derivatives of these 
measures. 
This proves the result.
\end{proof}

\begin{corollary}\label{cor prob kernel}
Let $\rho$ be a $\sigma$-finite measure on $(X_1 \times X_2, \A_1
\times \A_2)$, and a pair of measures $(\nu_1, \nu_2)$ is in $F(\rho)$.
Then there exists a pair $(\nu_1', \nu_2') \in F(\rho)$ such that
$\nu_i' \sim \nu_i$, and the
corresponding kernels $P'$ and $Q'$ are probability.
\end{corollary}

\begin{proof}
It follows from  $(\nu_1, \nu_2)\in F(\rho)$ that there exist
finite transition kernels $P(x, \cdot)$ and $Q(y, \cdot)$ with the 
property $d\nu_1(x) P(x, dy) = d\nu_2(y)Q(y, dx)$. For
$c_1(x) = P(x, X_2)$ and $c_2(y) = Q(y, X_1)$, we define
$d\nu_1'(x) = c_1(x) d\nu_1(x)$ and $d\nu_2'(y) = c_2(y) d\nu_2(y)$.
Then $(\nu_1', \nu_2') \in F(\rho)$ where the corresponding kernels 
$$
P'(x, dy) = \frac{1}{c_1(x)}P(x, dy), \ \ \ \ Q'(y, dx) =
\frac{1}{c_2(y)}Q(y, dx)
$$
are probability kernels.
\end{proof}

\begin{lemma}\label{lem_P,Q abs cont}
For given measure spaces $(X_i, \A_i, \nu_i), i=1,2,$  let $\rho$ be a
 measure  from
the set $L(\nu_1, \nu_2)$. Then, for $\nu_1$-a.e. $x\in X_1$, the 
measure $P(x, \cdot)$ is absolutely continuous with respect to $\nu_2$.
Similarly, $Q(y, \cdot) \ll \nu_1$ for $\nu_2$-a.e $y \in X_2$.
\end{lemma}

\begin{proof}
Let $C$ be a subset of $X_2$ with $\nu_2(C) = 0$. Then 
$$
\rho(X_1  \times C) = \int_C Q(y, X_1)\; d\nu_2(y)=  0. 
$$
Therefore $\int_{X_1} P(x, C) \; d\nu_1(x) = 0$. This means that 
$P(x, C) = 0$ for $\nu_1$-a.e. $x\in X_1$, hence $P(x, \cdot) \ll
\nu_2$. 

The other result is proved analogously. 
\end{proof}

For a $\sigma$-finite measure $\rho$ on $X_1 \times X_2$ and a fixed
 set 
$B\in \A_2$ such that $\rho(X_1 \times B) >0$, we define a 
$\sigma$-finite 
measure  $\tau_B$ on $(X_1, \A_1)$ by setting $\tau_B(C) = \rho(C
\times B)$. Similarly, we define a measure $\tau^A$ on $(X_2, \A_2)$
when a set $A \in \A_1$ is fixed. It follows from these definitions that
$\tau_B(A) = \tau^A(B)$.

\begin{theorem}\label{thm  on F(rho)} 
Let $\rho$ be a measure on the product space $(X_1 \times X_2, 
\A_1 \times \A_2)$. Then a pair of measures $(\nu_1, \nu_2)$ defined
on $X_1$ and $X_2$, respectively, is in $F(\rho)$ if and only if 
\be\label{eq tau ll nu}
\rho\circ \pi^{-1}_1 \ll \nu_1,\ \ \rho\circ \pi^{-1}_2 \ll \nu_2.
\ee
\end{theorem}

\begin{proof}
If $(\nu_1, \nu_2) \in F(\rho)$, then there are finite transition kernels
$P$ and $Q$ such that $d\rho(x, y) = d\nu_1(x) P(x, dy) = d\nu_2(y) 
Q(y, dx)$. Then the measures $\rho\circ \pi^{-1}_i$ are defined by
$$
\rho\circ \pi^{-1}_1(A) = \int_A P(x, X_2)\; d\nu_1(x),\ \ 
 \rho\circ \pi^{-1}_2(B) = \int_B Q(y, X_1)\; d\nu_2(y).
$$
These equations show that $\rho\circ \pi^{-1}_i \ll \nu_i, i=1,2$. In 
particular, we can use the kernels $P$ and $Q$ to represent the 
measures
$$
\tau_B(A) = \int_A P(x, B)\; d\nu_1(x), \ \ \tau^A(B) = \int_B
 Q(y, A)\; d\nu_2(y)
$$

To prove that the converse holds for a given $\rho$, we take the
 measures 
$\tau_B(\cdot) = \rho(\cdot \times B)$ and  $\tau^A(\cdot) = 
\rho(A \times \cdot)$ where the sets $B\in \A_2$ and $A\in \A_1$ are 
of finite measure. Then, it is clear from \eqref{eq tau ll nu} that 
$$
\tau_B \ll \rho\circ \pi^{-1}_1 \ll \nu_1,\ \ \tau^A \ll 
\rho\circ \pi^{-1}_2 \ll \nu_2.
$$
Hence, we can define the Radon-Nikodym derivatives 
\be\label{eq_tau_A tau_B}
P(x, B) = \frac{d\tau_B}{d\nu_1}(x), \ \ Q(y, A) = \frac{d\tau^A}
{d\nu_2}(y), 
\ee
where $B\in \A_2$ and $A\in \A_1$ are fixed subsets. 
 To see that 
$P$ and $Q$ are finite  transition kernels, one needs to check that they
 define finite measures $P(x,\cdot)$ and $Q(y, \cdot)$ on $X_2$ and
  $X_1$ respectively. We remark   that 
$\tau_{B_1 \cup B_2}(A) = \rho(A \times (B_1\cup B_2)) =
\tau_{B_1}(A) + \tau_{B_2}(A)$ where $B_1 \cap B_2 =\emptyset$. 
This property can be extended to sigma-additivity using the fact that 
$\rho$ is a measure on the product space. The functions  $P(x, X_2)$
 and $Q(y, X_1)$ are  finite a.e. because they are the Radon-Nikodym 
 derivatives. Hence, using the kernels $P$ and $Q$, we see that  
 $(\nu_1, \nu_2) \in  F(\rho)$. 

\end{proof}

\begin{theorem}\label{thm_P determines Q}
Let $(X_1, \A_1)$ and $(X_2, \A_2)$ be   standard  Borel spaces, and
let  $P : X_1\times \A_2 \to [0,1]$ be a probability transition kernel.
 Suppose 
that $\nu_1$ is a Borel $\sigma$-finite measure on $(X_1, \A_1)$. 
Then there exist a Borel 
$\sigma$-finite measure $\rho$ on $Z = X_1 \times X_2$, a Borel 
$\sigma$-finite measure $\nu_2$ on $(X_2, \A_2)$,  and a transition
 kernel $Q : X_2\times \A_1 \to [0,1]$ such that $\rho \circ
 \pi^{-1}_i = \nu_i, i =1,2 $, 
\be\label{eq P_Q dual}
d\nu_1(x) P(x, dy) = d\nu_2(y) Q(y, dx),
\ee 
and 
\be\label{eq P-Q on nu}
\nu_1 P = \nu_2,\ \ \ \nu_2 Q = \nu_1.
\ee
The objects $\rho, \nu_2$, and $Q$ are uniquely defined by $P$ and 
$\nu_1$.
\end{theorem}

\begin{proof}
We first define a measure $\rho$ on rectangles $A\times B \in
\A_1 \times \A_2$ from $Z$ by setting
\be\label{eq rho via P}
\rho(A \times B) := \int_A (\{\delta_x\} \times P(x, B))\; d\nu_1(x)
\ee
where  $A$ is a set of finite measure $\nu_1$, and $B\in 
\A_2$ (we recall that $P(x, B) \leq P(x, X_2) =1$).
Then $\rho$ can be extended to a Borel $\sigma$-finite 
measure on the the product $\sigma$-algebra  $\A_1 \times \A_2$. 
By construction,  we have $\rho\circ\pi^{-1}_1 = \nu_1$ (see also 
Proposition \ref{cor_P,Q is 1} for more details).  

Let $E$ be the essential support of the measure $\rho$.
Define $\nu_2 := \rho\circ\pi^{-1}_2$.  We can disintegrate $\rho$
with respect to the measurable partitions $\{E_x : x \in X_1\}$ and 
$\{E^y : y \in X_2\}$ and measures $\nu_1$ and $\nu_2$. By 
Theorem \ref{thm Simmons}, 
\be\label{eq_disint rho}
\int_{X_1} \rho_x \; d\nu_1(x) =\rho = \int_{X_2} \rho^y \; 
d\nu_2(y).
\ee
By uniqueness of disintegrating family of measures, $\rho_x(\cdot)$ is
 supported by $E_x$ and can be identified with $P(x, \cdot)$.
 The measure $\rho^y$ is supported by the set
$E^y$.  We use the identification of $E^y$ with its projection onto 
$X_1$ to define a transition kernel $Q(y, A)$. It follows from 
 \eqref{eq_disint rho} that  we can set
$Q : X_2\times \A_1 \to [0,1]$ by letting 
\be\label{eq Q via rho}
\{\delta_y\} \times Q(y, A) = \rho^y(A), \ \ A\in \A_1.
\ee
In general, the kernel $Q : X_2\times \A_1 \to [0,1]$ need not to be 
finite.

Since the pairs $(\nu_1, P)$ and $(\nu_2, Q)$ generate the same
measure $\rho$, we see that equality \eqref{eq P_Q dual} holds. 

It remains to check that \eqref{eq P-Q on nu} holds: 
$$
\ba 
\nu_2 (B) &  = \rho(X_1 \times B)\\
& = \int_{X_1} P(x, B)\; d\nu_1(x)\\
&= \int_B \; d(\nu_1P)\\
&= (\nu_1P)(B), 
\ea
$$
and
$$
\ba 
(\nu_2Q)(A) & =   = \int_{X_2} Q(\chi_A)\; d\nu_2(y)\\
& = \int_{X_2} Q(y, A)\; d\nu_2(y)\\
& = \int_{A} P(x, X_2)\; d\nu_1(x)\\
&= \nu_1(A).
\ea
$$
We used \eqref{eq R and chi_A} here.

The fact that the measures $\rho, \nu_2$, and the kernel $Q$ are 
uniquely defined by $P$ and $\nu_1$ is based on the uniqueness of 
disintegration and follows directly from \eqref{eq rho via P} 
and \eqref{eq Q via rho}.
\end{proof}

Using a similar approach, we can show that the following corollaries 
hold. The proofs are straightforward and follows from the statements 
considered above in this section.

\begin{corollary}\label{cor prob P}
In conditions of Theorem \ref{thm_P determines Q}, suppose that 
the measure $\nu$ and kernel $P$ are probability. Then the measures
$\rho$ and $\nu_2$, and the transition kernel $Q$ are also probability.

If $c_1(x) =P(x, X_2)$ is in $L^1(\nu_1)$, then the kernel 
$Q(y, \cdot)$ is finite.
Moreover, the function $c_2(y)= Q(y, X_1)$ is locally integrable with
respect to $\nu_2$.  
\end{corollary}

\begin{corollary} \label{cor rho determ P Q}
Suppose that $\rho$ is a $\sigma$-finite Borel measure on the direct
product $Z$ of standard Borel spaces $(X_1, \A_1)$ and $(X_2, \A_2)$. 
Let   $\nu_i := \rho\circ \pi^{-1}_i$ be measures on $(X_i, \A_i), i 
=1,2$. Suppose that $\rho$ is disintegrated, $\rho = \int_{X_1} \rho_x
\; d\nu_1(x) = \int_{X_2} \rho^y\; d\nu_2(y)$ with respect to 
the projections $\pi_1$ and $\pi_2$ such that the fiber measures 
$\rho_x$ and $\rho^y$  are finite. Then there are  finite transition 
 kernels $P : X_1 \times \A_2 \to \R_+$ and 
$Q : X_2 \times \A_1 \to \R_+$ such that \eqref{eq P-Q on nu} holds
and
\be\label{eq_differential of rho}
d\rho(x, y) = P(x, dy)d\nu_1(x) = Q(y, dx) d\nu_2(y).
\ee
\end{corollary}

\begin{corollary} (1) 
Let $(X_i, \A_i, \nu_i), i=1,2$ be two $\sigma$-finite
measure spaces and let $P : X_1 \times \A_2$ be a transition kernel such
that $P(x, \cdot) \ll \nu_2$ for $\nu_1$-a.e. $x$. Denote
$$
\varphi(x, y) = \frac{P(x,dy)}{d\nu_2(y)}.
$$
Then the formula 
$$
Q(y, A) = \int_A \varphi(x, y)\; d\nu_1(x), \quad A \in \A_1,
$$
determines a dual transition kernel associated to $(\nu_1, \nu_2)$. 

(2) For a given pair of kernels $P$ and $Q$ such that $P(x, \cdot) \ll
\nu_1$ and $Q(y, \cdot) \ll\nu_2$, we have that $(\nu_1, \nu_2) \in 
M(P,Q) $ if and only if 
$$
\frac{P(x,dy)}{d\nu_2(y)} = \frac{Q(y,dx)}{d\nu_1(x)}.
$$
\end{corollary}
\begin{proof}
(1) Since all objects are well defined, we need to check only that, for any 
Borel bounded function $f(x, y)$,
$$
\iint_{X\times Y}d\nu_1(x) P(x, dy) f(x, y)= \iint_{X\times Y} d\nu_2(y) Q(y, dx) f(x, y).$$
It suffices to show that this relation holds for $f(x, y) = \chi_A(x)
\chi_B(y)$:
$$
\ba 
\int_{B} d\nu_2(y) Q(y, A) & =  \int_{B} d\nu_2(y) \int_A \varphi(x, y)\; 
d\nu_1(x)\\
 & =  \int_{B} d\nu_2(y) 
\int_A \frac{P(x,dy)}{d\nu_2(y)}d\nu_1(x) \\
& = \int_{B} \int_A P(x,dy) d\nu_1(x)\\
& = \int_A d\nu_1(x) P(x, B).
\ea 
$$

Statement (2) follows from (1).
\end{proof}

\begin{corollary}\label{cor equiv i - iii}
 Let $(X_i, \A_i, \nu_i)$ be standard 
$\sigma$-finite measure spaces, and let $P :\  X_1 \times \A_2 \to [0, 
\infty)$ and $ Q : \ X_2 \times \A_1 \to [0,\infty)$ be transition 
kernels, $i =1,2$. The following are equivalent:

(i) there exists a  measure $\rho$ on $X_1 \times X_2$ such 
that 
$$
\rho(A \times B) = \int_A P(x, B)\; d\nu_1(x) = \int_B Q(y, A)\; 
d\nu_2(y), \ \ A\in \A_1, B\in \A_2,
$$

(ii) $d\nu_1(x)P(x, dy) = d\nu_2(y)Q(y, dx) $, 

(iii) there exists a measure $\rho$ on $X_1 \times X_2$ such that,
 for any $B \in \A_2$ and  $A \in \A_1$, the kernels
$P$ and $Q$ are Radon-Nikodym derivatives:
$$
P(x, B) =\frac{\rho(dx, B)}{d\nu_1(x)},\ \ \ Q(y, A) =
\frac{\rho(A, dy)}{d\nu_2(y)}.
$$
\end{corollary}

\begin{proposition}\label{prop finite kernels}
Let $(X_i, \A_i, \nu_i), i=1,2,$ be two $\sigma$-finite measure spaces, and 
let $P$ and $Q$ be a pair of dual transition kernels, i.e., 
$d\nu_1(x) P(x, dy) = d\nu_2(y) Q(y, dx)$. Suppose that $P(\chi_B) \in 
L^1(\nu_1)$ for every $B\in \mc D(\nu_2)$. Then  $c_2(y) = Q(y, X_1)$ 
is  finite  a.e. and locally integrable. 
\end{proposition}  

\begin{proof}
The condition that $P$ and $Q$ form a dual pair can be written in the 
form
$$
\int_{X_1}\int_{X_2} d\nu_1(x) P(x, dy) f(x,y) = 
\int_{X_1}\int_{X_2} d\nu_2(y) Q(y, dx) f(x,y)
$$
for every bounded Borel function. Take a set $B\in \mc D(\nu_2)$. Then 
for $f(x, y) = \chi_B(y)$, we get 
$$
\int_{X_1} d\nu_1(x) \; \int_{X_2} P(x, dy) \chi_B(y) = 
\int_{X_2} d\nu_2 (y)\chi_B(y) \int_{X_1}  Q(y, dx) 
$$
or
$$
\int_{X_1} d\nu_1(x) \; P(x, B) =
\int_{B} d\nu_2 (y) \; c_2(y).
$$
Since the left hand side here is finite, we conclude that $c_2$ is finite
a.e. and locally integrable. 
\end{proof}

\begin{remark} Obviously, the condition $P(\chi_B) \in  L^1(\nu_1)$ for every $B\in \mc D(\nu_2)$ from Proposition \ref{prop finite kernels} holds
 for a probability transition kernel $P$. It follows that the dual transition
 kernel $Q$ is finite. By normalizing $Q$ and replacing the measure 
 $\nu_2$ by an equivalent measure, we can always assume that the 
 two kernels $P$ and $Q$ are probability.

\end{remark}

\subsection{More results on kernels and measures; examples}

Let $\rho$ and $\rho'$ are two equivalent measures on $Z = X 
\times X_2$. What can be said about the relation between the  
transition kernels $P, Q$ and $P', Q'$ 
generated by these measures? In the following lemma, we consider a
partial case when the Radon-Nikodym derivative for the measures
$\rho$ and $\rho'$ has some special form.

\begin{lemma}\label{lem rho and rho' equiv} 
(1) Let $\rho$ and $\rho'$ are two equivalent measures on $Z = X_1 
\times X_2$. Suppose that there exist two Borel functions $f : X_1 \to
(0,\infty)$ and $g : X_2 \to (0, \infty)$ such that 
$$
\frac{d\rho'}{d\rho}(x, y) = f(x) g(y).
$$
Let $P(x, \cdot) = \rho_x(\cdot)$ and $P'(x, \cdot) = \rho'_x(\cdot)$
be the corresponding transition kernels, see Corollary 
\ref{cor rho determ P Q} . 
Then 
$$
\frac{P'(x, dy)}{P(x,dy)} = k_x g(y)
$$
where the coefficient $k_x  = f(x) \dfrac{d\nu_1}{d\nu'_1}(x)$ 
depends on $x$ only.

In particular, if $\rho'$ and $\rho$ are probability measures such that
$d\rho'(x, y) = f(x) d\rho(x, y)$, then $P'(x, dy)= P(x, dy)$ and 
$f(x)$ is the Radon-Nikodym $\dfrac{d\nu'}{d\nu}(x)$.

(2)  If $d\nu'_1(x) := \varphi(x) d\nu_1(x)$ where $\varphi$ is a Borel
positive function, then, in 
notation of Corollary \ref{cor rho determ P Q}, we obtain 
$$
P(x, dy)d\nu'_1(x) = Q'(y, dx)d\nu_2(y)
$$
where $ Q'(y, dx) := \varphi(x) Q(y, dx)$.
\end{lemma}

\begin{proof}
It follows from Corollary  \ref{cor equiv i - iii} (ii) that 
$$
P'(x, dy) d\nu'_1(x) = f(x) g(y) P(x, dy) d\nu_1(x).
$$
This equality can be written as in (1). Assuming that $g(y) =1$ and 
using the fact that the measures $\rho$ and $\rho'$ are probability,
we obtain the other statement from (1). 

Relation (2) follows from Corollary \ref{cor equiv i - iii}.
\end{proof}

\begin{lemma} Suppose that $P$ and $P'$ are finite transition kernels
on $X_1 \times \A_2$. 
Let $(nu, P)$ and $(\nu', P')$ be pairs that determine measures $\rho
= \int P(x, \cdot) \; d\nu$ and $\rho' \int P'(x, \cdot) \; d\nu'$ on 
$X_1 \times X_2$. Then $\rho = \rho'$ if and only if 
$\int_A c\; d\nu = \int_A c'\;d\nu'$ where $c(x) = P(x, X_2)$,
$c'(x) = P'(x, X_2)$. In particular, if $P$ and $P'$ are probability, then
$\rho = \rho'$ if and only if $\nu= \nu'$.
\end{lemma} 

\begin{proof}
We leave the proof to the reader.
\end{proof}

\begin{proposition}\label{cor_P,Q is 1}
Let $(X_i, \A_i, \nu_i), i =1,2,$ be standard 
measure spaces with $\sigma$-finite measures $\nu_1$ and $\nu_2$,
 and let 
$P :\  X_1 \times \A_2 \to [0,1]$
and $ Q : \ X_2 \times \A_1 \to [0,1]$ be  transition 
kernels such that 
$$P(x, dy)d\nu_1(x) = Q(y, dx)d\nu_2(y).$$ Then,
for the measure $d\rho(x,y) =P(x, dy)d\nu_1(x) = Q(y, dx)d\nu_2(y)$
on $X_1\times X_2$, we have

(i)\ $\nu_1 P = \nu_2 \ \Longleftrightarrow \ Q(y, X_1) = 1$ for 
$\nu_2$-a.e. $y$,

\hskip 0.55cm  $\nu_2 Q = \nu_1 \ \Longleftrightarrow \ P(x, X_2) = 1$
 for  $\nu_1$-a.e. $x$;

(ii) $\nu_1  = \rho\circ\pi_1^{-1} \ \Longleftrightarrow \ 
P(x, X_2) = 1$ for  $\nu_1$-a.e. $x$,

\hskip 0.65cm  $\nu_2  = \rho\circ\pi_2^{-1} \ \Longleftrightarrow \ 
Q(y, X_1) =  1$ for  $\nu_2$-a.e. $y$;

(iii) $\nu_1 P = \nu_2 \ \Longleftrightarrow \ \nu_1  = \rho\circ
\pi_1^{-1}$,

\hskip 0.7cm  $\nu_2 Q = \nu_1 \ \Longleftrightarrow \ \nu_2  = \rho\circ\pi_2^{-1}$.
 
\end{proposition}

\begin{proof} 
In the proof,  we use the measure $\rho$ on $X_1 \times X_2$ defined
by $\nu_1, P$ or $\nu_2, Q$:
\be\label{eq rho(AxB)}
\rho(A\times B) = \int_A P(x, B)\; d\nu_1(x) = \int_B Q(y, A)\; 
d\nu_2(y)
\ee
where $B \in \A_2$ and  $A \in \A_1$ are any Borel sets. 

For (ii), 
\be\label{eq equiv i ii iii}
\rho\circ\pi_1^{-1}(A) = \rho(A \times X_2) = \int_A P(x, X_2)\; 
d\nu_1(x).  
\ee
If $P(x, X_2 =1$ a.e., then $\rho\circ\pi_1^{-1} = \nu_1$. Conversely,
if for any $A\in X_1$, $\rho\circ\pi_1^{-1}(A) = \nu_1(A)$, then 
$P(x, X_2 =1$ a.e. (otherwise, equation \eqref{eq equiv i ii iii} 
immediately leads to a contradiction). The other statement in (ii) is 
proved similarly.

For (iii), suppose that $\nu_i = \rho\circ\pi^{-1}_i, i =1,2$. Then
$$
(\nu_1P)(B) = \int_{X_1} P(\chi_B)\; d\nu_1 = \rho (X_1 \times B) 
=  \rho\circ\pi^{-1}_2(B) = \nu_2(B).
$$
Conversely, if $\nu_1 P = \nu_2$, the above relation proves that
 $\nu_2 = \rho\circ\pi^{-1}_2$. The other statement of (iii) can be 
 checked analogously.
 
Statement  (i) follows automatically from (ii) and (iii).
\end{proof}

We consider now the notion of the \textit{product of transition 
kernels.}
Let $(X_i, \A_i)$ be  standard Borel spaces, $i =1,2,3$. Suppose that 
$P_i : X_i \times \A_{i+1} \to [0,1], i = 1,2,$ is a $\sigma$-finite 
transition kernel. For $x_1\in X_1, C \in \A_3$,  we define
$$
(P_1P_2)(x_1, C) := \int_{X_2} P_1(x_1, dx_2) P_2(x_2, C).
$$
We define the action of $P_1P_2$ on Borel functions and
measures: for $f \in \mc F(X_3, \A_3)$, we set
$$
P_1P_2(f)(x_1) = \int_{X_2}\int_{X_3} P_1(x_1, dx_2)P_2(x_2, 
dx_3) f(x_3)
$$
is a Borel function on $(X_1, \A_1)$. If $\nu_1$ is a Borel measure on
$(X_1, \A_1)$, then one can consequently define $\nu_2 = \nu_1 P_1$
 and $\nu_3 = \nu_2 P_2 = \nu_1 P_1P_2$.

Hence, we have two measures $\rho_1$ and $\rho_2$ on $X_1 \times 
X_2$ and $X_2 \times X_3$, respectively,  such that
$$
d\rho_1(x_1, x_2) = P_1(x_1, dx_2) d\nu_1(x_1), \ \ \  
d\rho_2(x_2, x_3) = P_2(x_2, dx_3) d\nu_2(x_2).
$$
On the other hand, the transition probability kernel $P= P_1P_2$ and
the initial measure $\nu_1$ can be used to define a probability measure
$\rho$ on $(X_1 \times X_3, \A_1 \times \A_3)$.

\begin{lemma}
In the above notation,
$$
d\rho(x_1, x_3) = \int_{X_2} d\rho_1(x_1, x_2) P_2(x_2, dx_3).
$$
\end{lemma}

\begin{proof}
Indeed, it follows from the relations
$$
P(x_1, dx_3) = \int_{X_2} P_1(x_1, dx_2) P_2(x_2, dx_3).
$$
and 
$$
d\rho(x_1, x_3) = d\nu_1(x_1) P(x_1, dx_3).
$$
\end{proof}

\begin{example}
We consider a particular case of \textit{atomic
measures} on a \textit{standard Borel space}. 
Suppose that we have two copies, 
$(X_i, \A_i), i =1,2,$ of a standard Borel space $(X, \A)$.
According to Theorem \ref{thm_P determines Q}, we begin with a
 probability kernel $P(x, B)$ and a measure $\nu_1$ and define other
 objects.

\begin{lemma} \label{lem exmpl delta}
Let $P(x, \cdot)$ be a finite transition kernel and 
$\nu_1 = \delta_{x_0}, x_0 \in X_1$. 
 Then the pair $(\delta_{x_0}, P)$ defines the measure $\rho 
= \rho(\delta_{x_0})$ on $X_1 \times X_2$ by the formula 
$$
\rho(A\times B)= \chi_A(x_0) P(x_0, B).
$$
 Moreover, the projection 
of $\rho$ onto $X_2$ is the  measure $\nu_2(B) = P(x_0, B)$, and 
the dual kernel $Q(y, \cdot)$ is $\delta_{x_0}(\cdot)$.
\end{lemma}

\begin{proof} We compute the measure $\rho$ as in Theorem 
\ref{thm_P determines Q}:
\be\label{eq_rho via delta}
\ba
\rho(A\times B)  &= \int_A P(x, B)\; \delta_{x_0}(dx)\\
&  = \begin{cases}
P(x_0, B), \ & x_0 \in A\\
0, & x_0 \notin A 
\end{cases}\\ 
& = \chi_A(x_0) P(x_0, B)\\
& = \int_B \chi_A(x_0) P(x_0, dy).
\ea
\ee
Therefore, we can take the projection of $\rho$ onto $X_2$ and find
$$ 
\nu_2(B)  = \rho(X_1\times B) = \chi_{X_1}(x_0) 
P(x_0, B) = P(x_0, B).
$$ 
On the other hand, when we  use the dual kernel 
$Q(y, A)$ to calculate $\rho(A\times B)$, we have 
\be\label{eq_rho via Q}
\rho(A\times B) = \int_B Q(y, A)\; \nu_2(dy) = \int_B Q(y, A)\; 
P(x_0, dy). 
\ee
By uniqueness of disintegration, it follows from  
\eqref{eq_rho via delta} and \eqref{eq_rho via Q} that $Q(y, A) =
\delta_{x_0}(A)$.
\end{proof}

The probability transition kernels $P$ and $Q$, which are defined as in 
Lemma \ref{lem exmpl delta}, act on functions as follows:
$$
(Pf)(x) = \int_{X_2} P(x, dy) f(y), \ \ f \in \mc F(X_2, \A_2),
$$
$$  (Qg)(y) = \int_{X_1}Q(y, dx)g(x) = g(x_0), \ \ g \in \mc F(X_1, 
\A_1).
$$

\end{example}

\section{Actions of transition kernels in $L^2$-spaces}
\label{ssect Kernels on L^2}

We introduce a new \textit{duality framework} for transition operators. While the 
setting below is that of pairs of measure spaces in duality, our applications 
(in section 8 below) will be to path-space measures and the corresponding 
 Markov processes.

In this section, we will focus on the actions of operators $T_P$ and 
$T_Q$ in $L^2$-spaces.  It turns out that they produce the self-adjoint
operators $T_{PQ}$ and $T_{QP}$ acting in $L^2(\nu_2)$ and 
$L^2(\nu_1)$. 
Our analysis of such \textit{dual pairs of operators} is in the framework of Hilbert 
space. This holds both for the results in the present Section 
\ref{ssect Kernels on L^2}, and Section \ref{ssect L-set}
below. Indeed, these duality results will serve as key tools in our 
subsequent introduction of non-stationary Markov processes, and their 
harmonic analysis; see Section \ref{sect meas BD} below. In more detail,
 the present duality 
results will be used at each level in \textit{discrete-time Markov process} dynamics.

\subsection{Symmetric measures and symmetric operators}
\label{ssect symm meas}

In this subsection we give a few definitions and results about symmetric
measures and associated with them operators. We refer to our
previous works \cite{BezuglyiJorgensen_2018, 
BezuglyiJorgensen_2019, BezuglyiJorgensen_2019a} where the reader can find more details. 

Let $\lambda$ be a measure on the  Cartesian product $(X \times X, 
\A \times \A)$ of a standard Borel space $(X, \A)$ such that 
$$
\lambda(A \times B) = \lambda(B\times A), \ \ A, B \in \A.
$$ 
 Then $\lambda$ is called a \textit{symmetric measure.}
 
 A positive linear operator $R ; \mc F(X, \A) \to \mc F(X, \A)$ is called 
 \textit{symmetric} if it satisfies the following relation:
 $$
 \int_X f R(g)\; d\mu = \int_X R(f) g\; d\mu,  
 $$
where $\mu$ is a $\sigma$-finite measure on $(X, \A)$ and $f, g \in 
\mc F(X, \A)$.

\begin{remark}
(1) It can be easily  seen that the operator $R$ is symmetric if and only 
if the measure 
$$
\lambda(A \times B) = \int_X \chi_A R(\chi_B)\; d\mu 
$$
is symmetric.  

(2) When $R$ is considered as an operator (in general, unbounded) 
acting in  $L^2(\nu)$, then  we can use the methods of operators in
Hilbert spaces for the study of its properties, see Theorem
\ref{thm from BJ19} and Theorem \ref{thm RKHS} below for 
further details.

(3) If $\lambda$ is a symmetric measure on the product space $X \times
X$, then the operator $R$ is not uniquely determined by $\lambda$, see
e.g., Theorem \ref{thm RKHS}. 

\end{remark}

Let $(X, \A, \mu)$ be a $\sigma$-finite measure space, and let 
$\lambda$ be a 
symmetric measure on $X\times X$ supported by a symmetric set 
$E$. Let   $x\mapsto \lambda_x$ be the measurable family of conditional 
measures  on $(X, \A)$ that  \textit{disintegrates} 
$\lambda$. We assume that the function  $c(x) = \lambda_x(X)$ is
 finite for $\mu$-a.e. $x$. 

\begin{definition}\label{def R, P, Delta}
For  a symmetric measure $\lambda$  on $(X \times X, \A\times \A)$, 
we define two linear operators $R$ and $R_1$ acting on the space of
bounded Borel functions $\mc F(X, \A)$.\:

(i) The \textit{symmetric operator}$R$:
\be\label{eq def of R} 
R(f)(x) := \int_V f(y) \; d\lambda_x(y) = \lambda_x(f). 
\ee 

(ii) The\textit{ Markov operator} $R_1$\footnote{The more natural 
notation $P$ has been already reserved for a finite transition kernel between two standard spaces.}:
$$
R_1(f)(x) = \frac{1}{c(x)}R(f)(x)
$$
or
\be\label{eq formula for P}
 R_1(f)(x) := \frac{1}{c(x)}  \int_V f(y) \; d\rho_x(y) = \int_V f(y) \; 
 P_1(x, dy)
 \ee
 where  $P_1(x, dy)$ is the probability measure obtained by normalization
 of $d\lambda_x(y)$, i.e. 
 $$
 P_1(x, dy) := \frac{1}{c(x)}d\lambda_x(y).
 $$
 \end{definition}

In the case when the operator $R$ is considered in a $L^2$-space,
we will have to deal with self-adjoint operators which are unbounded,
in general. We refer to the well known books 
\cite{DanfordSchwartz1988, Kato1995, Schmeudgen2012}
 on unbounded operators and their self-adjoint extensions.

\begin{theorem}[\cite{BezuglyiJorgensen_2019}] 
\label{thm from BJ19}
\label{prop prop of R, P, Delta}
For a standard measure space $(X, \A, \mu)$, let $\lambda$ be a
 symmetric 
measure on $X \times X$ such that the function $c(x) = 
\lambda_x(X)$ is locally integrable. 
Let $d\nu(x) = c(x) d\mu(x)$ be the $\sigma$-finite  measure on 
$(X, \A)$ equivalent to $\mu$, and let the operators $R$ and 
$R_1$ be defined as in Definition \ref{def R, P, Delta}. 

(1) Suppose that the function $x \mapsto \rho_x(A) \in L^2(\mu)$ for 
every $A \in \mc D(\mu)$\footnote{This means that the operator $R$ is densely 
defined on functions from $\mc D(\mu)$; in particular, this property holds
if $c \in L^2(\mu)$}.  Then $R$ is a symmetric unbounded operator  in 
$L^2(\mu)$, i.e.,   
\be\label{eq-R symm}
\langle g, R(f) \rangle_{L^2(\mu)} = \langle R(g), f \rangle_{L^2(\mu)}.
\ee
If $c \in L^\infty(\mu)$, then $R :  L^2(\mu) \to 
L^2(\mu)$ is a bounded operator, and 
$$
||R||_{L^2(\mu) \to L^2(\mu)} \leq ||c||_{\infty}.
$$
Relation \eqref{eq-R symm} is equivalent to the symmetry of the 
measure $\lambda$. 

(2) The operator $R : L^1(\nu) \to L^1(\mu)$ is contractive, i.e.,  
$$
||R(f)||_{L^1(\mu)} \leq ||f||_{L^1(\nu)}, \qquad f \in L^1(\nu).
$$
Moreover,   for any function $f \in L^{1} (\nu)$,  the formula 
\be\label{eq c(x) and rho_x}
\int_V R(f) \; d\mu(x) = \int_V f(x) c(x) \; d\mu(x)
\ee
holds. In other words, $\nu = \mu R$, and 
$$
\frac{d(\mu R)}{d\mu}(x) = c(x).
$$

(3) The bounded operator $R_1: L^2(\nu) \to L^2(\nu)$ is self-adjoint.
Moreover, $\nu R_1 = \nu$. 

(4) The operator $R_1$, considered in the spaces $L^2(\nu)$ and $L^1(\nu)$, 
is  contractive, i.e., 
$$
|| R_1(f) ||_{L^2(\nu)} \leq || f ||_{L^2(\nu)}, \qquad 
|| R_1(f) ||_{L^1(\nu)} \leq || f ||_{L^1(\nu)}.
$$ 

(5) The spectrum of $R_1$ in $L^2(\nu)$ is a subset of $[-1, 1]$.
\end{theorem}

\begin{definition}\label{def reversible MP-1} 
Suppose that $x \mapsto P_1(x, \cdot )$ is a measurable family of
probability  transition kernels on the space $(X, \A, \mu)$, 
and let $R_1$ be the  Markov operator 
generated by $x \mapsto P_1(x, \cdot )$. It is said that the corresponding
 Markov process  is  \textit{reversible} with respect to a measurable 
 functions $c: X \to  (0, \infty)$ on $(X, \A)$ if, for any sets 
 $A, B \in \B$, the following relation holds:
\be\label{eq def reversible P}
 \int_B c(x) P_1(x, A)\; d\mu(x) = \int_A c(x) P_1(x, B)\; d\mu(x).
\ee
\end{definition} 

It turns out that the notion of reversibility is equivalent to the following
properties.

\begin{theorem}[\cite{BezuglyiJorgensen_2019a, BezuglyiJorgensen_2018}]
\label{prop_reversible}
Let $(X, \A, \mu)$ be a standard $\sigma$-finite measure space, 
$x \mapsto c(x)
\in (0, \infty)$ a measurable locally integrable function. Suppose that 
$x \mapsto P_1(x, \cdot )$ is a probability kernel. 
The following are equivalent:

(i) $x \mapsto P_1x, \cdot )$ is reversible (i.e., it satisfies 
(\ref{eq def reversible P});

(ii) the Markov operator $R_1$ defined by $x\to P_1(x, \cdot)$ is  
self-adjoint 
on $L^2(\nu)$ and $\nu R_1= \nu$ where $d\nu(x) = c(x) d\mu(x)$;

(iii) 
$$
c(x) P_1(x, dy) d\mu(x) = c(y) P_1(y, dx)d\mu(y);
$$

(iv) the operator $R$ defined by the relation $R(f)(x) = c(x)R_1(f)(x)$
is symmetric;

(v) the measure $\lambda$ on $(X\times X, \A\times \A)$ defined by 
$$
\lambda(A \times B) = \int_X \chi_A R(\chi_B)\; d\mu 
$$
is symmetric.

\end{theorem}

\subsection{Transition kernels and corresponding operators
$T_P$ and $T_Q$}

As we seen in Section \ref{sect Trans kernels}, the transition kernels $P$
 and $Q$ generate  linear operators 
acting on the space of bounded Borel functions. To emphasize the
 difference 
between the kernel $P$ and the corresponding operator, we will denote
the latter by $T_P$.  Let $\mc F(X, \A) = \mc F(X)$ denote the space of
bounded Borel functions. Define $T_P : \mc F(X_2) \to \mc F(X_1)$ and
$T_Q : \mc F(X_1) \to \mc F(X_2)$ by setting
\be\label{eq_def T_P T_Q}
T_P(g)(x) = \int_{X_2}  P(x, dy) g(y),\ \ \  
T_Q(f)(x) = \int_{X_1}  Q(y, dx) f(x).
\ee

We observe that the kernels $P$ and $Q$ admit the following 
representations: for $A\in \A_1, B\in \A_2$,
$$
P(x, B) = T_P(\chi_B)(x), \ \ \ Q(y, A) = T_Q(\chi_A)(y).
$$

Recall that any finite transition kernel $R : X \times \A \to [0,1]$
 acts also on the set of  $\sigma$-finite measures $M(X)$: for $\mu \in 
 M(X)$, 
$$
(\mu R)(f) = \int_X R(f)\; d\mu = \int_X f \; d\mu R.
$$

In this section, we focus on actions of the operators $T_P$ and $T_Q$
in the corresponding 
$L^2$-spaces.  In general, these operators are unbounded with 
dense domain. The following references are devoted to various aspects 
of \textit{dual pairs of unbounded operators}: \cite{JorgensenPearse_2016,
JorgensenPearse_2017, JorgensenPearseTian_2018}. A discrete 
analogue of Theorem \ref{thm T_P contractive} was considered in 
Proposition \ref{prop T_P T_Q}. 

\begin{theorem}\label{thm T_P contractive}
(1) Let $(X_i, \A_i, \nu_i), i=1,2,$ be standard $\sigma$-finite 
measure spaces. Let $P$ and $Q$ form a dual pair of probability kernels
associated to the measures $\nu_1$ and $\nu_2$. 
Then  $P : X_1 \times \A_2 \to [0,1]$ defines
a contractive operator $T_P : L^2(\nu_2) \to L^2(\nu_1)$. Similarly, 
$T_Q$ is a contactive operator from $L^2(\nu_1)$ to $L^2(\nu_2)$.

(2) Suppose that 
$(P,Q)$ is a pair of dual probability transition kernels associated to 
the measures $\nu_1$ and $\nu_2$ as in (1).  Then
$$
\langle f, T_P(g)\rangle_{L^2(\nu_1)} = \langle T_Q(f), g
\rangle_{L^2(\nu_2)}, \ \ \ f \in L^2(\nu_1), g \in L^2(\nu_2),
$$
i.e., $(T_P)^* = T_Q$ and $(T_Q)^* = T_P$.

(3) Let  $\nu_1$ and $\nu_2$ be  Borel
$\sigma$-finite measures on $(X_i, \A_i), i =1,2,$ satisfying Corollary 
\ref{cor equiv i - iii} (ii).  
Suppose that the finite transition kernels $P(x, \cdot)$ and $Q(y, \cdot)$ have the property:  $P(\chi_B) \in L^2(\nu_1)$ and 
$Q(\chi_A) \in L^2(\nu_2)$ for any Borel sets $A \in \A_1$ and 
$B\in \A_2$ of finite measure. 
Then the operators $T_P$ and $T_Q$ 
$$
T_P : L^2(\nu_2) \longrightarrow L^2(\nu_1): \ \ f \mapsto 
\int_{X_2} P(x, dy) f(y),
$$ 
$$
T_Q : L^2(\nu_1) \longrightarrow L^2(\nu_2): \ \ f \mapsto 
\int_{X_1} Q(y, dx) f(x)
$$ 
are densely defined, and they satisfy the relations $T_P \subset 
(T_Q)^*$, $T_Q \subset (T_P)^*$.
The operators $T_P$ and $T_Q$ are  closable.
\end{theorem}

\begin{remark}
 The containment in the statement of Theorem \ref{thm T_P contractive} 
 (3) refers to containment of unbounded densely defined linear operators,
 i.e., containment of the respective graphs. A pair of operators, with each 
 one contained in the adjoint of the other, is called a 
 \textit{symmetric pair}. Such pairs arise naturally in many areas of pure
  and applied analysis, see e.g., \cite{Jorgensen2012,  
  JorgensenPearse-2020, JorgensenPearse_2016, 
 JorgensenPearse_2017, JorgensenPearseTian_2018, Kato1995, 
 Schmeudgen2012}, and the papers cited there.
\end{remark}

\begin{proof} For (1), let $\rho $ be the measure on $X_1 \times X_2$ 
defined by the dual pair $P, Q$ and the measures $\nu_1, \nu_2$ as in 
\eqref{eq_rho first def}. 
We first show that for any function $g \in L^2(\nu_2)$ the function
$T_P(g)$ belongs to $L^2(\nu_1)$. For this, we  compute

$$
\ba 
\int_{X_1} |T_P(g)|^2\; d\nu_1(x) & = \int_{X_1}
\left| \int_{X_2} P(x, dy) g(y)\right|^2 \; d\nu_1(x) \\
& \leq \iint_{X_1\times X_2} g(y)^2 P(x, dy) d\nu_1(x)\\
& =   \iint_{X_1\times X_2} g(y)^2 \; d\rho(x, y)\\
& =  \iint_{X_1\times X_2} g(y)^2 Q(y, dx) d\nu_2(y)\\
& = \int_{X_2} g(y)^2 \; d\nu_2(y)\\
& = ||g||^2_{L^2(\nu_2)}. 
\ea
$$
We used here the fact that $P$ and $Q$ are probability kernels. Thus,
we proved that $|| T_P||_{ L^2(\nu_2) \to L^2(\nu_1)} \leq 1 $.
Similarly, one can show that $T_Q$ is contractive. This proves (1). 

For (2), let $\rho$ be the measure defined by $(nu_1, P)$ (or 
$\nu_2, Q)$. If $A\times B \in \A_1 \times \A_2$ (where $A, B$ are of
finite measure), then we can represent $\rho(A \times B)$ in two ways:
$$
\rho(A\times B) = \int_A P(x, B)\; d\nu_1(x) = \int_{X_1} \chi_A(x)
P(\chi_B)(x)\; d\nu_1(x) = \langle f, T_P(g)\rangle_{L^2(\nu_1)},
$$
and 
$$
\rho(A\times B)= \int_B Q(y, A)\; d\nu_2(y) = \int_{X_2} \chi_B(y)
Q(\chi_A)(y)\; d\nu_2(y) = \langle g, T_Q(f) \rangle_{L^2(\nu_2)}.
$$
Then (2) follows. 

To show (3), we first observe that the linear operators $T_P$ and 
$T_Q$, considered in the $L^2$-spaces, are in general 
unbounded in the case when the measures $\nu_1$ and $\nu_2$ are 
$\sigma$-finite and the transition kernels $P, Q$ are finite. It is not hard to
 see that the operators $T_P$ and $T_Q$ are densely defined on a set 
 $\mc F(\nu_i)$ which is spanned by
characteristic functions of subsets with finite measure $\nu_i$. 
Moreover, it follows from the relation
$$
\int_{X_2} (T_Q f) g\; d\nu_2 = \int_{X_1} (T_Pg) f\; d\nu_1
$$
(where $f$ and $g$ are  functions from $\mc D$) that $T_P \subset 
(T_Q)^*$. This proves that   $T_P$ is closable. 
Similarly, one can show that $T_Q$ is closable.
\end{proof}

\begin{corollary}\label{cor tfae 5}
Let $(X_i, \A_i, \nu)i)$ be $\sigma$-finite standard measure spaces,
$i =1,2$. Suppose that $P$ and $Q$ are transition kernels which define 
the operators $T_P$ and $T_Q$ such that 
$$
T_P(g) = \int_{X_2} P(x, dy)g(y) : L^2(\nu_2) \to L^2(\nu_1)
$$
$$
T_Q(f) = \int_{X_1} Q(y, dx)f(x) : L^2(\nu_1) \to L^2(\nu_2).
$$
Then the following are equivalent:

(i) $T_P \subset (T_Q)^*$,

(ii) $T_Q \subset (T_P)^*$,

(iii) $d\nu_1(x) P(x, dy) = d\nu_2(y) Q(y, dx)$, and this relation 
defines a measure $\rho$ on $X_1\times X_2$ which belongs to
the set $L(\nu_1, \nu_2)$,

(iv) for all sets $A\in \mc D(\nu_1), B\in \mc D(\nu_2)$, 
$$
\int_A P(x, B)  \; d\nu_1(x) = \int_{X_2} Q(y, A)\; d\nu_2(y),
$$


\end{corollary}

\begin{proof} These results follow from Theorem 
\ref{thm T_P contractive}. We leave the details for the reader.
\end{proof}

As we seen, the operators $T_P$ and $T_Q$ are, in general, unbounded. 
Here we address the question under what conditions these operators 
 are bounded in the $L^2$-spaces.  

\begin{proposition}\label{prop bounded T_P T_Q}
(1) Let $T_P$ and $T_Q$ be as above and $c_1(x) = P(x, X_2)$. 
If 
$$
M= \mathrm{ess\;sup}_{y\in X_2} \int_{X_1} Q(y, dx) c_1(x) < \infty,
$$
then the operator $T_P : L^2(\nu_2) \to L^2(\nu_1)$ is bounded,
 $||T_P||_{L^2(\nu_2) \to L^2(\nu_1)} \leq M$. A
 similar statement holds for the operator $T_Q$ when the roles of 
 $T_P$ and $T_Q$ are interchanged. 

(2) Suppose there are functions  $a\in L^2(\nu_1)$ and $b\in 
L^2(\nu_2)$ such that 
$$
\frac{P(x, dy)}{d\nu_2(y)} \leq a(x) b(y).
$$
Then 
$$
||T_P||_{L^2(\nu_2) \to L^2(\nu_1)} \leq || a ||_{L^2(\nu_1)} 
|| b ||_{L^2(\nu_2)}.
$$
A similar estimate holds for the norm of $T_Q$ if there are functions
$a_1\in L^2(\nu_1)$ and $b_1\in  L^2(\nu_2)$ such that 
$$
\frac{Q(y, dx)}{d\nu_1(x)} \leq a_1(x) b_1(y).
$$

\end{proposition}

\begin{proof} (1) To prove the first assertion, we use the 
Fubini-Tonelli theorem: in the computation below the right-hand side
is well defined, so that $P(f) \in L^2(\nu_1)$: 

$$
\ba
\int_{X_1} P(f)^2\; d\nu_1 & = \int_{X_1} \left(\int_{X_2} P(x, dy)
f(y) \right)^2\; d\nu_1\\
& \leq \int_{X_1} \left(\int_{X_2} P(x, dy)\right)  \left(\int_{X_2} 
P(x, dy) f^2(y) \right)  \; d\nu_1\\ 
& = \int_{X_1} \left(\int_{X_2} P(x, dy)
f^2(y) \right) c_1(x)\; d\nu_1(x)\\
& = \int_{X_2}f^2(y)\; d\nu_2(y)  \left(\int_{X_1} Q(y, dx) c_1(x)
 \right) \\
& = || f||^2_{L^2(\nu_2)} \int_{X_1} Q(y, dx) c_1(x).
\ea
$$
If 
$$
M = \mathrm{ess\;sup}_{y\in X_2} \int_{X_1} Q(y, dx) c_1(x) < 
\infty,
$$
then the norm of $T_P$ is bounded by $M$. 

(2) Recall that it was proved above that, for a.e. $x_1 \in X_1$, the
 measures $P(x, \cdot)$ is absolutely continuous with respect to  
 $\nu_2$. Hence,
the result follows from the following computation: 
for $f \in L^2(\nu_2)$,

$$
\ba 
||T_P(f) ||^2_{L^2(\nu_2) \to L^2(\nu_1)} &= 
\int_{X_1} \left(\int_{X_2} P(x, dy)f(y)\right)^2\;  d\nu_1(x)\\
& = \int_{X_1} a^2(x) \left(\int_{X_2} f(y) b(y) \; d\nu_2\right)^2\;  
d\nu_1(x)\\
& \leq || f ||^2_{L^2(\nu_2)} || b ||^2_{L^2(\nu_2)} \int_{X_1} 
a^2(x)  d\nu_1(x)\\
& = || a ||^2_{L^2(\nu_1)}|| b ||^2_{L^2(\nu_2)}
|| f ||^2_{L^2(\nu_2)}. 
\ea
$$
\end{proof}

The proved Theorem \ref{thm T_P contractive} has a few important
 consequences. We consider some
of them in the following statements. As follows from Theorem 
\ref{thm T_P contractive}, one can define the product of operators 
$T_P : L^2(\nu_2) \longrightarrow L^2(\nu_1)$ and 
$T_Q : L^2(\nu_1) \longrightarrow L^2(\nu_2)$. 

\begin{lemma}\label{lem kernels pq and qp}
Let $P$ and $Q$ be finite transition kernels defined on the measure spaces $(X_i, \A_i, \nu_i)$ as in Definition \ref{def P,Q first}. 
Suppose  that $d\nu_1(x)P(x,dy) = d\nu_2(y)Q(y, dx)$. 
Then they generate the finite transition kernels 
$PQ : X_1 \times \A_1 \to [0, \infty)$ and  
$QP : X_2 \times \A_2 \to [0, \infty)$:
$$
PQ(x, A) = \int_{X_2} P(x, dy) Q(y, A), \ \ A\in \A_1,
$$
$$
QP(y, B) = \int_{X_1} Q(y, dx) P(x, B), \ \ B\in \A_2.
$$
\end{lemma}
\begin{proof}
These formulas can be proved directly.
\end{proof}

Consider now the corresponding linear operators $T_{PQ}$ and 
$T_{QP}$ acting in the corresponding $L^2$-spaces. 

\begin{lemma}
Suppose that finite transition kernels $P$ and $Q$ satisfy the property:
$$
P(\chi_B) \in L^2(\nu_1), \ \ \forall B\in \A_2,\  \nu_2(B) < \infty,
$$  
and 
$$
Q(\chi_A) \in L^2(\nu_2), \ \ \forall A\in \A_1,\  \nu_1(A) < \infty.
$$  
Then the operators $T_{PQ} : L^2(\nu_1) \to L^2(\nu_1)$ and 
$T_{QP} : L^2(\nu_2) \to L^2(\nu_2)$ are self-adjoint densely defined 
linear operators acting by the formulas:
$$
T_{PQ}(\chi_A) = T_P(Q( \cdot, A)(x), \ \ A\in \A_1,\ \nu_1(A) < \infty,
$$
$$
T_{QP}(\chi_B) = T_Q(P( \cdot, B)(y), \ \ B\in \A_2,\ \nu_2(B) < \infty.
$$
\end{lemma}

\begin{proof}
It suffices to notice that, by the condition of the lemma, we can 
consequently apply the operators $T_Q$ and $T_P$ to get $T_{PQ}$ 
because the function $T_Q(\chi_A)$ is in $L^2(\nu_2)$. 
It gives the operator acting in $L^2(\nu_1)$. The same is true for 
$T_{QP}$ as an operator in $L^2(\nu_2)$.

Since $T_P \subset (T_Q)^*$ and $T_Q \subset (T_P)^*$, we see that 
$T_{PQ}$ and $T_{QP}$ are self-adjoint.
\end{proof}

Apply the approach used in Section \ref{sect Trans kernels} to the 
kernels $PQ$ and $QP$ considered in Lemma \ref{lem kernels pq and 
qp}. This means that we can define two measures $\lambda_1$ and 
$\lambda_2$ on the product spaces $X_1\times X_1$ and 
$X_2\times X_2$, respectively:
$$
\lambda_1(A_1 \times A_2) = \int_{A_1} d\nu_1(x) (PQ)(x, A_2)
$$
and 
$$
\lambda_2(B_1 \times B_2) = \int_{B_1} d\nu_2(y) (QP)(y, B_2).
$$

\begin{theorem}\label{thm symm meas}

For any sets $A_1$ and $A_2$ of finite measure $\nu_1$, 
$$
\lambda_1(A_1 \times A_2) = \int_{X_2} Q(y, A_1)Q(y, A_2)\; 
d\nu_2(y)
$$
Similarly,
$$
\lambda_2(B_1 \times B_2) = \int_{X_1} P(x, B_1)P(x, B_2)\; 
d\nu_1(x).
$$
The measures $\lambda_1$ and $\lambda_2$ are symmetric
$$
\lambda_1(A_1 \times A_2) = \lambda_1(A_2 \times A_1), \ \ 
 \lambda_2(B_1 \times B_2) = \lambda_2(B_2 \times B_1).
$$

\end{theorem}

\begin{proof} In the following equality we use the fact that $T_P$ and 
$T_Q$ form a symmetric pair of operators (see Theorem 
\ref{thm T_P contractive}).
$$
\ba 
\lambda_1(A_1 \times A_2) & = \int_{A_1} d\nu_1(x) (PQ)(x, A_2)\\
& = \int_{A_1} d\nu_1(x) T_P(Q(\cdot, A_2))(x)\\
& = \int_{X_1} d\nu_1(x) \chi_{A_1}(x) T_PT_Q(\chi_{A_2})(x)\\
& = \int_{X_2} d\nu_2(y) T_Q(\chi_{A_1})(y) T_Q(\chi_{A_2})(y)\\
& = \int_{X_2} Q(y, A_1)Q(y, A_2)\; d\nu_2(y).
\ea
$$ 
The other equality is proved analogously.

The fact that the measures $\lambda_1$ and $\lambda_2$ are 
symmetric obviously follows from the proved formulas.

\end{proof}

\section{RKHSs, symmetric measures, and transfer operators $R$}
\label{ssect L-set}

In this section, we give a description of measures $\rho$ that belong
to the set $L(\nu_1, \nu_2)$.   As a preliminary part, we remind first
the notion of a \textit{reproducing kernel Hilbert space (RKHS)}. 
The well known references to the theory of RKHS are 
 \cite{Aronszajn1950,  AronszajnSmith1957, 
 Adams_et_al1994, PaulsenRaghupathi2016, SaitohSawano2016}, 
 see also more recent  
 results and various applications in \cite{BerlinetThomas-Agnan2004, 
 AplayJorgensen2014,
  AlpayJorgensen2015, JorgensenTian2015, JorgensenTian-2016, 
 JT_2019, JT-2019}.

\subsection{RKHS and symmetric measures}
Let $S$ be an arbitrary  set, and let $K: S\times S \to \R$ be a \textit{positive  definite function}, i.e., the  function $K(s, t)$ has the property 
 $$
 \sum_{i, j =1}^N \alpha_i\alpha_j K(s_i, s_j) \geq 0
 $$
 which holds for any $N \in \N$ and for any $s_i \in S, \alpha_i \in \R, \
  i=1,..., N$.
We consider here real-valued functions. (For a complex-valued function 
$K$ some obvious changes must be made.)
 
\begin{definition} \label{def RKHS}
Fix $s \in S$ and denote by $K_s$ the function $K_s(t) = K(s, t)$ 
of one variable  $t \in S$.  Let $\mathcal K := 
\mbox{span}\{ K_s : s\in S\}$. 
The \textit{reproducing kernel Hilbert space (RKHS)} $\mc H(K)$ is the 
 Hilbert space obtained by completion of 
$\mathcal K$ with respect to the inner product defined on $\mathcal K$ by
$$
\left\langle\ \sum_i \alpha_i K_{s_i}, \sum_j\beta_j K_{s_j}\ 
\right\rangle_{\mc H(K)} := \sum_{i, j =1}^N \alpha_i\beta_j K(s_i, s_j)
$$
\end{definition}

It immediately follows from Definition \ref{def RKHS} that 
$$
\langle K(\cdot, s), K(\cdot, t)\rangle_{\mc H(K)} = K(s, t)
$$
(here we use the notation $K(\cdot, s) = K_s(\cdot)$).
More generally, this result can be extended to the following property 
that  characterizes functions from the RKHS $\mc H(K)$. For any $f \in  
 \mc H(K)$ and  any $s\in S$, one has 
\be\label{eq char prop RKHS}
f(s) = \langle f(\cdot), K(\cdot, s)\rangle_{\mc H(K)}. 
\ee
It suffices to check that (\ref{eq char prop RKHS}) holds for any function
from $\mathcal K$ and then extend it by continuity. 

One can prove that the following property determines functions from
the reproducing kernel Hilbert space $\mc H(K)$ constructed by a
 positive definite function $K$ on the set $S$. We formulate it as a 
 statement for  further references.

\begin{lemma}\label{lem criterion}
A function $f$ is in $\mc H(K)$ if and only if there exists a constant 
$C = C(f)$
such that for any $n \in \N$, any $\{s_1, ... ,s_n\}  \subset S$, and any
$\{\alpha_1, ... , \alpha_n\} \subset \R$, one has
\be\label{eq criterion RKHS}
\left(\sum_{i=1}^n \alpha_i f(s_i)\right) ^2 \leq C(f) 
\sum_{i, j =1}^n \alpha_i\alpha_j K(s_i, s_j).
\ee
\end{lemma}

We will need the following result.

\begin{lemma}\label{lem K and basis}
Let $K(s, t)$ be a positive definite function, and let $\mc H(K)$ be the
corresponding RKHS. Take an orthonormal basis $\{g_n(\cdot)\}$ in 
$\mc H(K)$. Then 
\be\label{eq ONB in RKHS}
K(s, t) = \sum_{n} g_n(s) g_n(t).
\ee
For every $s$, $(g_n(s))$ is in $\ell^2$. Moreover, every vector from
$\mc H$ can be  represented uniquely as $\sum_n c_n g_n(\cdot)$ where 
$(c_n) \in \ell ^2$. 
\end{lemma}

\begin{proof}
We have obvious equalities:
$$
K(\cdot, s) = \sum_n \langle K(\cdot, s), g_n\rangle_{\mc H(K)} 
g_n, \quad K(\cdot, t) = \sum_n \langle K(\cdot, t), g_n
\rangle_{\mc H(K)} g_n.
$$
Then we use \eqref{eq char prop RKHS} in the following computation:
$$
\ba
K(s, t) & = \langle K(\cdot, s), K(\cdot, t) \rangle_{\mc H(K)}  \\
& = \left\langle \sum_n \langle K(\cdot, s), g_n\rangle_{\mc H(K)} g_n, \ 
\sum_m \langle K(\cdot, t), g_m\rangle_{\mc H(K)} g_m
   \right\rangle_{\mc H(K)}\\
& =  \sum_n \langle K(\cdot, s), g_n\rangle_{\mc H(K)}
\;  \langle K(\cdot, t), g_n\rangle_{\mc H(K)}\\
& = \sum_{n} g_n(s) g_n(t).
\ea
$$

The second statement of the lemma is obvious.
\end{proof}

\begin{example}
For a standard measure space $(X, \A, \mu)$, define a symmetric
measure $\lambda$ on $X \times X$ by setting $ \lambda(A\times B) =
\mu(A\cap B)$ where $A, B \in \mc D(\mu)$. Then the corresponding
 symmetric operator $R$ must 
be  the \textit{identity operator}, $R(f) = f$, because of the relations
$$
\lambda(A\times B) = \int_X \chi_A R(\chi_B)\; d\mu,\quad 
\mu(A\cap B) =  \int_X \chi_A \chi_B\; d\mu.
$$
The function $k_\mu : (A, B)\mapsto \mu(A\cap B), A, B \in 
\mc D(\mu),$ is positive definite and defines a RKHS $\mc H_\mu$. 
It follows from Lemma \ref{lem criterion} that a function $F$ on 
$\mc D(\mu)$ is in $\mc H_\mu$ if and only if there exists a function
$f  \in L^2_{\mathrm{loc}}(\mu)$ such that 
\be\label{eq RKHS Example}
F(A) = \int_A f\; d\mu, \ \ \ A\in \mc D(\mu). 
\ee
Moreover, the correspondence between the functions
$F_A = (B \mapsto k_\mu(A, B))$  and $\chi_A$ can be extended to 
an isometry such that $|| F ||_{\mc H_\mu} = || f ||_{L^2(\mu)}$. More 
details
about this example and other examples can be found in 
\cite{BezuglyiJorgensen_2019, JT-2019, JT_2019}.
\end{example}

In what follows, we will work with an unbounded operator $R$ acting in 
$L^2(\nu)$ where  $(X, \B, \nu)$ is a $\sigma$-finite standard measures
space. Let $\mc F(\nu)$ be the linear space spanned by the characteristic 
functions $\chi_A$ with $A\in \mc D(\nu)$.
We need some assumptions about $R$.
\medskip 

\textbf{Assumption.} In this section the operator $R$ is assumed to 
have the following properties:

(i) $R$ is a densely defined unbounded (in general) operator in 
$L^2(\nu)$ such that $\mc F(\nu) \subset Dom(R)$ and $R(\mc F(\nu))
\subset L^2(\nu)$,

(ii) $R$ is a symmetric positive operator, i.e., $R(f) \geq 0$ whenever
$f \geq 0$,

(iii) $R$ is a positive definite operator, i.e., the quadratic form 
$f \mapsto \langle R(f), f\rangle_{L^2(\nu)}$   is non-negative,

(iv) the kernel of $R$ is trivial.
\medskip

Using Subsection \ref{ssect symm meas}, we can associate a symmetric
measure $\lambda$  on $(X\times X, \A \times \A)$ to  the operator $R$:
\be\label{eq symm meas by R}
\lambda(A\times B) = \int_X \chi_A R(\chi_B)\; d\nu = \lambda(B \times
A).
\ee

It follows from Assumption that $R$ is a symmetric operator, and 
therefore $R$ has a \textit{self-adjoint extension (Friedrichs extension)}. 
We denote it by 
$\wh R$. Note that $\wh R$  is  also positive. It is obvious that 
$\wh R = R$ on the subset $\mc F(\nu)$.  

It can be seen that the above assumptions about the operator $R$ and
its extension $\wh R$ lead to the following statement.

\begin{lemma} For $R$ and $\wh R$ as above, 
the subset $\wh R^{1/2}(Dom(\wh R^{1/2}))$ of $L^2(\nu)$ is a 
Hilbert space with the inner product $\langle \wh R^{1/2} f, \wh R^{1/2}g
\rangle_{L^2(\nu)}$.

\end{lemma}

Our main result of this section is formulated as follows.

\begin{theorem}\label{thm RKHS} Let $\wh R$ be the  Friedrichs 
extension of  a symmetric positive operator $R$ acting in  $L^2(\nu)$. 
Denote by $\lambda = \lambda(R)$ the symmetric measure defined by 
\eqref{eq symm meas by R}. Then 
$$
K(A, B):= \lambda(A \times B), \quad A, B \in \mc D(\nu)
$$
is a positive definite function defined on the set $\mc D(\nu)$ and 
generates the RKHS $\mc H(\lambda)$. The elements $\wh f(A)$ of 
$\mc H(\lambda)$ are signed measures on $\mc D(\nu)$ which are 
defined by the formula
\be\label{eq f via R}
\wh f(A) = \int_X R(x, A) f(x)\; d\nu(x) = \langle R(\chi_A), 
f\rangle_{L^2(\nu)}
\ee
where $f \in Dom(\wh R^{1/2})$. They are absolutely continuous 
with respect to $\nu$ and 
$$
\frac{d\wh f}{d\nu} = R(f).
$$
Moreover,
\be\label{eq isomorphism}
|| \wh R^{1/2}(f) ||_{L^2(\nu)} = || \wh f ||_{\mc H(\lambda)}, 
\ee
and therefore the map  $ \wh R^{1/2}(Dom(\wh R^{1/2}))  \ni 
\wh R^{1/2}(f)  \longrightarrow  
\wh f \in \mc H(\lambda)$, is an isometric isomorphism between these
Hilbert spaces. 
\end{theorem}

\begin{proof}
We first show that $K: (A, B) \mapsto \lambda(A \times B)$ is a positive 
definite function, $A, B \in \mc D(\nu)$. Indeed, for any collection of 
sets $A_1, ..., A_k$ from $\mc D(\nu)$ and any real numbers $\alpha_1, 
...,  \alpha_k$, we consider two families of functions: $\mc G(\nu)
 =\mathrm{span}\{K(\cdot, A) : A \in \mc D(\nu)\}$ and   
 $\mc F(\nu) =\mathrm{span}\{\chi_A : A \in \mc D(\nu)\}$.  
  Then, for 
  $$\varphi(x) = \sum_{i=1}^k \alpha_i \chi_{A_i}(x),$$
  we compute 
 \be\label{eq-lambda pdf}
 \ba
 \sum_{i, j}^k \alpha_i\alpha_j \lambda(A_i\times A_j) &= 
 \sum_{i, j}^k \alpha_i\alpha_j  \int_{X} \chi_{A_i} R(\chi_{A_j})
 \; d\nu\\
 & = \int_X \left(\sum_{i=1}^k \alpha_i \chi_{A_i}\right)
R \left(\sum_{j=1}^k \alpha_j \chi_{A_j}\right)\; d\nu\\
& = \langle \varphi, R(\varphi)\rangle_{L^2(\nu)}\\
& \geq 0
\ea 
 \ee
 because $R$ is a positive defined operator. Therefore, the function 
 $K : (A, B) \mapsto \lambda(A \times B)$ is positive definite and 
 defines a RKHS $\mc H(\lambda) = \mc  H$ as in  Definition 
 \ref{def RKHS}.

It follows from the definition of functions $\wh f$, see \eqref{eq f via R},
 that 
$$
\wh f(A) = \int_A R(f)\; d\nu, \quad f\in Dom(\wh R^{1/2}.
$$ 
This means that, for fixed $f$, the signed measure $\wh f(A)$ is absolutely
 continuous with respect to $\nu$ and  $R(f)$ is the Radon-Nikodym
  derivative  $\dfrac{d\wh f}{d\nu}$.   
  
The computation used in \eqref{eq-lambda pdf} shows that the following 
equalities are true:
 for the  function $\varphi\in \mc F(\nu)$ as above,
\be\label{eq norm in RKHS}
|| \varphi ||^2_{\mc H(\lambda)} = \sum_{i, j}^k \alpha_i\alpha_j 
\lambda(A_i\times A_j) = \langle \varphi, R(\varphi)\rangle_{L^2(\nu)}
= || \wh R^{1/2}(\varphi) ||^2_{L^2(\nu)}.
\ee
 
Next, we observe that $\wh \chi_A(\cdot) =
 K(\cdot, A)$ and $K(A, B) = \langle \chi_A, 
 \chi_B\rangle_{\mc H(\lambda)}$. By definition, the family  $\mc G(\nu)$ 
 is dense in  the RKHS
$\mc H(\lambda)$, and $\mc F(\nu)$ is dense in $L^2(\nu)$.
These two families are in  the one-to-one correspondence $K(\cdot, A) 
\longleftrightarrow \wh R^{1/2}(\chi_A)$ and satisfy the property 
\be\label{eq-iso rkhs}
|| K(\cdot, A)||_{\mc H(\lambda)} = || \wh R^{1/2}(\chi_A) 
||_{L^2(\nu)},\ \ \ A\in \mc D(\nu).
\ee
Our goal is to extend this isometry to the closures of $\mc 
G(\nu)$ and $\wh R^{1/2}(\mc F(\nu))$ and show
that the Hilbert 
spaces $\mc H(\lambda)$ and $Dom(\wh R^{1/2})$ are isometrically 
isomorphic.

Prove that $\wh f(A) $ is a function from $\mc H(\lambda)$. Let $f$ be
the corresponding  function from  $Dom(\wh R^{1/2})$ which defines
$\wh f$. By Lemma \ref{lem criterion}, the function 
$$
\wh f(A) = \langle R(\chi_A), f\rangle_{L^2(\nu)}
$$
 belongs to $\mc H(\lambda)$ if and only if, for $A_i,\in \mc  D(\nu)$ 
 and $\alpha_i \in \R,  i=1,..., k$,
$$
\ba
\left(\sum_{i= 1}^k \alpha_i \wh f(A_i) \right)^2 & =
\left(\sum_{i= 1}^k \alpha_i \langle R(\chi_{A_i}), f 
\rangle_{L^2(\nu)} \right)^2\\
&= \langle R(\varphi), f\rangle^2_{L^2(\nu)}\\
& =\langle\wh R^{1/2}(\varphi), \wh R^{1/2}(f)
\rangle^2_{L^2(\nu)}\\
& \leq  ||\wh R^{1/2}(f)||^2_{L^2(\nu)}
 ||\wh R^{1/2}(\varphi)||^2_{L^2(\nu)}\\
 & = ||\wh R^{1/2}(f)||^2_{L^2(\nu)}  \sum_{i, j}^k \alpha_i\alpha_j 
 \lambda(A_i\times A_j) 
\ea
$$
 where $\varphi(x) = \sum_{i=1}^k \alpha_i \chi_{A_i}(x)$. 
 Denoting $||\wh R^{1/2}(f)||^2_{L^2(\nu)}$ by $C(f)$, we can 
use the criterion formulated in   Lemma \ref{lem criterion}. Hence, 
$\wh f \in \mc H(\lambda)$. 

Furthermore, we note that, as the operator $\wh R^{1/2}$ is closed,
relation \eqref{eq-iso rkhs} can be extended by completion to the
 equality  
  $$
  ||\wh f ||_{\mc H(\lambda)} =  || \wh R^{1/2}(f) ||_{L^2(\nu)},
  $$
  which holds for all functions $\wh f$  from  $\mc H (\lambda)$. 
  
We show that $\wh f$, defined by \eqref{eq f via R}, satisfies 
the reproducing property  for the kernel
$K(A, B)$, i.e., $ \wh f(A) =    \langle K(\cdot, A), \wh 
f\rangle_{\mc H(\lambda)}$ where the inner product in $\mc H(\lambda)$
is defined as in Definition \ref{def RKHS}. For this, we use 
\eqref{eq isomorphism} and  the facts
that $R$ is symmetric and that $Dom(R) \subset Dom(\wh R^{1/2})$:
$$
\ba 
\wh f(A) & = \int_A R(f)(x) \; d\nu(x)\\
& = \langle R(f), \chi_A\rangle_{L^2(\nu)}\\
& = \langle \wh R^{1/2}(f), \wh R^{1/2}(\chi_A)\rangle_{L^2(\nu)}\\
& = \langle K(\cdot, A),  \wh f\rangle_{\mc H(\lambda)}, 
\ea
$$
where $A\in \mc D(\nu).$  
  
It remains to prove that the map $f \mapsto \wh f$ from $Dom(\wh R^{1/2})$ to $\mc H(\lambda)$ is 
onto, i.e.,   for every $F \in \mc H(\lambda)$, there exists a unique
function $f \in Dom(\wh R^{1/2})$ such that $f = \wh f$.
For this, we note first that 
$$
\left(\sum_{i =1}^n \alpha_i F(A_i)\right)^2 \leq 
C\sum_{i,j =1}^n \alpha_i \alpha_j K(A_i, A_j)= C  \langle \varphi, R(\varphi)\rangle_{L^2(\nu)}
$$
where $\varphi =\sum_{i=1}^n \alpha_i \chi_{A_i} \in 
Dom(\wh R^{1/2})$. Therefore, $F$ defines a linear continuous 
functional $\tau$ on the space $\mc F(\nu)$:
$$
\tau_F(\varphi) = \sum_{i =1}^n \alpha_i F(A_i),
$$
and 
$$
|\tau_F(\varphi)|^2 \leq C  \langle \varphi, R(\varphi)\rangle_{L^2(\nu)}.
$$
By the definition of the Friedrichs extension, the latter can be extended
to functions from $Dom(\wh R^{1/2})$, so that  
$$
|\tau_F(\varphi)|^2 \leq C || \wh R^{1/2} ||^2_{L^2(\nu)}.
$$
By the Riesz theorem, there exists a unique function $ f \in 
Dom(\wh R^{1/2})$ such that 
$$
\tau_F(g) = \langle \wh R^{1/2}(g), \wh R^{1/2}(f) 
\rangle_{L^2(\nu)}.
$$
Show that $ F = \wh f$ considered as functions from $\mc H(\lambda)$.  
Indeed, we use the reproducing property of the generating family of
functions $(K\cdot, A)$ and the fact that  $R$ is symmetric to deduce
that 
$$
\ba 
F(A) & = \langle F, K(\cdot, A)\rangle_{\mc H(\lambda)}\\
& = \langle \wh R^{1/2} f, \wh R^{1/2} \chi_A
\rangle_{L^2(\nu)}\\
& = \int_{X_1} R(f) \chi_A\; d\nu\\
& = \wh f(A), 
\ea
$$
where $A\in \mc D(\nu)$. As was mentioned above, the isometric property 
$ ||F ||_{\mc H(\lambda)} =  || \wh R^{1/2}(f) ||_{L^2(\nu)}$ 
can be proved by taking the closure of dense subsets in $\mc H(\lambda)$
and $Dom(\wh R^{1/2})$.

This proves that  $\mc H(\lambda)$ and $\wh R(Dom(\wh R^{1/2}))$ are 
isometric  Hilbert spaces.

\end{proof}

\begin{remark}
(1) We note that the definition of functions $\wh f$ from 
$\mc H(\lambda)$ as in \eqref{eq f via R} is analogous to the formula 
\eqref{eq RKHS Example} where the operator $R$ was the identity 
operator. 
On the other hand, Theorem \ref{thm RKHS} states the the RKHS $\mc
 H(\lambda)$ can
be defined as the set $\{ \wh f(A) = \langle \chi_A, 
R(f) \rangle_{L^2(\nu)} : f \in Dom(\wh R^{1/2}) \}$.

(2) Let $R$ be a symmetric  positive definite 
unbounded operator in $L^2(\nu)$  which
 determines a symmetric measure $\lambda$ (see Assumption). Let 
$\wh R$ be the Friedrichs extension of $R$. Then $\wh R$ is also a positive 
definite  self-adjoint  operator. Hence, in its turn, it defines a symmetric
 measure $\wh  \lambda$ on $X \times X$. It is not hard to see that, 
 in fact, $\lambda =  \wh \lambda$.
\end{remark}

\subsection{Spectral properties  and factorization of $R$}
Suppose that we have  now two $\sigma$-finite measure spaces
$(X_i, \A_i, \nu_i), i=1,2.$ 
We will use Theorem \ref{thm RKHS} to describe measures $\rho$ 
that belong to the set $L(\nu_1, \nu_2)$. We recall that $\rho$ is in 
$L(\nu_1, \nu_2)$ if there is a pair of finite transition kernels $P$ and
$Q$ such that \eqref{eq_rho first def} holds. 
Let $R$ be a positive  symmetric operator acting on
functions from $L^2(\nu_1)$ and satisfying  Assumption formulated 
before Theorem \ref{thm RKHS}. Beginning with $(X_1, \A_1, \nu_1, R)$, 
we define the symmetric measure $\lambda_1$ on $(X_1 \times X_1,
\A_1 \times \A_1)$ by setting 
\be\label{eq lambda_1}
\lambda_1(A \times B) = \langle \chi_A, R(\chi_B)
\rangle_{L^2(\nu_1)}, \ \ A, B \in \mc D(\nu_1),
\ee
and  then extend it to all Borel sets of $X_1\times X_1$.

Using the symmetric measure $\lambda_1$, construct the RKHS 
$\mc H(\lambda_1)$ as in Theorem \ref{thm RKHS}.   
Fix an orthonormal basis (ONB) $\{\wh k_n\}$ in the RKHS 
$\mc H(\lambda_1)$. As follows from \eqref{eq f via R},  there are
 functions $\{f_n\}\in Dom(\wh R^{1/2}) \subset L^2(\nu_1)$ 
 such that 
\be\label{eq_basis k_n}
\wh k_n(A) = \langle R(\chi_A), f_n\rangle_{L^2(\nu_1)} = 
\langle \chi_A, R(f_n)\rangle_{L^2(\nu_1)}, \ \ \ \  A \in 
\mc D(\nu_1).
\ee
Choose an ONB $\{\varphi_n\}$ in the Hilbert space $L^2(\nu_2)$.
We will consider the two bases in the Hilbert spaces $\mc H(\lambda_1)$ 
and $L^2(\nu_2)$ which  are formed by 
non-negative functions, i.e., $f_n \geq 0$ and $\varphi_n \geq 0$. 

\begin{definition}\label{def_P and Q via bases}
We use the objects $R, (\wh k_n)$, and  $(\varphi_n)$ to define two 
transition kernels $P : X_1 \times \A_2 \to [0, \infty)$ and $Q :
X_2 \times \A_1 \to [0, \infty)$. For this, we set
\be\label{eq def Q via basis} 
Q(y, A) := \sum_{n=1}^\infty \wh k_n(A) \varphi_n(y) = 
 \sum_{n=1}^\infty \langle\chi_A, R(f_n)\rangle_{L^2(\nu_1)}
  \varphi_n(y) , \quad A \in \mc D(\nu_1),
\ee
and
\be\label{eq def P via basis}
\ba 
P(x, B) & := \sum_{n=1}^\infty \left(\int_B \varphi_n \; d\nu_2\right)
R(f_n)(x)\\
& = \sum_{n=1}^\infty \langle \chi_B, \varphi_n 
\rangle_{L^2(\nu_2)} R(f_n)(x) , \quad B \in \mc D(\nu_2).
\ea
\ee

\end{definition}

\begin{remark} (1) 
We first note that Theorem \ref{thm RKHS} establishes
a correspondence between orthonormal bases in $\mc H(\lambda)$ and
$\wh R^{1/2}(Dom \wh R^{1/2}))$. This means that one can take the 
functions $(f_n)$ such that $(\wh R^{1/2}(f_n))$ is an ONB.
Moreover, since $Dom(\wh R) \subset Dom(\wh R^{1/2})$ and $R(g) = 
\wh R(g)$ on a dense subset, we can take $f_n$ such that $R(f_n)$ is 
defined.

(2) As mentioned in Lemma \ref{lem K and basis}, the sequence 
$(\wh k_n(A))$ is in $\ell^2 $ for every set $A \in \mc D(\nu_1)$. Hence 
The kernels $Q$ and $P$ take finite values on the sets $A$ and $B$ of
 finite measure.
 
 (3) One can show that the operators $T_P$ and $T_Q$ defined by the 
 kernels $P$ and $Q$ do not depend on the choice of bases $(\wh k_n)$ 
 and $(\varphi_n)$.
 
(4) It can be shown that $P(x, B)$ and $Q(y, A)$, which are defined on sets
of finite measure, can be extended to the measures $P(x, \cdot)$ 
and $Q(y, \cdot)$ on $(X_2, \A_2)$ and 
$(X_1, \A_1)$. Obviously, these measures are absolutely 
continues with respect to $\nu_2$ and $\nu_1$, respectively.
We leave the details for the reader. 

\end{remark}

Recall that the transition kernels $P$ and $Q$ generate a pair of linear
 operators $T_P$ and $T_Q$ acting on bounded Borel functions by
$$
T_Q(f)(y) = \int_{X_1} Q(y, dx) f(x), \qquad T_P(g)(y) = \int_{X_2}
 P(x, dy) g(y). 
$$
We will consider these operators in the corresponding $L^2$-spaces.
It follows from our previous results that $T_P$ and $T_Q$ are , 
in general, unbounded  densely defined operators such that 
$$
T_P : L^2(\nu_2) \to L^2(\nu_1), \qquad T_Q : L^2(\nu_1) \to
 L^2(\nu_2),
$$
see Theorem \ref{thm T_P contractive}. 
In Proposition \ref{prop P-Q via basis}, we give exact formulas for the
operators $T_P$ and $T_Q$ in the case when the kernels are taken
from Definition \ref{def_P and Q via bases}.

Recall that the subspace $\mc F(\nu_1)$ of simple functions from
 $L^2(\nu_1)$ belongs to the domain of $T_Q$, and $\mc F(\nu_2)
\subset L^2(\nu_1)$ is in the domain of  $T_P$.

\begin{proposition}\label{prop P-Q via basis}
(1) 
The  kernels $Q$ and $P$, defined by \eqref{eq def Q via basis} and
\eqref{eq def P via basis},  form a dual pair  of kernels  associated to the
 measures $\nu_1$ and $\nu_2$ (see Definition \ref{def P,Q first}).
 
 (2) The dual pair of transition kernels defines a measure $\rho$ on
 $X_1 \times X_2$ such that, for $A \in \mc D(\nu_1)$ and 
 $B \in \mc D(\nu_2)$,
 
 $$
 \rho(A \times B) = \int_{B} Q(y, A)\; d\nu_2(y) 
  = \sum_{n=1}^\infty \wh k_n(A) \langle \chi_B, \varphi_n
  \rangle_{L^2(\nu_2)}.
$$

 (3) The transition kernels $P$ and $Q$  generate the 
operators $T_P$ and $T_Q$ that form a symmetric pair of operators, i.e., 
$$
\langle g, T_Q(f)\rangle_{L^2(\nu_2)} = 
\langle T_P(g), f\rangle_{L^2(\nu_1)},
$$
and $T_P \subset (T_Q)^*$, $T_Q \subset (T_P)^*$ on the corresponding
domains. 
\end{proposition}

\begin{proof} (1) To prove the result, we need to show that, for any sets
$A \in \mc D(\nu_1)$ and  $B\in \mc D(\nu_2)$, the property
$$
\int_A P(x, B)\; d\nu_1(x)  = \int_B  Q(y, A)\; d\nu_2(y)
$$
holds. 
In the computation below, we use relations \eqref{eq_basis k_n}, 
\eqref{eq def Q via basis}, \eqref{eq def P via basis}, and the fact that 
$R$ is symmetric:

$$
\ba 
\int_B Q(y, A) \;  d\nu_2(y) & = \int_{X_2} \chi_B(y) 
\sum_{n=1}^\infty k_n(A) \varphi_n(y)\; d\nu_2(y)\\
&= \int_{X_2} \chi_B(y) \sum_{n=1}^\infty \langle R(\chi_A), f_n
\rangle_{L^2(\nu_1)}\varphi_n(y) \; d\nu_2(y)\\
& = \sum_{n=1}^\infty \langle \chi_B, \varphi_n\rangle_{L^2(\nu_2)}
\int_{X_1} R(\chi_A) f_n\; d\nu_1\\
& =\sum_{n=1}^\infty \langle \chi_B, \varphi_n\rangle_{L^2(\nu_2)}
\int_A R(f_n)\; d\nu_1\\
& = \int_A \left( \sum_{n=1}^\infty \langle \chi_B, \varphi_n\rangle_{L^2(\nu_2)} R(f_n)\right) \; d\nu_1\\
& = \int_A   P(x, B)\; d\nu_1(x).
\ea
$$

(2) It follows from (1) that the measure $\rho$ can be defined as we did
before:
$$
 \rho(A \times B)= \int_A P(x, B)\; d\nu_1(x)  = \int_B  Q(y, A)\; 
 d\nu_2(y).
$$
To prove the result we compute

$$
\ba 
\int_B  Q(y, A)\;d\nu_2(y) &= \int_{X_2} \chi_B(y) 
\left(\sum_{n=1}^\infty \wh k_n(A) \varphi_n(y)\right) \; d\nu_2(y)\\
&= \int_{X_2} \chi_B(y)  \left(\sum_{n=1}^\infty \langle R(\chi_A),
f_n\rangle_{L^2(\nu_1)} \varphi_n(y)\right) \; d\nu_2(y)\\
& = \sum_{n=1}^\infty \left(  \int_B \langle R(\chi_A),
f_n\rangle_{L^2(\nu_1)} \varphi_n(y)\right) \; d\nu_2(y) \\
& =  \sum_{n=1}^\infty \langle R(\chi_A),f_n\rangle_{L^2(\nu_1)} 
\langle\chi_B, \varphi_n\rangle_{L^2(\nu_2)}\\
& = \sum_{n=1}^\infty \langle\chi_A, R(f_n)
  \rangle_{L^2(\nu_1)} \langle \chi_B, \varphi_n\rangle_{L^2(\nu_2)}\\
 & =\sum_{n=1}^\infty \wh k_n(A) \langle \chi_B, \varphi_n
  \rangle_{L^2(\nu_2)}.
\ea
$$

(3) 
The proof of this statement is similar to that given in (1). We sketch it here.
 Since 
$$
T_Q(f)(y) = \int_{X_1} Q(y, dx) f(x) = \sum_{n=1}^\infty 
 \langle f,  R(f_n)\rangle_{L^2(\nu_1)} \varphi_n(y),
$$ 
and
$$
T_P(g)(y) = \int_{X_2} P(x, dy) g(y) = \sum_{n=1}^\infty
 \langle g,  \varphi_n \rangle_{L^2(\nu_2)} R(f_n)(x)
$$
we have
$$
\ba
\langle g, T_Q(f)\rangle_{L^2(\nu_2)} & = \int_{X_2} g(y) \left(
\sum_{n=1}^\infty \varphi_n(y)
\int_{X_1} R(f_n)(x) f(x) \; d\nu_1(x) \right)\ d\nu_2(y)\\
& = \sum_{n=1}^\infty \langle g, \varphi_n \rangle_{L^2(\nu_2)}
\langle R(f_n), f \rangle_{L^2(\nu_1)}\\
& =\int_{X_1} \left(\sum_{n=1}^\infty \langle g, \varphi_n
\rangle_{L^2(\nu_2)} R(f_n)(x)\right) f(x) \; d\nu_1(x)\\
& = \langle T_P(g), (f)\rangle_{L^2(\nu_1)}.
\ea
$$
\end{proof}

\begin{remark}\label{rem 7.10}
We proved in Proposition \ref{prop P-Q via basis} that the measures 
$P(x, \cdot)$ and $Q(y, \cdot)$ have the property
$$
\frac{P(x, dy)}{d\nu_2(y)} = \sum_{n=1}^\infty 
(Rf_n)(x) \varphi_n(y) =  \frac{Q(y, dx)}{d\nu_1(x)}.
$$
\end{remark}

\begin{theorem}\label{thm factorization} Let the kernels $P$ and $Q$ be
defined as in Definition \ref{def_P and Q via bases}. Then 
$\wh R = T_PT_Q$, i.e., the corresponding
operators $T_P$ and $T_Q$ factorize the operator $\wh R :  L^2(\nu_1) 
\to L^2(\nu_1)$.
\end{theorem}

\begin{proof} 
To obtain the result, it suffices to show that, for any sets $A_1, A_2 \in 
\mc D(\nu_1)$,
\be\label{eq proof of factor of R}
\langle \chi_{A_1}, T_PT_Q(\chi_{A_2})\rangle_{L^2(\nu_1)} =
\langle \chi_{A_1}, R(\chi_{A_2})\rangle_{L^2(\nu_1)}.
\ee
Then relation \eqref{eq proof of factor of R} can be extended to all
functions from  dense sets $\mc F(\nu_1)$ and $\mc F(\nu_2)$. 
Since the operator $T_PT_Q$ is self-adjoint, we obtain the equality 
$\wh R = T_PT_Q$.

In the computation below we will use Proposition \ref{prop P-Q via basis}
(3), relation \eqref{eq def Q via basis}, Lemma \ref{lem K and basis}, and
Theorem \ref{thm RKHS}. Recall that $(\varphi_n)$ is an ONB in 
$L^2(\nu_2)$.

$$
\ba 
\langle \chi_{A_1}, T_PT_Q(\chi_{A_2})\rangle_{L^2(\nu_1)} & =
\langle \chi_{A_1}, (T_Q)^*T_Q(\chi_{A_2})\rangle_{L^2(\nu_1)} \\
&= \langle T_Q(\chi_{A_1}), T_Q(\chi_{A_2})\rangle_{L^2(\nu_2)} \\
& = \int_{X_2} Q(y, A_1) Q(y, A_2)\; d\nu_2(y)\\
&= \int_{X_2} \left(\sum_n \wh k_n(A_1)\varphi_n(y) \right)  
\left(\sum_m \wh k_m(A_2)\varphi_m(y)\right) d\nu_2(y)\\
&= \sum_n \wh k_n(A_1) \wh k_n(A_2)\\
&= K(A_1, A_2)\\
&= \lambda_1(A_1 \times A_2)\\
& = \langle \chi_{A_1}, R(\chi_{A_2})\rangle_{L^2(\nu_1)}.
\ea
$$

\end{proof}

\begin{remark} In this remark, we show another proof of the Theorem
\ref{thm factorization}. It is based on the direct computation of
 $T_P(T_Q(f))$ using the formulas given in 
\eqref{eq def Q via basis} and \eqref{eq def P via basis} and Remark 
\ref{rem 7.10}. As in the proof of Theorem \ref{thm factorization}
we take an arbitrary  function $f \in \mc F(\nu_1)$ and note that 
$\wh R$ and $R$ coincide on this set.
$$
\ba 
T_P(T_Q(f))(x) & = \int_{X_2} P(x, dy) T_Q(f)(y)\\
&= \int_{X_2} P(x, dy) \int_{X_1} Q(y, dx) f(x)\\ 
&=  \int_{X_2} P(x, dy) \left( \sum_{n=1}^\infty 
 \langle f,  R(f_n)\rangle_{L^2(\nu_1)} \varphi_n(y)\right) \\
 & =\int_{X_2} \left(\sum_{m=1}^\infty (Rf_m)(x) \varphi_m(y) \right) 
\left( \sum_{n=1}^\infty 
 \langle f,  R(f_n)\rangle_{L^2(\nu_1)} \varphi_n(y)\right) d\nu_2(y)\\
 & =  \sum_{n=1}^\infty  \langle f,  R(f_n)\rangle_{L^2(\nu_1)} 
 (Rf_n)(x)\\
 & = \wh R^{1/2} \left( \sum_{n=1}^\infty   \langle \wh R^{1/2}(f),  
 \wh R^{1/2}(f_n)\rangle_{L^2(\nu_1)}   \wh R^{1/2}(f_n) \right)\\
& = \wh R(f).
\ea
$$
\end{remark}

\begin{lemma} Let $\lambda_1$ on $X_1 \times X_1$ be the symmetric
measure defined by the operator $R : L^2(\nu_1) \to L^2(\nu_1)$ as
in (\ref{eq lambda_1}). Then 
$$
\lambda_1(A_1 \times A_2) = \int_{X_2} Q(y, A_1)Q(y, A_2)\; d\nu_2(y),
\quad A_1, A_2 \in \mc D(\nu_1).
$$

In particular, 
$$
\lambda(A \times A) < \infty \ \Longleftrightarrow \ Q(\cdot, A) 
\in L^2(\nu_2),
$$ 
for  $A \in \mc D(\nu_1)$.
\end{lemma}

\begin{proof} It follows from  Theorem \ref{thm factorization}  that
$R(x, A) = QP(x, A) = \int_{X_2} P(x, dy) Q(y, A)$. We use it 
in the following computation:
$$
\ba 
\lambda_1(A_1 \times A_2)  & = \int_{X_1} \chi_{A_1}R(\chi_{A_2})\; 
d\nu_1\\
& = \int_{X_1} \chi_{A_1} \left(\int_{X_2} P(x, dy) Q(y, A_2)\right) 
d\nu_1\\
& = \int_{X_2}  \int_{A_1} Q(y, A_2) Q(y, dx)\; d\nu_2(y)\\
& =\int_{X_2}  Q(y, A_1) Q(y, A_2)\; d\nu_2(y).
\ea
$$

The particular case $A_1 = A_2$ gives the second statement of 
the lemma.
\end{proof}

\section{Measurable Bratteli diagrams}\label{sect meas BD}

In this section, we consider a measurable version of generalized 
Bratteli diagrams. The main differences are: (i) the levels $V_n$ of a
 measurable Bratteli diagram are formed by standard Borel spaces
 $(X_n, \A_n)$, 
 and (ii)  the sets of edges $E_n$ are Borel subsets of $X_n \times
 X_{n+1}, n \in \N_0$. An important subclass of measurable Bratteli
  diagrams is obtained when $X_n = X$ and the sets  $E_n$ are
   represented by  equivalence relations.
 
Our goal in this final section is to show that the main notions, and some 
results, from the case of discrete levels carry over to the more general case 
when instead the levels are measure spaces. This entails new developments 
in the analysis of path-space measures and the associated Markov 
processes.  Our treatment is brief, and full details are planned for a future 
paper.
    
 \subsection{Definitions, dynamics, and applications}\label{ssect basic MBD} 
 We begin this section with definitions of main objects.  
 
 \begin{definition}\label{def meas BD}
 Let $(X, \A)$ be an uncountable standard Borel space. Let $(X_n, 
 \A_n) : n  \geq 0)$ be a sequence of standard Borel spaces (each of 
 them is Borel isomorphic to $(X, \A)$). Suppose that $(E_n)$ is a
  sequence of Borel subsets of $X_n \times  X_{n+1}$ such that the 
projections $s_n : E_n \to X_n $ and $r_n : E_n \to X_{n+1}$ are onto 
Borel maps where  $s_n(e) = s_n(x, y) = x$ and $r_n(e) = r_n(x,y) =
y$ for every $e = (x, y) \in E_n$.
Then we call $\B = (X_n, E_n)$ a \textit{measurable Bratteli diagram}. 
The pair $(X_n, E_n)$ is called the $n$-th level of the   
measurable Bratteli diagram $\B$. 
 \end{definition}

\begin{remark}\label{rem_on_MBD}

 (1) We emphasize that all levels $X_n$ are formed by
the same Borel space $X$, so that we can (if necessary)  treat a subset 
$A$ of  $X$ as a subset of every $X_n$. Recall that this identification is 
used for
stationary Bratteli diagrams (standard and generalized diagrams). On the 
other hand,  the sets $E_n$ are not  identified.

(2) An important particular case of a measurable  Bratteli diagram is 
formed by \textit{countable Borel equivalence
 relations (CBER)} $E_n$, see e.g. \cite{JacksonKechrisLouveau_2002}, 
 \cite{DoughertyJacksonKechris_1994} for details on CBER. Briefly 
 speaking, this means 
 that (i) $E_n$ is a  symmetric Borel subset of $X_n \times  X_{n+1}$
  containing the diagonal, (ii) $(x,y) \in E_n$  and 
 $(y,z) \in E_n$ implies $(x,z) \in E_n$, and (iii) for any $x \in X_n$, 
 the set $\{y \in X_{n+1} : (x, y) \in E_n\}$ is countable.   
We note that the requirement  that $s_n, r_n$ are Borel onto maps 
is included in the definition of a measurable Bratteli diagram. This
 property is automatically true when $E_n$ is a CBER.
 
 (3) Generalizing the definition of $s_n$ and $r_n$, 
one can define two maps $r,s : \bigsqcup_n E_n \to \bigsqcup_n X_n$ 
by setting $s(e) = s_n(e) $ and $r(e) = r(e_n)$ where $e \in E_n$.
 Clearly, $r$ and $s$ are Borel maps since they coincide with $r_n$ and 
 $s_n$ on each $E_n$. This observation allows us to  
 simplify our notation and omit the subindex $n$ working with the maps
 $r, s$.
 
 (4) In some cases, it is useful to represent the set $E_n, n\in \N_0,$ as the union of  sections:
 $$
 \bigcup_{x\in X_n} s^{-1}(x) = E_n = \bigcup_{y\in X_{n+1}}
 r^{-1}(y).
 $$ 
 They are analogous to ``vertical'' and ``horizontal'' sections of a 
 subset in the product of two spaces. The fact that and $r_n$ in 
 Definition \ref{def meas BD} is an 
  onto map means that $\forall y \in X_{n+1}$ \ $\exists x \in X_n$ 
  such that $e = (x,y) \in E_n$. A similar property holds for $s_n$.
 
 \end{remark}

Suppose now that the standard Borel space $(X_0, \A_0)$, the initial 
level of a measurable Bratteli diagram $\B$, is 
endowed with  a $\sigma$-finite atomless Borel positive 
measure $\nu_0$ so that 
$(X_0,\A_0, \nu_0)$ is a standard measure space. 
We will consider  a sequence of  probability transition kernels $(P_n)$ 
(they are called \textit{Markov kernels}) on a 
measurable Bratteli diagram $\B$. Recall that a map $R : X \times \A
\to [0,1]$
is called a probability transition  kernel if: (a) $x \mapsto R(x, A)$ is a measurable map 
 for every set $A\in \A$, (b) for every $x \in X$, the function
$A \mapsto R(x, A)$ is a probability measure. 

\begin{definition}\label{def Markov kernels}
Let $\B = (X_n, E_n)$ be a measurable Bratteli diagram as in Definition 
\ref{def meas BD}. Let $(P_n)$ be a sequence of probability Markov
kernels such that $P_n : X_n \times \A_{n+1} \to \mathbb 
[0, 1]$,  $n \in \N_0$. We say that the sequence $(P_n)$ is  
\textit{consistent} with  the diagram $\B$, if, for every $x \in X_n$ and  
 $n \in  \N_0$, the map 
$x \mapsto P_n(x, \cdot)$ determines a probability measure on 
the set $r(s^{-1}(x_n)) \subset X_{n+1}$ such that $P(x_n, \cdot) \ll
\nu_{n+1}$, i.e., $P(x_n, r(s^{-1}(x_n)))  =1$ for all $x_n\in X_n$ 
and $n\in \N_0$. 

If all $(X_n, \A_n) = (X, \A)$ and $P_n = P, n\in \N_0$, then the
corresponding measurable  Bratteli diagram $\B$ is called
 \textit{stationary.}
\end{definition}

We are now in the setting that we used in Section 
\ref{sect Trans kernels}.
The only difference is that we have to work with a \textit{sequence} of
Markov kernels $(P_n)$ which generates a sequence of linear operators
$$
T_{P_n}(f) = \int_{X_{n+1}}  P_n(x_n, dx_{n+1}) f(y), \quad x_n
 \in X_n.
$$
Every $T_{P_n}$ sends  bounded Borel functions $f \in \mc F(X_{n+1}, 
\A_{n+1})$ into the set $\mc F(X_n, \A_n)$.

Similarly, we can define an action of operators $P_n$ on the set of
 measures $M(X_n, \A_n)$. More precisely, we set
\be\label{eq_nuP(n)}
(\mu P_n)(f)  = \int_{X_n} P_n(f) \; d\mu, \quad f \in \mc F(X_{n+1},
\A_{n+1}), \ \mu\in M (X_{n+1}, \A_{n+1}).
\ee
Then $P_n :  M(X_{n}, \A_{n}) \to M(X_{n+1}, \A_{n+1})$.
Note that $\mu P_n$ is a probability measure if and only if $\mu$ is 
probability.

\begin{theorem}\label{prop Markov kernels}
Let $\nu_0$ be a $\sigma$-finite measure on the initial level 
$(X_0, \A_0)$
of a measurable Bratteli diagram $\B = (X_n, E_n)$. Suppose that
$(P_n)$ is a sequence of Markov kernels consistent with $\B$. Then: 

(a) there exists a sequence of $\sigma$-finite measures $(\nu_n)$, 
where $\nu_n \in M(X_n, \A_n)$, such that 
$$
\nu_{n} = \nu_{n+1} Q_n, \ \ \nu_{n+1} = \nu_n P_n, \quad n \in \N_0,
$$ 

(b) there exists a sequence of transition kernels $(Q_n)$, where
$Q_n : X_{n+1} \times \A_n \to [0, \infty)$, such that 
$(P_n, Q_n)$ form a dual pair for all $n$, i.e.,  
$$
d\nu_n(x_n)P_n(x_n, dx_{n+1}) = d\nu_{n+1}(x_{n+1}) 
Q(x_{n+1}, dx_n),
$$ 

(c) there exists a sequence of measures $(\rho_n)$, which supported by 
the sets $E_n$, such that 
$$
\rho_n(A\times B) = \int_A d\nu_n(x_n) P(x_n, B) = 
\int_B d\nu_{n+1}(x_{n+1}) Q_n(x_{n+1}, A), 
$$
where $A \times B \in \A_n \times \A_{n+1}$,

(d) for all $n \geq 0$, $P_n(x_n, \cdot) \ll \nu_{n+1}$ $\nu_n$-a.e. 
and $Q_n(x_{n+1}, \cdot) \ll \nu_{n}$ $\nu_{n+1}$-a.e. 
\end{theorem}

\begin{proof} 
The proposition follows from the results proved in Section 
\ref{sect Trans kernels}, see Lemma \ref{lem_P,Q abs cont},
 Theorem \ref{thm_P determines Q},
Corollary \ref{cor prob P}, and
Corollary \ref{cor rho determ P Q}. We mention here  a few 
moments that are needed for further developments. 

Beginning with the initial measure $\nu_0$ and using the sequence of
Markov kernels $(P_n)$ we define first a measure $\nu_1 = \nu_0 P_0$,
and then inductively $\nu_{n+1} = \nu_n P_n$ for all $n$. 

Next, we define the measures $\rho_n$ by setting 
\be \label{eq_def rho_n}
d\rho_n(x_n, x_{n+1}) = d\nu_n(x_0) P_n(x_n, dx_{n+1}).
\ee
 Since
$(P_n)$ is consistent with the measurable Bratteli diagram and 
$E_n = \bigcup_{x_n} E_n(x_n) = \bigcup_{x\in X_n} s^{-1}(x_n)$, 
 we see from \eqref{eq_def rho_n} that $\rho_n$ is defined on $E_n$.
Moreover, the measures $\nu_n$ and $\nu_{n+1}$ are projections
of $\rho_n$ on $X_n$ and $X_{n+1}$, i.e., $s(\rho_n) = \nu_n$ and
$r(\rho_n) = \nu_{n+1}$.

In order to define the dual transition kernel $Q_n$, we use the formula
$$
Q_n(x_{n+1}, A) = \frac{\rho_n(A, dx_{n+1})}{d\nu_{n+1}
(x_{n+1})},  \quad A \in \A_n,
$$
(see also Theorem \ref{thm_P determines Q} for the definition of
$Q_n$).
 \end{proof}
 
\begin{remark} The results formulated in Theorem 
\ref{prop Markov kernels} can be obtained from another equivalent 
setting. Namely, let $\B$ be a measurable Bratteli diagram with levels
$(X_n, \A_n)$ and $E_n$.  Suppose
 that a sequence of  $\sigma$-finite measures $(\nu_n)$ is chosen on 
 the spaces $(X_n, \A_n)$. Take measures $\rho_n$, supported by 
 $E_n$, such that $\rho_n \in L(\nu_n, \nu_{n+1})$. Then there are
 probability transition kernels $P_n$ that satisfy Theorem 
 \ref{prop Markov kernels}, see Corollaries \ref{cor prob kernel} and
 \ref{cor rho determ P Q}.

\end{remark} 
 
Using the objects defined in Theorem  \ref{prop Markov kernels}, we       
can determine the path-space measure $\mathbb P$ on the space 
$\mc X_\B$  of paths  of a measurable  Bratteli diagram $\B$. 

\begin{definition}\label{def path space meas}
Let $\B$ be a measurable Bratteli diagram suc as in  Definitions 
\ref{def meas BD} and \ref{def Markov kernels}. Define the path-space
of $\mc X_\B$ to be a 
subset of $E_0 \times E_1 \times \cdots $ formed by 
sequences $(e_i)$ such that $s(e_{i+1}) =r(e_i), i \geq 0$. The space
$\mc X_\B$ can be equivalently described as a set of sequences 
$(x_i), i \geq 0$ where $x_i = s(e_i)$. Denote by $\mc X_{\B}(x)$
the subset of $X_\B$ consisting of all paths that begin at $x \in X_0$.

Let $\nu_0$ be a $\sigma$-finite measure on $(X_0, \A_0)$ and 
$(P_n)$ the sequence of probability transition kernels. Define a 
probability measure $\mathbb P_x$ for $x \in X_0$. Let $f = 
f(x_1, \cdots, f_n)$ be a Borel bounded function on $\mc X_{\B}(x)$ 
depending on the first $n$ coordinates (cylinder function). We set
\be\label{eq mathbb P_x}
\int f\; d\mathbb P_x = \int_{X_1} \int_{X_2}\cdots \int_{X_n} 
P_0(x, dx_1) P_1(x_1, dx_2) \cdots P_{n-1} (x_{n-1}, dx_n)\; f(x_1,
 \cdots, x_n). 
\ee
Then $\mathbb P_x$ is extended to a probability measure on the set
$\mc X_{\B}(x)$. 

Finally, we define a measure $\mathbb P$ on $\mc X_\B$ by setting
\be\label{eq mathbb P}
\mathbb P = \int_{X_0} \mathbb P_x\; d\nu_0(x).
\ee
In more details, the measure $\mathbb P$ is computed as follows. 
$$
\ba 
\int_{\mc X_\B} g\; d\mathbb P & =  \int_{X_0}\mathbb P_{x}
(g)\; d\nu_0(x_0)\\ 
& =\int_{X_0} \int_{X_1} \cdots \int_{X_n} d\nu_0(x_0) 
P_0(x_0, dx_1)  \cdots P_{n-1} (x_{n-1}, dx_n)\; 
g(x_0, \cdots, x_n)
\ea
$$ 
where $g$ is a cylinder function. Extending  $\mathbb P$ to all Borel 
bounded functions, we obtain a measure which is
called a \textit{path-space measure} on $\mc X_\B$. 
\end{definition}

\begin{remark} Let $\xi_n : \mc X_\B \to (X_n, \A_n)$ be a random 
variable such that $\xi_n(x_i) = x_n$, $n\in \N$. The values of the 
path-space measure $\mathbb P$
on the cylinder sets can be also found by the formula:
\be\label{eq meas P_x} 
 \mathbb P_x(\xi_1 \in A_1, ... , \xi_n \in A_n) =
 \int_{A_n}\cdots \int_{A_1} P_0(x,  dx_1)  \cdots  P_{n-1} (x_{n-1}, dx_n) 
 \ee
 In particular, $\mathbb P_x(\xi_1 \in B) = P(x, B)$.
 
 As a consequence of \eqref{eq R and chi_A},
 we have also the following formula:
\be\label{eq meas P_x 2}
 \mathbb P_x(\xi_1 \in A_1, ... , \xi_n \in A_n) = 
 P_0(\chi_{A_1} P_1(\chi_{A_2}P_2(\ \cdots\  P_{n-1}(\chi_{A_n}) 
 \cdots )))(x).
\ee

\end{remark}

\begin{lemma} For a measurable Bratteli diagram $\B$ as above, let 
$\xi_n : \mc X_\B \to X_n$ be the random variable 
defined by $\xi_n(x) = x_n$ where $x = (x_i)$. Then the transition 
probability kernels $P_n$ and $Q_n$ can be restored as follows:
$$
\mathbb E(\xi_{n+1} \in B \ |\ \xi_n = x) = P_n(x, B), \quad 
\mathbb E(\xi_{n} \in A \ |\ \xi_{n+1} = y) = Q_n(y, A)
$$
where $\mathbb E$ is the mathematical expectation with respect 
to the path-space measure $\mathbb P$, $A \in \A_n$, $B \in 
\A_{n+1}$, and $n \in \N_0$.
\end{lemma}

As in the case of discrete Bratteli diagrams, we can define the notion of 
tail equivalence relation $\mc R$: two paths $x = (x_n)$ and $y = (y_n)$
 from 
the path-space $\mc X_\B$ are called tail equivalent if there exists some 
$m$ such that $x_n = y_n$ for all $n \geq m$. Clearly, $\mc R$ is a Borel
equivalence relation with uncountable classes. 

Now we define an analogue of the tail invariant measure in the case of
measurable Bratteli diagrams. 

\begin{definition}
Let $\mathbb P$ be a path-space measure on the set $\mc X_\B$ where
$\B$ is a measurable Bratteli diagram, see Definition \ref{def path 
space meas} for notation. The measure $\mathbb P$ is called \textit{tail 
invariant} if for every $n \in \N$ and every $A \subset X_n$, $\nu_n
(A) >0$ the function $\mathbb P(\xi_n \in A \ |\ \xi_0 =x) = 
\mathbb P_x(\xi_n \in A)$ does not depend on $x\in X_0$ $\nu_0$-a.e.

\end{definition}
The existence of tail invariant measures follows from the following 
observation.

\begin{lemma}
Let $\B$ be a measurable Bratteli diagram with levels represented by 
probability measure spaces $(X_n, \A_n, \nu_n)$. Suppose that the 
probability transition kernels are defined by $P_n(x, B) = \nu_{n+1}(B)$,
$B\in \A_{n+1}$. Define the measures $\mathbb P_x$ and $\mathbb P$
as in \eqref{eq mathbb P_x} and \eqref{eq mathbb P}. Then $\mathbb P$
is  a tail equivalent measure.
\end{lemma}

The proof is obvious because $\mathbb P_x(\xi_n \in B) = \nu_{n+1}(B)$.

\subsection{Harmonic functions and graph Laplacians on measurable Bratteli diagrams}

Let $\B$ be a measurable Bratteli diagram defined in Subsection 
\ref{ssect basic MBD}. Denote the $n$-th level of $\B$ by 
$(X_n, \A_n, \nu_n)$,  a $\sigma$-finite measure space.  
Suppose that we have two sequences of finite transition kernels $P =
(P_n)$ and $Q = (Q_n)$ compatible with the diagram 
$\B$ and such that $P_n : X_n \times \A_{n+1} \to [0, \infty)$ and 
$Q_n : X_{n+1} \times \A_{n} \to [0, \infty)$. Assume also that 
they satisfy the following conditions:

(i) $(P_n, Q_n)$ is a dual pair, i.e., they satisfy the equality 
$$d\nu_n(x_n)
P_n(x_n, dx_{n+1}) = d\nu_{n+1}(x_{n+1}) Q_n(x_{n+1}, dx_n),\ 
n \in \N_0,$$ which determines a measure $\rho_n$ on $X_n \times
X_{n+1}$,

(ii) the Borel function $c_n(x_n) = P_n(x_n, X_{n+1})$ and 
$d_{n+1}(x_{n+1}) =  Q_n(x_{n+1}, X_n)$ are locally integrable. 

It follows from the definition of $(P_n, Q_n)$ that
$$
\int_{X_n} c_n(x_n)\; d\nu_n(x_n) = \int_{X_{n+1}} d_{n+1}
(x_{n+1})\; d\nu_{n+1}(x_{n+1}), \ \ n \in \N_0. 
$$

Let $F = (F_n)$ be an arbitrary Borel function such that  every 
$F_n \in L^2(\nu_n)$. 
Define the actions of $P$ and $Q$ on the functions $F$:
$$
P(F) = (T_{P_n} (F_{n+1})), \quad T_{P_n}(F_{n+1})(x_n) =
 \int_{X_{n+1}}P_n(x_n, dy) F_{n+1}(y),
$$
$$
Q(F) = (T_{Q_n} (F_{n})), \quad T_{Q_n}(F_{n})(x_{n+1})  =
 \int_{X_{n}} Q_n(x_{n+1}, dy) F_{n}(y).
$$
Recall that by Theorem 
\ref{thm T_P contractive} we have the following properties: 
(a) if $P_n$ and $Q_n$ are probability kernels, then $ T_{P_n}$ and 
$T_{Q_n}$ are contractive operators such that  $T_{P_n} : 
L^2(\nu_{n+1}) \to L^2(\nu_{n}) $,  $T_{Q_n} : 
L^2(\nu_{n}) \to L^2(\nu_{n+1}) $, and $T_{P_n}^* = T_{Q_n}$;
 (b) if $P_n$ and $Q_n$ are finite transition kernels, then $T_{P_n}$
  and  $T_{Q_n}$ constitute a dual pair of
densely defined   operators between the two $L^2$-spaces, provided 
the conditions $T_{P_n} (\mc D(\nu_{n+1}) \subset L^2(\nu_{n})$,
$T_{Q_n} (\mc D(\nu_{n}) \subset L^2(\nu_{n+1})$ hold. Moreover,
$T_{P_n} \subset (T_{Q_n})^*$ and $T_{Q_n} = (T_{P_n})^*$.

In the next definition, we apply the approach  used in Section 
\ref{sec Laplace}. It is worth noting that every finite transition kernel
can be normalized, if necessary. 

\begin{definition}\label{def M and Delta}
 Let $(P_n)$ and $(Q_n)$ be two sequences of 
finite transition kernels satisfying the properties (i) and (ii) described above. 
Let $F = (F_n)$ be a sequence of bounded Borel functions. Define a new
probability transition kernel $M= (M_n)$ by the formula
\be\label{eq_Markov op}
M(F) = (M(F)_n), \quad M(F)_n = \frac{1}{2}\left(\frac{1}{c_n} 
P_n(F_{n+1})+ \frac{1}{d_n} Q_{n-1}(F_{n-1})\right),
\ee
or $M(F)_n = \wt P_n(F_{n+1}) + 
\wt Q_{n-1}(F_{n-1})$
where $2c_n\wt P_n = P_n$ and $2d_n\wt Q_{n-1} = Q_{n-1}$.

The operator $\Delta$ such that
\be\label{eq_Lapl op}
(\Delta F)_n =  (c_n + d_n)F_n - P_n(F_{n+1}) - Q_{n-1}(F_{n-1})
\ee
is called a \textit{Laplace operator}.

A function $H = (H_n)$ such that $M(H)_n = H_n$ (or, equivalently,
$\Delta (H)_n =0$ is called a \textit{harmonic function}.
\end{definition}

From Definition \ref{def M and Delta}, we see that $M = (M_n)$ has the property $M_n(\mathbbm 1)=1$.
 
\begin{lemma} Formulas \eqref{eq_Markov op} and 
\eqref{eq_Lapl op} can be expanded as follows:
$$
M(F)_n (x_n)= \frac{1}{2}\left(\frac{1}{c_n(x_n)} \int_{X_{n+1}} 
P_n(x_n, dy)F_{n+1}(y) + \frac{1}{d_n(x_n)} \int_{X_{n}} 
Q_{n-1}(x_{n}, dy)F_{n-1}(y) \right)
$$
and
$$
(\Delta F)_n(x_n) = \int_{X_{n+1}} P_n(x_n, dy)(F_n(x_n) -
F_{n+1}(y)) + \int_{X_{n-1}} Q_{n-1}(x_n, dy)(F_n(x_n) -
F_{n-1}(y)).
$$
If $c_n = d_n$, then
$$
(\Delta F)_n = 2(c_n F_n - (M F)_n) = 2c_n(Id - M)(F)_n.
$$
\end{lemma}

In what follows we show another approach for the definition of a 
Laplace operator. Our main references here are 
\cite{BezuglyiJorgensen_2019a, JorgensenPearse_2019, 
BezuglyiJorgensen_2019}. 

\begin{remark}
Let $(\mc X_\B, \mathbb P)$ be the path-space of a measurable 
Bratteli diagram equipped with a path-space measure $\mathbb P$.
It is a standard measure space with a $\sigma$-finite measure, in 
general.  Consider the Cartesian product $\mc X_\B \times \mc X_\B$
and take a symmetric measure $\lambda$ on $\mc X_\B \times 
\mc X_\B$. The support of $\lambda$ is a symmetric subset 
$Z(\lambda) \subset
\mc X_\B \times \mc X_\B$. We discussed symmetric measures  in 
Sections \ref{ssect Kernels on L^2} and \ref{ssect L-set}. Then
$\lambda$ admits a disintegration 
$$
\lambda = \int_{\mc X_B} \lambda_x \; d\mathbb P(x).
$$
Our choice of $\lambda$ is restricted by the property that $c(x) = 
\lambda_x(\mc X_\B)$ must be finite for a.e. $x$.

Define a linear operator $\ol \Delta$ acting on the space of bounded
 Borel functions on $(\mc X_\B, \mathbb P)$:
\be\label{eq-Delta bar} 
\ol \Delta (f)(x) = \int_{\mc X_\B} (f(x) - f(y)) \; d\lambda_x(y)= 
c(x) \left(f(x) - \int_{\mc X_\B} f(y) \; d\tau_x(y)\right)
\ee
where $d\tau_x(y) = \dfrac{1}{c(x)} d\lambda_x(y)$. 

The operator $\mc T : f \mapsto \int_{\mc X_\B} f \; d\tau_x$ is a 
Markov operator. 
 \end{remark}

We define now the \textit{finite energy space} $\mc H_E$. 
Suppose that a measurable Bratteli diagram $\mc B$ is defined by an
initial distribution $\nu_0$ and a sequence $(P_n)$ of Markov kernels.
These objects generate a sequence of  dual kernels $(Q_n)$ and measures
$(\rho_n)$. Recall that $\rho_n$ is the measure  which is defined by
the equality $\rho(dx, dy) = d\nu_n(x)P_n(x, dy) = d\nu_{n+1}(y)
Q(y, dx)$, $x \in X_n, y \in X_{n+1}$. 

\begin{definition}\label{def fin energy}
 For a measurable Bratteli diagram, let $F = (F_n)$ be a
 Borel function on the path-space $\mc X_{\mc B}$ of $\B$. It is said that
a function $F$ has \textit{finite energy} and belongs to the finite energy
space $\mc H_E$ if
\be\label{eq-norm fin en}
|| F ||^2_{\mc H_E}= \sum_{n=0}^\infty \ \iint_{X_n \times X_{n+1}}
 (F_n (x) -  F_{n+1}(y))^2\; d\rho_{n}(x, y) < \infty \ \ x\in X_n, y\in 
 X_{n+1}.
\ee

\end{definition}

As a matter of fact, the elements of $\mc H_E$ are the classes of functions: 
two functions $F$ and $F'$ are identified if $F-F' = \mathrm{const}$. 
It can be checked that $\mc H_E$, equipped with the norm as in 
\eqref{eq-norm fin en}, is a Hilbert space.

\begin{proposition} \label{prop meas energy}
Let $\B$ be a measurable Bratteli diagram with levels
represented by measure spaces  $(X_n, \A_n, \nu_n)$. Suppose that  
$(P_n)$ and $(Q_n)$ are  sequences of dual probability transition kernels
 compatible with $\B$. Then a function $F = (F_n)$ on the path-space of 
 $ \B$ belongs to the finite energy space $\mc H_E$ if and only if 
$$
\sum_{n\geq 0} \ \left( || F_n ||^2_{L^2(\nu_n)} - 
2 \langle F_n, T_{P_n}(F_{n+1})\rangle_{L^2(\nu_n)}  
+ || F_{n+1} ||^2_{L^2(\nu_{n+1})} \right)< \infty.
$$
\end{proposition}

\begin{proof}
We first remind that if $P_n$ is a probability transition kernel, then 
$T_{P_n} : L^2(\nu_n) \to L^2(\nu_{n+1})$ is a contractive operator,
see Theorem \ref{thm T_P contractive}. Expanding the expression in 
\eqref{eq-norm fin en}, we obtain 
$$
\ba 
|| F ||^2_{\mc H_E} &= \sum_{n=0}^\infty \ \iint_{X_n \times 
X_{n+1}} \left(F_n^2(x) -2 F_n(x) F_{n+1}(y) + F_{n+1}^2(y) \right)
\; P_n(x, dy) d\nu_n(x)\\
& =
\sum_{n=0}^\infty \left[ \int_{X_n}\int_{X_{n+1}} F_n^2(x)\;
P_n(x, dy) d\nu_n(x) - 2\int_{X_n}\int_{X_{n+1}}  F_n(x) F_{n+1}(y)    \; P_n(x, dy) d\nu_n(x) \right.\\
 & \ \ \ \ \ \ \left. +\int_{X_n}\int_{X_{n+1}} F_{n+1}^2(y) \; P_n(x, dy) 
 d\nu_n(x) \right]\\
&= \sum_{n=0}^\infty \left(\int_{X_n} F_n^2(x) \; d\nu_n(x) -
\int_{X_n}  F_n(x) T_{P_n}(F_{n+1})(x) \; d\nu_n(x) \right.\\
& \ \ \ \ \ \  +\left.\int_{X_n}\int_{X_{n+1}} F_{n+1}^2(y) \; Q_n(y, dx)  
d\nu_{n+1}(y)  \right)\\
& = \sum_{n=0}^\infty \left( || F_n ||^2_{L^2(\nu_n)} - 
2 \langle F_n, T_{P_n}(F_{n+1})\rangle_{L^2(\nu_n)}  
+ || F_{n+1} ||^2_{L^2(\nu_{n+1})} \right)\\
\ea
$$
\end{proof}

\textbf{Acknowledgments.} The authors are pleased to thank colleagues 
and collaborators, especially members of the seminars in Mathematical 
Physics  and Operator Theory at the University of Iowa, where versions of 
this work  have been presented. We acknowledge very helpful conversations 
with  among others Professors Paul Muhly, Wayne Polyzou; and 
conversations at  distance with Professors Daniel Alpay, and his colleagues 
at both Ben Gurion  University, and Chapman University. 
We thank Dr. Feng Tian for help with the graphics in the Figures. 

\newpage

\bibliographystyle{alpha}
\bibliography{references_MBD.bib}

\newcommand{\etalchar}[1]{$^{#1}$}
\def\ocirc#1{\ifmmode\setbox0=\hbox{$#1$}\dimen0=\ht0 \advance\dimen0
  by1pt\rlap{\hbox to\wd0{\hss\raise\dimen0
  \hbox{\hskip.2em$\scriptscriptstyle\circ$}\hss}}#1\else {\accent"17 #1}\fi}
\begin{thebibliography}{LZW{\etalchar{+}}16}

\bibitem[Ada20]{Adams2020}
Jason~R. Adams.
\newblock {\em Plant {S}egmentation by {S}upervised {M}achine {L}earning
  {M}ethods and {P}henotypic {T}rait {E}xtraction of {S}oybean {P}lants {U}sing
  {D}eep {C}onvolutional {N}eural {N}etworks with {T}ransfer {L}earning}.
\newblock ProQuest LLC, Ann Arbor, MI, 2020.
\newblock Thesis (Ph.D.)--The University of Nebraska - Lincoln.

\bibitem[AFMP94]{Adams_et_al1994}
Gregory~T. Adams, John Froelich, Paul~J. McGuire, and Vern~I. Paulsen.
\newblock Analytic reproducing kernels and factorization.
\newblock {\em Indiana Univ. Math. J.}, 43(3):839--856, 1994.

\bibitem[AJ14]{AplayJorgensen2014}
Daniel Alpay and Palle Jorgensen.
\newblock Reproducing kernel {H}ilbert spaces generated by the binomial
  coefficients.
\newblock {\em Illinois J. Math.}, 58(2):471--495, 2014.

\bibitem[AJ15]{AlpayJorgensen2015}
Daniel Alpay and Palle Jorgensen.
\newblock Spectral theory for {G}aussian processes: reproducing kernels,
  boundaries, and {$L^2$}-wavelet generators with fractional scales.
\newblock {\em Numer. Funct. Anal. Optim.}, 36(10):1239--1285, 2015.

\bibitem[Ana11]{Anandam_2011}
Victor Anandam.
\newblock {\em Harmonic functions and potentials on finite or infinite
  networks}, volume~12 of {\em Lecture Notes of the Unione Matematica
  Italiana}.
\newblock Springer, Heidelberg; UMI, Bologna, 2011.

\bibitem[Ana12a]{Anandam2012}
Victor Anandam.
\newblock Integral representation of positive separately harmonic functions in
  a product tree.
\newblock {\em J. Anal.}, 20:91--101, 2012.

\bibitem[Ana12b]{Anandam_2012}
Victor Anandam.
\newblock Subordinate harmonic structures in an infinite network.
\newblock In {\em Complex analysis and potential theory}, volume~55 of {\em CRM
  Proc. Lecture Notes}, pages 301--314. Amer. Math. Soc., Providence, RI, 2012.

\bibitem[Aro50]{Aronszajn1950}
N.~Aronszajn.
\newblock Theory of reproducing kernels.
\newblock {\em Trans. Amer. Math. Soc.}, 68:337--404, 1950.

\bibitem[AS57]{AronszajnSmith1957}
N.~Aronszajn and K.~T. Smith.
\newblock Characterization of positive reproducing kernels. {A}pplications to
  {G}reen's functions.
\newblock {\em Amer. J. Math.}, 79:611--622, 1957.

\bibitem[Bau19]{Baudot2019}
Pierre Baudot.
\newblock The {P}oincar\'{e}-{S}hannon machine: statistical physics and machine
  learning aspects of information cohomology.
\newblock {\em Entropy}, 21(9):Paper No. 881, 39, 2019.

\bibitem[BDK06]{BezuglyiDooleyKwiatkowski2006}
S.~Bezuglyi, A.~H. Dooley, and J.~Kwiatkowski.
\newblock Topologies on the group of {B}orel automorphisms of a standard
  {B}orel space.
\newblock {\em Topol. Methods Nonlinear Anal.}, 27(2):333--385, 2006.

\bibitem[BJa]{BezuglyiJorgensen_2019}
Sergey Bezuglyi and Palle E.~T. Jorgensen.
\newblock Laplace operators in finite energy and dissipation spaces.
  \textit{arXiv:1903.09572}.

\bibitem[BJb]{BezuglyiJorgensen_2018}
Sergey Bezuglyi and Palle E.~T. Jorgensen.
\newblock Symmetric measures, continuous networks, and dynamics.
  \textit{arXiv:1812.00081}.

\bibitem[BJ15]{BezuglyiJorgensen2015}
S.~Bezuglyi and Palle E.~T. Jorgensen.
\newblock Representations of {C}untz-{K}rieger relations, dynamics on
  {B}ratteli diagrams, and path-space measures.
\newblock In {\em Trends in harmonic analysis and its applications}, volume 650
  of {\em Contemp. Math.}, pages 57--88. Amer. Math. Soc., Providence, RI,
  2015.

\bibitem[BJ19a]{BezuglyiJorgensen_2019a}
Sergey Bezuglyi and Palle E.~T. Jorgensen.
\newblock Graph {L}aplace and {M}arkov operators on a measure space.
\newblock In {\em Linear systems, signal processing and hypercomplex analysis},
  volume 275 of {\em Oper. Theory Adv. Appl.}, pages 67--138.
  Birkh\"{a}user/Springer, Cham, 2019.

\bibitem[BJ19b]{BJ-2019}
Sergey Bezuglyi and Palle E.~T. Jorgensen.
\newblock Monopoles, dipoles, and harmonic functions on {B}ratteli diagrams.
\newblock {\em Acta Appl. Math.}, 159:169--224, 2019.

\bibitem[BJKR01]{BratteliJorgensenKimRoush2001}
Ola Bratteli, Palle E.~T. Jorgensen, Ki~Hang Kim, and Fred Roush.
\newblock Decidability of the isomorphism problem for stationary {AF}-algebras
  and the associated ordered simple dimension groups.
\newblock {\em Ergodic Theory Dynam. Systems}, 21(6):1625--1655, 2001.

\bibitem[BJKR02]{BratteliJorgensenKimRoush2002}
Ola Bratteli, Palle E.~T. Jorgensen, Ki~Hang Kim, and Fred Roush.
\newblock Computation of isomorphism invariants for stationary dimension
  groups.
\newblock {\em Ergodic Theory Dynam. Systems}, 22(1):99--127, 2002.

\bibitem[BK16]{BezuglyiKarpel2016}
S.~Bezuglyi and O.~Karpel.
\newblock Bratteli diagrams: structure, measures, dynamics.
\newblock In {\em Dynamics and numbers}, volume 669 of {\em Contemp. Math.},
  pages 1--36. Amer. Math. Soc., Providence, RI, 2016.

\bibitem[BK20]{BezuglyiKarpel2020}
Sergey Bezuglyi and Olena Karpel.
\newblock Invariant measures for {C}antor dynamical systems.
\newblock In {\em Dynamics: topology and numbers}, volume 744 of {\em Contemp.
  Math.}, pages 259--295. Amer. Math. Soc., [Providence], RI, [2020] \copyright
  2020.

\bibitem[BKM09]{BezuglyiKwiatkowskiMedynets2009}
S.~Bezuglyi, J.~Kwiatkowski, and K.~Medynets.
\newblock Aperiodic substitution systems and their {B}ratteli diagrams.
\newblock {\em Ergodic Theory Dynam. Systems}, 29(1):37--72, 2009.

\bibitem[Bra72]{Bratteli1972}
O.~Bratteli.
\newblock Inductive limits of finite dimensional {$C^{\ast} $}-algebras.
\newblock {\em Trans. Amer. Math. Soc.}, 171:195--234, 1972.

\bibitem[BTA04]{BerlinetThomas-Agnan2004}
Alain Berlinet and Christine Thomas-Agnan.
\newblock {\em Reproducing kernel {H}ilbert spaces in probability and
  statistics}.
\newblock Kluwer Academic Publishers, Boston, MA, 2004.
\newblock With a preface by Persi Diaconis.

\bibitem[CD19]{CerfDalmau2019}
Rapha\"{e}l Cerf and Joseba Dalmau.
\newblock Galton-{W}atson and branching process representations of the
  normalized {P}erron-{F}robenius eigenvector.
\newblock {\em ESAIM Probab. Stat.}, 23:797--802, 2019.

\bibitem[CFS82]{CornfeldFominSinai1982}
I.~P. Cornfeld, S.~V. Fomin, and Ya.~G. Sina\u\i.
\newblock {\em Ergodic theory}, volume 245 of {\em Grundlehren der
  Mathematischen Wissenschaften [Fundamental Principles of Mathematical
  Sciences]}.
\newblock Springer-Verlag, New York, 1982.
\newblock Translated from the Russian by A. B. Sosinski\u\i.

\bibitem[CG12]{ChungGraham2012}
Fan Chung and Ron Graham.
\newblock Edge flipping in graphs.
\newblock {\em Adv. in Appl. Math.}, 48(1):37--63, 2012.

\bibitem[CHL{\etalchar{+}}20]{Chi-2020}
Nan Chi, Fangchen Hu, Guoqiang Li, Chaofan Wang, and Wenqing Niu.
\newblock A{I} based on frequency slicing deep neural network for underwater
  visible light communication.
\newblock {\em Sci. China Inf. Sci.}, 63(6):160303, 2020.

\bibitem[Cho14]{Cho2014}
Ilwoo Cho.
\newblock {\em Algebras, graphs and their applications}.
\newblock CRC Press, Boca Raton, FL, 2014.
\newblock Edited by Palle E. T. Jorgensen.

\bibitem[Chu07]{Chung2007}
Fan Chung.
\newblock Random walks and local cuts in graphs.
\newblock {\em Linear Algebra Appl.}, 423(1):22--32, 2007.

\bibitem[Chu10]{Chung2010}
Fan Chung.
\newblock Graph theory in the information age.
\newblock {\em Notices Amer. Math. Soc.}, 57(6):726--732, 2010.

\bibitem[Chu14]{Chung2014}
Fan Chung.
\newblock From quasirandom graphs to graph limits and graphlets.
\newblock {\em Adv. in Appl. Math.}, 56:135--174, 2014.

\bibitem[CK77]{ConnesKrieger1977}
Alain Connes and Wolfgang Krieger.
\newblock Measure space automorphisms, the normalizers of their full groups,
  and approximate finiteness.
\newblock {\em J. Functional Analysis}, 24(4):336--352, 1977.

\bibitem[CK14]{ChungKenter2014}
Fan Chung and Franklin Kenter.
\newblock Discrepancy inequalities for directed graphs.
\newblock {\em Discrete Appl. Math.}, 176:30--42, 2014.

\bibitem[CN19]{ChaystiNoutsos2019}
Thaniporn Chaysri and Dimitrios Noutsos.
\newblock On the {P}erron-{F}robenius theory of {$M_v$}-matrices and equivalent
  properties to eventually exponentially nonnegative matrices.
\newblock {\em Electron. J. Linear Algebra}, 35:424--440, 2019.

\bibitem[CP16]{CarollPetersen_2016}
Kathleen Carroll and Karl Petersen.
\newblock Markov diagrams for some non-{M}arkovian systems.
\newblock In {\em Ergodic theory, dynamical systems, and the continuing
  influence of {J}ohn {C}. {O}xtoby}, volume 678 of {\em Contemp. Math.}, pages
  73--101. Amer. Math. Soc., Providence, RI, 2016.

\bibitem[CSS19]{CiolettiSilvaStadlbauer2019}
L.~Cioletti, E.~Silva, and M.~Stadlbauer.
\newblock Thermodynamic formalism for topological {M}arkov chains on standard
  {B}orel spaces.
\newblock {\em Discrete Contin. Dyn. Syst.}, 39(11):6277--6298, 2019.

\bibitem[CVLX19]{ChenVongLiXu2019}
Xiao~Shan Chen, Seak-Weng Vong, Wen Li, and Hongguo Xu.
\newblock Noda iterations for generalized eigenproblems following
  {P}erron-{F}robenius theory.
\newblock {\em Numer. Algorithms}, 80(3):937--955, 2019.

\bibitem[DH03]{DooleyHamachi2003}
A.~H. Dooley and Toshihiro Hamachi.
\newblock Nonsingular dynamical systems, {B}ratteli diagrams and {M}arkov
  odometers.
\newblock {\em Israel J. Math.}, 138:93--123, 2003.

\bibitem[DHS99]{DurandHostSkau1999}
F.~Durand, B.~Host, and C.~Skau.
\newblock Substitutional dynamical systems, {B}ratteli diagrams and dimension
  groups.
\newblock {\em Ergodic Theory Dynam. Systems}, 19(4):953--993, 1999.

\bibitem[DJ10]{Dutkay_Jorgensen2010}
Dorin~Ervin Dutkay and Palle E.~T. Jorgensen.
\newblock Spectral theory for discrete {L}aplacians.
\newblock {\em Complex Anal. Oper. Theory}, 4(1):1--38, 2010.

\bibitem[DJK94a]{DoughertyJacksonKechris_1994}
R.~Dougherty, S.~Jackson, and A.~S. Kechris.
\newblock The structure of hyperfinite {B}orel equivalence relations.
\newblock {\em Trans. Amer. Math. Soc.}, 341(1):193--225, 1994.

\bibitem[DJK94b]{DoughertyJacksonKechris1994}
R.~Dougherty, S.~Jackson, and A.~S. Kechris.
\newblock The structure of hyperfinite {B}orel equivalence relations.
\newblock {\em Trans. Amer. Math. Soc.}, 341(1):193--225, 1994.

\bibitem[DMPS18]{Douc_2018}
Randal Douc, Eric Moulines, Pierre Priouret, and Philippe Soulier.
\newblock {\em Markov chains}.
\newblock Springer Series in Operations Research and Financial Engineering.
  Springer, Cham, 2018.

\bibitem[DS88]{DanfordSchwartz1988}
Nelson Dunford and Jacob~T. Schwartz.
\newblock {\em Linear operators. {P}art {II}}.
\newblock Wiley Classics Library. John Wiley \& Sons, Inc., New York, 1988.
\newblock Spectral theory. Selfadjoint operators in Hilbert space, With the
  assistance of William G. Bade and Robert G. Bartle, Reprint of the 1963
  original, A Wiley-Interscience Publication.

\bibitem[DSST20]{Dunlop2020}
Matthew~M. Dunlop, Dejan Slep\v{c}ev, Andrew~M. Stuart, and Matthew Thorpe.
\newblock Large data and zero noise limits of graph-based semi-supervised
  learning algorithms.
\newblock {\em Appl. Comput. Harmon. Anal.}, 49(2):655--697, 2020.

\bibitem[DT19]{DeyTrivedi2019}
Santanu Dey and Harsh Trivedi.
\newblock Bures distance and transition probability for
  {$\alpha$}-{CPD}-kernels.
\newblock {\em Complex Anal. Oper. Theory}, 13(5):2171--2190, 2019.

\bibitem[Dur10]{Durand2010}
Fabien Durand.
\newblock Combinatorics on {B}ratteli diagrams and dynamical systems.
\newblock In {\em Combinatorics, automata and number theory}, volume 135 of
  {\em Encyclopedia Math. Appl.}, pages 324--372. Cambridge Univ. Press,
  Cambridge, 2010.

\bibitem[EP19]{EpsteinPop2019}
Charles~L. Epstein and Camelia~A. Pop.
\newblock Transition probabilities for degenerate diffusions arising in
  population genetics.
\newblock {\em Probab. Theory Related Fields}, 173(1-2):537--603, 2019.

\bibitem[Fer06]{Ferenczi2006}
S\'{e}bastien Ferenczi.
\newblock Substitution dynamical systems on infinite alphabets.
\newblock volume~56, pages 2315--2343. 2006.
\newblock Num\'{e}ration, pavages, substitutions.

\bibitem[FN18]{FalkNussbaum2018}
Richard~S. Falk and Roger~D. Nussbaum.
\newblock {$C^m$} eigenfunctions of {P}erron-{F}robenius operators and a new
  approach to numerical computation of {H}ausdorff dimension: applications in
  {$\Bbb R^1$}.
\newblock {\em J. Fractal Geom.}, 5(3):279--337, 2018.

\bibitem[For97]{Forrest1997}
A.~H. Forrest.
\newblock {$K$}-groups associated with substitution minimal systems.
\newblock {\em Israel J. Math.}, 98:101--139, 1997.

\bibitem[FP10]{FrickPetersen_2010}
Sarah~Bailey Frick and Karl Petersen.
\newblock Reinforced random walks and adic transformations.
\newblock {\em J. Theoret. Probab.}, 23(3):920--943, 2010.

\bibitem[Geo10]{Georgakopoulos2010}
Agelos Georgakopoulos.
\newblock Uniqueness of electrical currents in a network of finite total
  resistance.
\newblock {\em J. Lond. Math. Soc. (2)}, 82(1):256--272, 2010.

\bibitem[GPS95]{GiordanoPutnamSkau1995}
Thierry Giordano, Ian~F. Putnam, and Christian~F. Skau.
\newblock Topological orbit equivalence and {$C^*$}-crossed products.
\newblock {\em J. Reine Angew. Math.}, 469:51--111, 1995.

\bibitem[GR19]{GiladiRuffer2019}
Ohad Giladi and Bj\"{o}rn~S. R\"{u}ffer.
\newblock A {P}erron-{F}robenius type result for integer maps and applications.
\newblock {\em Positivity}, 23(3):545--570, 2019.

\bibitem[GS20]{GaoSu2020}
Wenyou Gao and Chang Su.
\newblock Analysis on block chain financial transaction under artificial neural
  network of deep learning.
\newblock {\em J. Comput. Appl. Math.}, 380:112991, 2020.

\bibitem[GTH19]{GautierTudiscoHein2019}
Antoine Gautier, Francesco Tudisco, and Matthias Hein.
\newblock The {P}erron-{F}robenius theorem for multihomogeneous mappings.
\newblock {\em SIAM J. Matrix Anal. Appl.}, 40(3):1179--1205, 2019.

\bibitem[HN20]{HaninNica2020}
Boris Hanin and Mihai Nica.
\newblock Products of many large random matrices and gradients in deep neural
  networks.
\newblock {\em Comm. Math. Phys.}, 376(1):287--322, 2020.

\bibitem[HPS92]{HermanPutnamSkau1992}
Richard~H. Herman, Ian~F. Putnam, and Christian~F. Skau.
\newblock Ordered {B}ratteli diagrams, dimension groups and topological
  dynamics.
\newblock {\em Internat. J. Math.}, 3(6):827--864, 1992.

\bibitem[JKL02]{JacksonKechrisLouveau_2002}
S.~Jackson, A.~S. Kechris, and A.~Louveau.
\newblock Countable {B}orel equivalence relations.
\newblock {\em J. Math. Log.}, 2(1):1--80, 2002.

\bibitem[Jor06]{Jorgensen2006}
Palle E.~T. Jorgensen.
\newblock {\em Analysis and probability: wavelets, signals, fractals}, volume
  234 of {\em Graduate Texts in Mathematics}.
\newblock Springer, New York, 2006.

\bibitem[Jor12]{Jorgensen2012}
Palle E.~T. Jorgensen.
\newblock Unbounded graph-{L}aplacians in energy space, and their extensions.
\newblock {\em J. Appl. Math. Comput.}, 39(1-2):155--187, 2012.

\bibitem[JP08]{JorgensenPearse-2020}
Palle E.~T. Jorgensen and Erin P.~J. Pearse.
\newblock Operator theory of electrical resistance networks, 2008.

\bibitem[JP10]{Jorgensen_Pearse2010}
Palle E.~T. Jorgensen and Erin Peter~James Pearse.
\newblock A {H}ilbert space approach to effective resistance metric.
\newblock {\em Complex Anal. Oper. Theory}, 4(4):975--1013, 2010.

\bibitem[JP11]{JorgensenPearse2011}
Palle E.~T. Jorgensen and Erin P.~J. Pearse.
\newblock Resistance boundaries of infinite networks.
\newblock In {\em Random walks, boundaries and spectra}, volume~64 of {\em
  Progr. Probab.}, pages 111--142. Birkh\"{a}user/Springer Basel AG, Basel,
  2011.

\bibitem[JP13]{Jorgensen_Pearse2013}
Palle E.~T. Jorgensen and Erin P.~J. Pearse.
\newblock A discrete {G}auss-{G}reen identity for unbounded {L}aplace
  operators, and the transience of random walks.
\newblock {\em Israel J. Math.}, 196(1):113--160, 2013.

\bibitem[JP16]{JorgensenPearse_2016}
Palle E.~T. Jorgensen and Erin P.~J. Pearse.
\newblock Symmetric pairs and self-adjoint extensions of operators, with
  applications to energy networks.
\newblock {\em Complex Anal. Oper. Theory}, 10(7):1535--1550, 2016.

\bibitem[JP17]{JorgensenPearse_2017}
Palle E.~T. Jorgensen and Erin P.~J. Pearse.
\newblock Symmetric pairs of unbounded operators in {H}ilbert space, and their
  applications in mathematical physics.
\newblock {\em Math. Phys. Anal. Geom.}, 20(2):Art. 14, 24, 2017.

\bibitem[JP19]{JorgensenPearse_2019}
Palle E.~T. Jorgensen and Erin P.~J. Pearse.
\newblock Continuum versus discrete networks, graph {L}aplacians, and
  reproducing kernel {H}ilbert spaces.
\newblock {\em J. Math. Anal. Appl.}, 469(2):765--807, 2019.

\bibitem[JPT18]{JorgensenPearseTian_2018}
Palle Jorgensen, Erin Pearse, and Feng Tian.
\newblock Unbounded operators in {H}ilbert space, duality rules, characteristic
  projections, and their applications.
\newblock {\em Anal. Math. Phys.}, 8(3):351--382, 2018.

\bibitem[JT15]{JorgensenTian2015}
Palle Jorgensen and Feng Tian.
\newblock Discrete reproducing kernel {H}ilbert spaces: sampling and
  distribution of {D}irac-masses.
\newblock {\em J. Mach. Learn. Res.}, 16:3079--3114, 2015.

\bibitem[JT16]{JorgensenTian-2016}
Palle Jorgensen and Feng Tian.
\newblock Nonuniform sampling, reproducing kernels, and the associated
  {H}ilbert spaces.
\newblock {\em Sampl. Theory Signal Image Process.}, 15:37--72, 2016.

\bibitem[JT19a]{JT-2019}
Palle Jorgensen and Feng Tian.
\newblock On reproducing kernels, and analysis of measures.
\newblock {\em Markov Process. Related Fields}, 25(3):445--482, 2019.

\bibitem[JT19b]{JT_2019}
Palle Jorgensen and Feng Tian.
\newblock Realizations and factorizations of positive definite kernels.
\newblock {\em J. Theoret. Probab.}, 32(4):1925--1942, 2019.

\bibitem[Kat95]{Kato1995}
Tosio Kato.
\newblock {\em Perturbation theory for linear operators}.
\newblock Classics in Mathematics. Springer-Verlag, Berlin, 1995.
\newblock Reprint of the 1980 edition.

\bibitem[Kec95]{Kechris1995}
Alexander~S. Kechris.
\newblock {\em Classical descriptive set theory}, volume 156 of {\em Graduate
  Texts in Mathematics}.
\newblock Springer-Verlag, New York, 1995.

\bibitem[Kig01]{Kigami2001}
Jun Kigami.
\newblock {\em Analysis on fractals}, volume 143 of {\em Cambridge Tracts in
  Mathematics}.
\newblock Cambridge University Press, Cambridge, 2001.

\bibitem[Kit98]{Kitchens1998}
Bruce~P. Kitchens.
\newblock {\em Symbolic dynamics}.
\newblock Universitext. Springer-Verlag, Berlin, 1998.
\newblock One-sided, two-sided and countable state Markov shifts.

\bibitem[KKLS19]{KunszentiLovasz2019}
D\'{a}vid Kunszenti-Kov\'{a}cs, L\'{a}szl\'{o} Lov\'{a}sz, and Bal\'{a}zs
  Szegedy.
\newblock Measures on the square as sparse graph limits.
\newblock {\em J. Combin. Theory Ser. B}, 138:1--40, 2019.

\bibitem[KL16]{KerrLi_2016}
David Kerr and Hanfeng Li.
\newblock {\em Ergodic theory}.
\newblock Springer Monographs in Mathematics. Springer, Cham, 2016.
\newblock Independence and dichotomies.

\bibitem[Kle14]{Klenke_2014}
Achim Klenke.
\newblock {\em Probability theory}.
\newblock Universitext. Springer, London, second edition, 2014.
\newblock A comprehensive course.

\bibitem[KS19]{KovachkiStuart2019}
Nikola~B. Kovachki and Andrew~M. Stuart.
\newblock Ensemble {K}alman inversion: a derivative-free technique for machine
  learning tasks.
\newblock {\em Inverse Problems}, 35(9):095005, 35, 2019.

\bibitem[Lov12]{Lovasz2012}
L\'{a}szl\'{o} Lov\'{a}sz.
\newblock {\em Large networks and graph limits}, volume~60 of {\em American
  Mathematical Society Colloquium Publications}.
\newblock American Mathematical Society, Providence, RI, 2012.

\bibitem[LP16]{LyonsPeres_2016}
Russell Lyons and Yuval Peres.
\newblock {\em Probability on trees and networks}, volume~42 of {\em Cambridge
  Series in Statistical and Probabilistic Mathematics}.
\newblock Cambridge University Press, New York, 2016.

\bibitem[LT80]{LazarTaylor1980}
A.~J. Lazar and D.~C. Taylor.
\newblock Approximately finite-dimensional {$C^{\ast} $}-algebras and
  {B}ratteli diagrams.
\newblock {\em Trans. Amer. Math. Soc.}, 259(2):599--619, 1980.

\bibitem[LZW{\etalchar{+}}16]{Li-2016}
Xi~Li, Liming Zhao, Lina Wei, Ming-Hsuan Yang, Fei Wu, Yueting Zhuang, Haibin
  Ling, and Jingdong Wang.
\newblock Deepsaliency: multi-task deep neural network model for salient object
  detection.
\newblock {\em IEEE Trans. Image Process.}, 25(8):3919--3930, 2016.

\bibitem[Med06]{Medynets2006}
Konstantin Medynets.
\newblock Cantor aperiodic systems and {B}ratteli diagrams.
\newblock {\em C. R. Math. Acad. Sci. Paris}, 342(1):43--46, 2006.

\bibitem[Nad90]{Nadkarni1990}
M.~G. Nadkarni.
\newblock On the existence of a finite invariant measure.
\newblock {\em Proc. Indian Acad. Sci. Math. Sci.}, 100(3):203--220, 1990.

\bibitem[Nad95]{Nadkarni1995}
M.~G. Nadkarni.
\newblock {\em Basic ergodic theory}, volume~6 of {\em Texts and Readings in
  Mathematics}.
\newblock Hindustan Book Agency, Delhi; distributed outside Asia by Henry
  Helson, Berkeley, CA, 1995.

\bibitem[Num84]{Nummelin_1984}
Esa Nummelin.
\newblock {\em General irreducible {M}arkov chains and nonnegative operators},
  volume~83 of {\em Cambridge Tracts in Mathematics}.
\newblock Cambridge University Press, Cambridge, 1984.

\bibitem[Pet12]{Petit2012}
Camille Petit.
\newblock Harmonic functions on hyperbolic graphs.
\newblock {\em Proc. Amer. Math. Soc.}, 140(1):235--248, 2012.

\bibitem[Phi05]{Phillips2005}
N.~Christopher Phillips.
\newblock Crossed products of the {C}antor set by free minimal actions of
  {$\Bbb Z^d$}.
\newblock {\em Comm. Math. Phys.}, 256(1):1--42, 2005.

\bibitem[PR16]{PaulsenRaghupathi2016}
Vern~I. Paulsen and Mrinal Raghupathi.
\newblock {\em An introduction to the theory of reproducing kernel {H}ilbert
  spaces}, volume 152 of {\em Cambridge Studies in Advanced Mathematics}.
\newblock Cambridge University Press, Cambridge, 2016.

\bibitem[Put18]{Putnam2018}
Ian~F. Putnam.
\newblock {\em Cantor minimal systems}, volume~70 of {\em University Lecture
  Series}.
\newblock American Mathematical Society, Providence, RI, 2018.

\bibitem[Ren18]{Renault2018}
Jean Renault.
\newblock Random walks on {B}ratteli diagrams.
\newblock In {\em Operator theory: themes and variations}, volume~20 of {\em
  Theta Ser. Adv. Math.}, pages 187--204. Theta, Bucharest, 2018.

\bibitem[Rev84]{Revuz_1984}
D.~Revuz.
\newblock {\em Markov chains}, volume~11 of {\em North-Holland Mathematical
  Library}.
\newblock North-Holland Publishing Co., Amsterdam, second edition, 1984.

\bibitem[Roh49]{Rohlin1949}
V.~A. Rohlin.
\newblock On the fundamental ideas of measure theory.
\newblock {\em Mat. Sbornik N.S.}, 25(67):107--150, 1949.

\bibitem[Sch12]{Schmeudgen2012}
Konrad Schm\"{u}dgen.
\newblock {\em Unbounded self-adjoint operators on {H}ilbert space}, volume 265
  of {\em Graduate Texts in Mathematics}.
\newblock Springer, Dordrecht, 2012.

\bibitem[Sim12]{Simmons2012}
David Simmons.
\newblock Conditional measures and conditional expectation; {R}ohlin's
  disintegration theorem.
\newblock {\em Discrete Contin. Dyn. Syst.}, 32(7):2565--2582, 2012.

\bibitem[SL19]{SumLeung2019}
John Sum and Chi-Sing Leung.
\newblock Learning algorithm for {B}oltzmann machines with additive weight and
  bias noise.
\newblock {\em IEEE Trans. Neural Netw. Learn. Syst.}, 30(10):3200--3204, 2019.

\bibitem[Smi19]{Smirnov2019}
S.~N. Smirnov.
\newblock A {F}eller transition kernel with measure supports given by a
  set-valued mapping.
\newblock {\em Tr. Inst. Mat. Mekh.}, 25(1):219--228, 2019.

\bibitem[SS16]{SaitohSawano2016}
Saburou Saitoh and Yoshihiro Sawano.
\newblock {\em Theory of reproducing kernels and applications}, volume~44 of
  {\em Developments in Mathematics}.
\newblock Springer, Singapore, 2016.

\bibitem[SSQ{\etalchar{+}}20]{Sun-2020}
Guangling Sun, Yuying Su, Chuan Qin, Wenbo Xu, Xiaofeng Lu, and Andrzej
  Ceglowski.
\newblock Complete {D}efense {F}ramework to {P}rotect {D}eep {N}eural
  {N}etworks against {A}dversarial {E}xamples.
\newblock {\em Math. Probl. Eng.}, pages Art. ID 8319249, 17, 2020.

\bibitem[Tak03]{Takesaki2003}
M.~Takesaki.
\newblock {\em Theory of operator algebras. {III}}, volume 127 of {\em
  Encyclopaedia of Mathematical Sciences}.
\newblock Springer-Verlag, Berlin, 2003.
\newblock Operator Algebras and Non-commutative Geometry, 8.

\bibitem[VCH19]{Venkitaraman-2019}
Arun Venkitaraman, Saikat Chatterjee, and Peter H\"{a}ndel.
\newblock Predicting graph signals using kernel regression where the input
  signal is agnostic to a graph.
\newblock {\em IEEE Trans. Signal Inform. Process. Netw.}, 5(4):698--710, 2019.

\bibitem[Ver81]{Vershik1981}
A.~M. Vershik.
\newblock Uniform algebraic approximation of shift and multiplication
  operators.
\newblock {\em Dokl. Akad. Nauk SSSR}, 259(3):526--529, 1981.

\bibitem[Ver15]{Vershik2015}
A.~M. Vershik.
\newblock Equipped graded graphs, projective limits of simplices, and their
  boundaries.
\newblock {\em J. Math. Sci. (N.Y.)}, 209(6):860--873, 2015.

\bibitem[VO16]{VianaOliveira2016}
Marcelo Viana and Krerley Oliveira.
\newblock {\em Foundations of ergodic theory}, volume 151 of {\em Cambridge
  Studies in Advanced Mathematics}.
\newblock Cambridge University Press, Cambridge, 2016.

\bibitem[Xia20]{Xia2020}
Ying Xia.
\newblock Research on statistical machine translation model based on deep
  neural network.
\newblock {\em Computing}, 102(3):643--661, 2020.

\bibitem[XPL20]{Xiao2020}
Yatie Xiao, Chi-Man Pun, and Bo~Liu.
\newblock Adversarial example generation with adaptive gradient search for
  single and ensemble deep neural network.
\newblock {\em Inform. Sci.}, 528:147--167, 2020.

\bibitem[YD20]{YangDing2020}
Gang Yang and Fei Ding.
\newblock Associative memory optimized method on deep neural networks for image
  classification.
\newblock {\em Inform. Sci.}, 533:108--119, 2020.

\bibitem[YE19]{YankelevskyElad2019}
Yael Yankelevsky and Michael Elad.
\newblock Finding {GEMS}: multi-scale dictionaries for high-dimensional graph
  signals.
\newblock {\em IEEE Trans. Signal Process.}, 67(7):1889--1901, 2019.

\bibitem[Zer06]{Zerr2006}
Ryan~J. Zerr.
\newblock Minimal {B}ratteli diagrams and the dimension groups of {AF}
  {$C^*$}-algebras.
\newblock {\em Int. J. Math. Math. Sci.}, pages Art. ID 65737, 13, 2006.

\end{thebibliography}

\end{document}